\documentclass[11pt]{amsart}
\usepackage{float}
\usepackage{amsfonts}
\usepackage{amsmath}
\usepackage{amssymb}
\usepackage{graphicx,wrapfig}
\usepackage{adjustbox}
\usepackage[left=2.5cm, right=2.5cm, top=3cm]{geometry}
\usepackage{mathtools}
\usepackage{amsthm}
\usepackage{latexsym}
\usepackage{fancyhdr}
\usepackage{array}
\usepackage{amscd}
\usepackage{lscape}
\usepackage{tikz}
\usepackage{float}
\usepackage{bm}
\usepackage{subfigure}
\usepackage{grffile}
\usepackage{array}
\usepackage{caption}
\usepackage{longtable}
\usepackage{booktabs}
\usepackage{lipsum}
\usepackage[tableposition=below]{caption}
\newcolumntype{C}[1]{>{\centering\arraybackslash}p{#1}}
\usepackage[toc,page, title, titletoc]{appendix}
\definecolor{navy}{HTML}{2F729C}
\definecolor{red1}{HTML}{FF0000} % this is red
\usepackage[hyperfootnotes=false, colorlinks, linkcolor={blue}, citecolor={magenta}, filecolor={blue}, urlcolor={blue}, plainpages=false, pdfpagelabels]{hyperref}
\newcommand{\Q}{{\mathbb Q}}

\newtheorem{theorem}{Theorem}[section]
\newtheorem{lemma}[theorem]{Lemma}
\newtheorem{proposition}[theorem]{Proposition}
\newtheorem{corollary}[theorem]{Corollary}
\theoremstyle{definition}

\newtheorem{example}[theorem]{Example}

\theoremstyle{remark}

\numberwithin{equation}{section}
\begin{document}

\title[Local Data of Rational Elliptic Curves]{Local Data of Rational Elliptic Curves \\with non-Trivial Torsion}

\author{Alexander J. Barrios}
\address{Department of Mathematics and Statistics, Carleton College, Northfield, Minnesota 55057}
\email{abarrios@carleton.edu}
\author{Manami Roy}
\address{Department of Mathematics, Fordham University, Bronx, New York 10458}
\email{manami.roy.90@gmail.com}

\begin{abstract}
By Mazur's Torsion Theorem, there are fourteen possibilities for the non-trivial torsion subgroup $T$ of a rational elliptic curve. For each $T$, such that $E$ may have additive reduction at a prime $p$, we consider a parameterized family $E_T$ of elliptic curves with the property that they parameterize all elliptic curves $E/\mathbb{Q}$ which contain $T$ in their torsion subgroup. Using these parameterized families, we explicitly classify the Kodaira-N\'{e}ron type, the conductor exponent and the local Tamagawa number at each prime $p$ where $E/\mathbb{Q}$ has additive reduction. As a consequence, we find all rational elliptic curves with a $2$-torsion or a $3$-torsion point that have global Tamagawa number~$1$.

\end{abstract}
\maketitle
\tableofcontents

\section{Introduction}
Let $E$ be an elliptic curve over a local field $K$ with ring of integers
$R_{K}$. Then there is a unique minimal proper regular model $\mathcal{C}$ of
$E$ over $R_{K}$ \cite{Neron1964}. Tate's Algorithm \cite{Tate1975}
provides instructions on how to attain $\mathcal{C}$ from the Weierstrass
model of $E$. As a consequence of applying Tate's Algorithm, one attains the
following local data for $E/K$: the Kodaira-N\'{e}ron type of the special fiber of
$\mathcal{C}$, the exponent of the conductor of $E/K$, and the local Tamagawa
number. In this paper, we consider the case when $E/\mathbb{Q}$ is an elliptic curve with a non-trivial torsion subgroup and explicitly
classify the local data of $E/\mathbb{Q}_{p}$ at each prime $p$ for which $E/\mathbb{Q}$ has additive reduction. 

Let $T$ be one of the fourteen possible non-trivial torsion subgroups for a rational elliptic curve with a non-trivial torsion point \cite{Mazur1977}. Then the modular curve parameterizing isomorphism classes of pairs $(E,T)$ where $E$ is an elliptic curve with $T\hookrightarrow E$ has genus $0$; these parameterizations can be found in \cite[Table~3]{Kubert1976}. In this article, we consider a modification of these parameterized families, namely the families of elliptic curves $E_T$ as given in \cite{Barrios2020}. The models for the families $E_T$'s are found in Table~\ref{ta:ETmodel} in Section~\ref{Preliminaries}, and we note that the family $E_{C_2 \times C_8}$ is not listed in the table. We do not consider the family $E_{C_2 \times C_8}$ in this article since each member of the family $E_{C_2 \times C_8}$ is semistable by \cite[Theorem~7.1]{Barrios2020}.

The families $E_T$ have the property that if $E/\Q$ has non-trivial torsion subgroup $E(\Q)_{\text{tors}}$, then $E$ is $\Q$-isomorphic to a member of each family $E_T$ such that $ T \hookrightarrow E(\Q)_{\text{tors}}$, except for $T=C_2,C_3,C_3^0$. To be precise, if $C_2\times C_2 \hookrightarrow E(\Q)_{\text{tors}}$, then $E$ is not a member of $E_{C_2}$. If $C_3 \hookrightarrow E(\Q)_{\text{tors}}$ and $E$ has nonzero $j$-invariant, then $E$ is $\Q$-isomorphic to a member of $E_{C_3}$. If instead, $C_3 \hookrightarrow E(\Q)_{\text{tors}}$ and $E$ has $j$-invariant equal to $0$, then $E$ is $\Q$-isomorphic to either $E_{C_3}\!\left(  24,1\right)  $ or to the curve $E_{C_{3}^{0}}\!\left(  a\right)$ for some positive cubefree integer $a$. 
In particular, the study of rational elliptic curves with a non-trivial torsion point is equivalent to understanding the families $E_T$ for $T=C_2,C_2\times C_2, C_3, C_3^0, C_5$, and $C_7$. We also consider $E_T$ for the remaining $T$'s as they yield finer results on the local data of rational elliptic curves with a non-trivial torsion point. 
For instance, a consequence of our work is that if a member of $E_{C_2\times C_4}$ has additive reduction at $2$, then it must have Kodaira-N\'eron type $\mathrm{I}_n^*$ for some positive integer $n$. But there are more possible Kodaira-N\'eron types at $2$ for members of the families $E_{C_4}$ and $E_{C_2\times C_2}$. For example, the elliptic curve $E:y^2=x^3+x^2-20x$  (LMFDB label  \href{https://www.lmfdb.org/EllipticCurve/Q/120/b/4}{120.b4} \cite{lmfdb}) has additive reduction at $2$ with torsion subgroup $C_2\times C_4$. We observe that $E$ is $\Q$-isomorphic to $E_{C_2\times C_4}(2,1),$  $E_{C_4}(2,3)$, and $E_{C_2\times C_2}(4,9,1)$. 
Since it occurs in the family $E_{C_2\times C_4}$, then it must have Kodaira-N\'eron type $\mathrm{I}_m^*$ for some positive integer $m$, rather than the other possible types that occur in the families $E_{C_4}$ and  $E_{C_2\times C_2}$. In fact, our work shows that $m=1$.

In summary, the parameterized families $E_{T}$ have the property that they parameterize all elliptic curves $E/\Q$ with a non-trivial torsion point (see Proposition~\ref{rationalmodels}). 
We note that in the case when $\left\vert T\right\vert>4$, the parameterized family
$E_{T}$ is the universal elliptic curve over the genus $0$ modular curve
$X_{1}\!\left(  N\right)  $ or $X_{1}\!\left(  2,N\right)  $.

 In
\cite{Barrios2020}, the first author found necessary and sufficient
conditions on the parameters of $E_{T}$ that determine the primes $p$ at which
$E_{T}$ has additive reduction. 
In this article, we classify the local data of $E_{T}$ at those primes $p$ via Tate's Algorithm \cite{Tate1975}. Specifically,
   \textit{for every parametrized family $E_T$, we calculate the Kodaira-N\'{e}ron type at $p$, conductor exponent $f_p$, and local Tamagawa number $c_p$ at every
prime $p$ for which $E_T$ has additive reduction}. 

 These results are a consequence of Theorems~\ref{ThmC2at2}, \ref{ThmforC2podd}, \ref{ThmforC30}, \ref{ThmforC3}, \ref{ThmforC4}, \ref{ThmC2xC2}, and \ref{ThmotherTs}.
The explicit conditions can be found in Tables \ref{TableforC2}, \ref{TableforC2podd}, \ref{Table for C30}, \ref{TableforC3}, \ref{TableforC4}, \ref{Table for C2xC2}, and \ref{OtherTs}.
We label the Kodaira-N\'{e}ron type as K-N type in these tables. Consequently, we attain our explicit classification of the local data of a rational elliptic curve with a non-trivial torsion point at primes for which it has additive reduction.

In Section~\ref{Preliminaries}, we review the necessary background needed for
the proof of our results. First, we review some details about Tate's
Algorithm and state a result which simplifies the use of Tate's Algorithm in the
setting of a parameterized family. We also discuss Papadopoulos's
\cite{Papadopoulos1993} reduction of Tate's Algorithm as this will be helpful in our work. We conclude Section~\ref{Preliminaries} by giving
an explicit description of the parameterized families $E_{T}$ and reviewing
the relevant results that will be used in this paper. Section~\ref{PfMainThm} is devoted to the proofs of our main results.

In Section~\ref{Section on Consequences of the local data} we discuss some interesting examples and consequences of our results. For instance, we attain the following result: if $E$ is a rational elliptic curve with a torsion point of order $2$ (resp. $3$), then $f_p \leq 2$ for each prime $p \neq 2$ (resp. $p \neq 3$). This and several other consequences of our
results are provided in Section~\ref{Section on Consequences of the local data} (see Corollaries~\ref{CorC2xC2Instat},~\ref{CorollaryonConductor},~\ref{CorIIstar}~and~\ref{NeronC2}).

In addition, our results allow us to construct some interesting elliptic curves with prescribed local data. The following example illustrates a consequence of Corollary~\ref{NeronC2} where the elliptic curve has Kodaira-N\'{e}ron type $\mathrm{I}_{n}^{\ast}$ at each prime $p$ where it has additive reduction.
	\begin{example}\label{introexample}
	The rational elliptic curve 
	\[
	E: y^{2}=x^{3}+13440x^{2}-54296487559248001716600x
	\]
	has Kodaira-N\'{e}ron type $\mathrm{I}_{n}^{\ast}$ at $p$
	where $\left(  p,n\right)  $ is one of the following pairs: $\left(
	2,12\right)  ,\ \left(  3,7\right)  ,\ \left(  5,11\right)  ,$ or $\left(
	7,8\right)  $. Moreover, these are the only primes for which $E$ has additive reduction.
\end{example}
We note that the conductor of rational elliptic curves with a non-trivial torsion point grows rapidly. In particular, if one is interested in finding elliptic curves with prescribed local data, then such elliptic curves may not be available in the existing databases. For instance, the elliptic curve in Example~\ref{introexample} has a $25$-digit conductor.

In Section~\ref{Global Tamagwa}, we look at the global Tamagawa number $c$ of elliptic curves with a non-trivial torsion point. This section is motivated by the work of Lorenzini \cite{Lorenzini2012}, where he looked at the quotient $\frac{c}{|E(\mathbb{Q})_{\mathrm{tors}}|}$. This ratio appears as a factor in the leading term of the $L$-function
of $E/\mathbb{Q}$ in the Birch and Swinnerton-Dyer Conjecture \cite{BSD} (also see \cite[F.4.1.6]{HindrySilverman2000}). Lorenzini \cite[Proposition 1.1]{Lorenzini2012} showed that with finitely many exceptions, this
ratio is an integer for elliptic curves with a torsion point of order $N\geq5$. When $N=4$, he showed that the ratio $\frac{c}{|E(\mathbb{Q})_{\mathrm{tors}}|}=\frac{1}{2}$ occurs infinitely often if Schinzel's Hypothesis H is true. 
In addition, he showed that there are exactly four non-isomorphic elliptic curves (LMFDB labels \href{https://www.lmfdb.org/EllipticCurve/Q/11a3/}{11.a3}, \href{https://www.lmfdb.org/EllipticCurve/Q/15a7/}{ 15.a4}, \href{https://www.lmfdb.org/EllipticCurve/Q/15a8/}{ 15.a7}, \href{https://www.lmfdb.org/EllipticCurve/Q/17a4/}{17.a4})
with a torsion point of order $N\geq 4$ such that $c=1$. 
He also provided an infinite family of non-isomorphic elliptic curves with a $3$-torsion point and $c=1$. Thus, the ratio $\frac{c}{|E(\mathbb{Q})_{\mathrm{tors}}|}=\frac{1}{3}$ occurs infinitely often. As an application of our results in Section~\ref{Sect:3tors}, we find all rational elliptic curves with a $3$-torsion point and $c=1$. This is Theorem~\ref{GlobalTamaat3}, which is given below:
\begin{theorem}
\label{GlobalTamaat3}
Let $E$ be a rational elliptic curve with a $3$-torsion point. Then the global
Tamagawa number of $E$ is $c=1$ if and only if $E$ is $\mathbb{Q}$-isomorphic to either $y^{2}+y=x^{3}$ or $y^2+a^3xy\pm a^6y=x^3$ for some positive integer $a$ such that the following are satisfied for each prime $p$:%
\begin{equation}\label{C3tamacond}
v_{p}\!\left(  \pm a^{3}\mp27\right)  =\left\{
\begin{array}
[c]{cl}%
0, 1, \text{or } 4,& \\
3 & \text{with }p=3\text{ and }1\mp a\not \equiv -2\ \operatorname{mod}9,\\
5 & \text{with }p=3\text{ and }a^{6}\mp27a\not \equiv 3^{8}\ \operatorname{mod}3^{9},\\
6 & \text{with }p=3\text{ and }a^{6}\mp27a\equiv3^{9}\ \operatorname{mod}3^{10},\\
m & \text{where }m \text{ is odd with }  p\equiv5,11\ \operatorname{mod}12.
\end{array}
\right. 
\end{equation}
\end{theorem}
As a consequence of Theorem~\ref{GlobalTamaat3}, we provide a new infinite family of elliptic curves with a $3$-torsion point and global Tamagawa number $1$ in Corollary~\ref{CorC3c1}. We also find all rational elliptic curves with a $2$-torsion point and $c=1$ in Theorem~\ref{GlobalTamaat2}. As a consequence of this result, we show that there are infinitely many integers $k$ such that the rational elliptic curve $$y^{2}=x^{3}+\left(4k+2\right)x^{2}-x$$ has a
$2$-torsion point and $c=1$. Theorems~\ref{GlobalTamaat3} and \ref{GlobalTamaat2}, coupled with Lorenzini's results, give a classification of all rational elliptic curves with a non-trivial torsion point which have global Tamagawa number $c=1$.

\section{Preliminaries\label{Preliminaries}}
In this section, we recall some basic facts about elliptic curves; for details see \cite{Silverman2009} and \cite{Silverman1994}. Let $F$ be a field and let $E/F$ be an elliptic curve given by the (affine) \textit{Weierstrass model}
{
\begin{equation}
E:y^{2}+a_{1}xy+a_{3}y=x^{3}+a_{2}x^{2}+a_{4}x+a_{6}.\label{ch:inintroweier}
\end{equation}
From the Weierstrass coefficients $a_{j}$, we define the following quantities%
\begin{equation}
\begin{split}
\label{basicformulas}
b_2&=a_1^2+4a_2,\quad b_4=2a_4+a_1a_3,\quad
b_6=a_3^2+4a_6,\\ b_8&=a_1^2a_6+4a_2a_6-a_1a_3a_4+a_2a_3^2-a_4^2,\\
c_4&=b_2^2-24b_4,\quad c_6=36b_2b_4-216b_6-b_2^3,\quad
\Delta=9b_2b_4b_6-b_2^2b_8-8b_4^3-27b_6^2.
\end{split}
\end{equation}}

Here $\Delta$ is the \textit{discriminant} of $E$ and $c_{4}$, $c_{6}$
are the \textit{invariants associated to the Weierstrass model} of $E$. The \textit{$j$-invariant} of $E$ is given by $j=\frac{c_{4}^{3}}{\Delta}$. An
elliptic curve $E^{\prime}$ is $F$\textit{-isomorphic} to $E$ if $E^{\prime}$ arises from $E$ via an
\textit{admissible change of variables} $x\longmapsto u^{2}x+r$ and
$y\longmapsto u^{3}y+u^{2}sx+w$ for $u,r,s,w\in F$ and $u\neq0$.  

Each elliptic curve $E/\mathbb{Q}_{p}$ is $\mathbb{Q}_{p}$-isomorphic to an elliptic curve given by a \textit{minimal model} of the
form \eqref{ch:inintroweier} such that $a_{j}\in\mathbb{Z}_{p}$ and 
$v_{p}\!\left(  \Delta\right)\ge0 $ is minimal. Here $v_{p}$ is the $p$-adic valuation of $\mathbb{Q}_{p}$. We call $\Delta$ the \textit{minimal discriminant} of $E/\mathbb{Q}_{p}$. Let $c_{4}$ and $\Delta$ be the invariants associated to a minimal model of $E/\mathbb{Q}_{p}$. We say that $E/\mathbb{Q}_{p}$ has \textit{additive
reduction} if $v_{p}\!\left(  \Delta\right) >0$ and $v_{p}\!\left(
c_{4}\right)>0$, and $E/\mathbb{Q}_{p}$ is said to
have \textit{multiplicative reduction} if $v_{p}\!\left(  \Delta\right)>0$ and $v_{p}\!\left(  c_{4}\right)=0$. If $v_{p}\!\left(  \Delta\right)=0$, then $E/\mathbb{Q}_{p}$ is said to have \textit{good reduction}.

Similarly, $E/\mathbb{Q}$ is $\mathbb{Q}$-isomorphic to a model with the property that it is a minimal model over $\mathbb{Q}_{p}$ at each prime $p$. We call this model a \textit{global minimal model} for $E/\mathbb{Q}$ and its discriminant is called the \textit{minimal discriminant}. We say $E/\mathbb{Q}$ has additive, multiplicative, or good reduction at a prime $p$ if $E/\mathbb{Q}_{p}$ has the same reduction, respectively. We say $E/\mathbb{Q}$ is \textit{semistable at a prime} $p$ if it has good or multiplicative reduction at $p$. Moreover, $E$ is said to be \textit{semistable} if $E$ is semistable at all primes. The
\textit{conductor} $N_{E}$ of a rational elliptic curve $E$
is given by%
\[
N_{E}=\prod_{p|\Delta}p^{f_{p}}
\]
where $\Delta$ is the minimal discriminant and $f_{p}$ is a positive integer. As usual, we denote by $f_{p}$ \textit{the conductor exponent at} $p$. We note that \[
f_{p}=\left\{
\begin{array}
[c]{cl}%
0 & \text{if }E\text{ has good reduction at }p\text{,}\\
1 & \text{if }E\text{ has multiplicative reduction at }p\text{,}\\
2+\delta_{p} & \text{if }E\text{ has additive reduction at }p\text{,}%
\end{array}
\right.
\]
where $\delta_{p}=0$ if $p\geq5$, $\delta_{p}\leq 3$ if $p=3$, and $\delta_{p}\leq 6$ if $p=2$.
\vspace{0.5em}

\noindent \textbf{Local data of rational elliptic curves.} Given a rational elliptic curve $E$, Tate's Algorithm \cite{Tate1975} (see also \cite[Chapter IV]{Silverman1994}) computes the minimal proper regular
model $\mathcal{C}_{p}$ of $E$ over $\mathbb{Z}_{p}$ at a prime $p$. In particular, we attain the following local data for $E$ at a given prime~$p$:

\begin{enumerate}
\item The special fiber of $\mathcal{C}_{p}$, i.e., the Kodaira-N\'{e}ron type at $p$ of $E$. We use Kodaira symbols to describe the Kodaira-N\'{e}ron type at $p$.

\item The exponent $f_{p}$ appearing at the prime $p$ of the conductor of $E$.

\item The local Tamagawa number $c_{p}$, which is the number of components of
the special fiber of $\mathcal{C}_{p}$ defined over $\mathbb{F}_{p}$ with multiplicity $1$.
\end{enumerate}

Tate's Algorithm reduces the question of computing the special fiber of $\mathcal{C}_{p}$ to considering polynomial equations over $\mathbb{F}_{p}$. This, in turn, allows us to compute the Kodaira-N\'{e}ron type and the local Tamagawa number of $E/\mathbb{Q}_p$ by considering the splitting of certain polynomials modulo $p$. The possible Kodaira-N\'{e}ron types of $E/\Q_p$ are described in \cite[Theorem 8.2]{Silverman1994}. 
Specifically, if $E$ is semistable at $p$, then the Kodaira-N\'{e}ron type at $p$ is $\mathrm{I}_{n}$, where $n$ is a nonnegative integer. If $E$ has additive reduction at $p$, then the possible Kodaira-N\'{e}ron types at $p$ are $\mathrm{I}_{n}^{\ast}, \mathrm{II}, \mathrm{III}, \mathrm{IV}, \mathrm{II}^{\ast},\mathrm{III}^{\ast}$, and $\mathrm{IV}^{\ast}$, where $n$ is a nonnegative integer. The following lemma, which is
a direct consequence of Tate's Algorithm and is inspired by \cite{DokchitserDokchitser2013}, determines when an elliptic curve $E$ has Kodaira-N\'{e}ron type $\mathrm{I}_{n}^{\ast}$ at a given prime $p$ for $n\ge1$.

\begin{lemma}
\label{LemmaIn}Suppose $n$ is a positive integer and let $E/
\mathbb{Q}_{p}$ be an elliptic curve given by the Weierstrass model
\[
E:y^{2}+a_{1}xy+a_{3}y=x^{3}+a_{2}x^{2}+a_{4}x+a_{6}%
\]
with each $a_{j}\in
\mathbb{Z}_{p}$ such that
\begin{equation}
v_{p}\!\left(  a_{1}\right)  \geq1,\quad v_{p}\!\left(  a_{2}\right)  =1,\quad
v_{p}\!\left(  a_{3}\right)  \geq\frac{n+3}{2},\quad v_{p}\!\left(
a_{4}\right)  \geq\frac{n+4}{2},\quad v_{p}\!\left(  a_{6}\right)  \geq
n+3.\label{LemmaInval}%
\end{equation}
Let $\Delta$ denote the minimal discriminant of $E$ and set $a_{i,j}%
=\frac{a_{i}}{p^{j}}$. If one of the following equations
\begin{equation}%
\begin{cases}
t^{2}+a_{3,\frac{n+3}{2}}t-a_{6,n+3}\equiv0\ \operatorname{mod}p\qquad
\qquad\ \ \ \ \  & \text{if }n\text{ is odd},\\
a_{2,1}t^{2}+a_{4,\frac{n+4}{2}}t+a_{6,n+3}\equiv0\ \operatorname{mod}p\qquad
& \text{if }n\text{ is even,}
\end{cases}
\label{LemmaInpoly}
\end{equation}
has distinct roots in $\overline{\mathbb{F}}_{p}$, then $E$ has Kodaira-N\'{e}ron type
$\mathrm{I}_{n}^{\ast}$ at $p$ and $f_{p}=v_{p}\!\left(
\Delta\right)  -4-n$. In particular, $f_{p}=2$ for $p$ odd. Moreover, $c_{p}=4$ if the distinct roots are in $\mathbb{F}%
_{p}$ and $c_{p}=2$ otherwise.
\end{lemma}

\begin{proof}
By Tate's Algorithm \cite[Chapter IV]{Silverman1994}, we see that the assumptions
on the valuations of the Weierstrass coefficients (\ref{LemmaInval}) imply
that the Weierstrass model satisfies the first six steps of Tate's Algorithm.
Since $P\!\left(  t\right)  =t^{3}+a_{2,1}t^{2}+a_{4,2}t+a_{6,3}\equiv
t^{2}\left(  t+1\right)  \ \operatorname{mod}p$, Tate's Algorithm
proceeds to the subprocedure of Step $7$ which implies that the Kodaira-N\'{e}ron type
is I$_{m}^{\ast}$ for some $m\geq1$. Next, we show that $m=n$. Indeed, if $n=1$,
then Tate's Algorithm shows that $E$ has Kodaira-N\'{e}ron type $\mathrm{I}_{1}^{\ast}$ since
$t^{2}+a_{3,4}t-a_{6,4}\equiv0\ \operatorname{mod}p$ has distinct roots in
$\overline{\mathbb{F}}_{p}$. Now suppose $n>1$ and $l<n$. Then%
\[%
\begin{cases}
t^{2}+a_{3,\frac{l+3}{2}}t-a_{6,l+3}\equiv t^{2}\ \operatorname{mod}%
p\qquad\qquad\ \ \ \ \  & \text{if }l\text{ is odd},\\
a_{2,1}t^{2}+a_{4,\frac{l+4}{2}}t+a_{6,l+3}\equiv a_{2,1}t^{2}%
\ \operatorname{mod}p\qquad & \text{if }l\text{ is even.}%
\end{cases}
\]
Thus, the polynomial has a double root at $t=0$ and the subprocedure of Tate's
Algorithm continues. From this we conclude that $E$ has Kodaira-N\'{e}ron
type $\mathrm{I}_{n}^{\ast}$ with the claimed $c_{p}$ and~$f_{p}$.
\end{proof}

Papadopoulos \cite{Papadopoulos1993} showed how the Kodaira-N\'{e}ron type and conductor exponent~$f_{p}$ of $E/\mathbb{Q}_p$ can be determined from the triplet
$\left(  v_{p}\!\left(  c_{4}\right)  ,v_{p}\!\left(  c_{6}\right)
,v_{p}\!\left(  \Delta\right)  \right)  $. We list this result in Tables~\ref{ta:PapTableIV} and \ref{ta:PapTableIandII} in Appendix~\ref{AppendixTables}. Tables~\ref{ta:PapTableIV}~and~\ref{ta:PapTableIandII} are reproduced from \cite[Tables I, II, and IV]{Papadopoulos1993} and the possible local Tamagawa numbers are obtained by Tate's Algorithm. Observe that when the Kodaira-N\'{e}ron type is II, $\mathrm{III}$, $\mathrm{II}^{\ast}$, or $\mathrm{III}^{\ast}$, the local
Tamagawa number is uniquely determined. For the Kodaira-N\'{e}ron types $\mathrm{IV}$, IV$^{\ast
}$, or $\mathrm{I}_{n}^{\ast}$, Tate's Algorithm has to be used to compute the local Tamagawa number. The following lemma summarizes those results.
\begin{lemma}\label{LemmaPap}
Let $E/\mathbb{Q}_{p}$ be an elliptic curve given by a minimal model of the form \eqref{ch:inintroweier} with additive reduction. Then the triplet $\left(  v_{p}\!\left(  c_{4}\right)
,v_{p}\!\left(  c_{6}\right)  ,v_{p}\!\left(  \Delta\right)  \right)$ determines the possible Kodaira-N\'{e}ron types, conductor
exponents $f_{p}$, and local Tamagawa numbers $c_{p}$ as listed in Tables~\ref{ta:PapTableIV} and \ref{ta:PapTableIandII}. 

Moreover, the triplet does
not necessarily determine the Kodaira-N\'{e}ron type and conductor exponent of $E/\mathbb{Q}_{p}$ for $p=2,3$. For $p=3$, uniqueness is attained by the additional conditions given in Table~\ref{ta:PapTableIandII}. For $p=2$, uniqueness is attained by running Tate's Algorithm to at most Step $n$, where $n$ is as given in Table~\ref{ta:PapTableIV}.
\end{lemma}

\vspace{0.5em}

\noindent\textbf{Parameterization of rational elliptic curves with non-trivial
torsion.} Let $E$ be a rational elliptic curve. The following theorem lists
all the possible torsion subgroups of $E$.

\begin{theorem}
[Mazur's Torsion Theorem \cite{Mazur1977}]\label{MazurTorThm}Let $E$ be a
rational elliptic curve and let $C_{N}$ denote the cyclic group of order $N$.
Then%
\[
E\!\left(  \mathbb{Q}\right)  _{\text{tors}}\cong\left\{
\begin{array}
[c]{ll}%
C_{N} & \text{for }N=1,2,\ldots,10,12,\\
C_{2}\times C_{2N} & \text{for }N=1,2,3,4.
\end{array}
\right.
\]

\end{theorem}

In this article, we consider rational elliptic curves with a non-trivial torsion subgroup and additive reduction at some prime. Note that a rational elliptic curve having torsion subgroup $C_{2}\times C_{8}$ is semistable (see \cite[Theorem 7.1]{Barrios2020}). We also note that a rational elliptic curve $E$ is said to have \textit{full $2$-torsion} if  $C_{2}\times C_{2} \hookrightarrow E(\mathbb{Q})_{\text{tors}}$. Now let $E_T$ be the families of elliptic curves given in Table~\ref{ta:ETmodel}. These families of elliptic curves parameterize all rational elliptic curves with a non-trivial torsion point that have additive reduction at some prime, as made precise by the following proposition.

{\begingroup
\renewcommand{\arraystretch}{1.15} \small
 \begin{longtable}{C{0.5in}C{1.1in}C{1.8in}C{1.7in}c}
	\hline
	$T$ & $a_{1}$ & $a_{2}$ & $a_{3}$ & $a_{4}$ \\
	\hline

	\endfirsthead
	\hline
	$T$ & $a_{1}$ & $a_{2}$ & $a_{3}$ & $a_{4}$ \\
	\hline
	\endhead
	\hline

	\multicolumn{4}{r}{\emph{continued on next page}}
	\endfoot
	\hline
	\caption[Weierstrass Model for the family $E_{T}$]{Model for the family $E_{T}:y^{2}+a_{1}xy+a_{3}y=x^{3}+a_{2}x^{2}+a_{4}x$
}\label{ta:ETmodel}
	\endlastfoot
	
$C_{2}$ & $0$ & $2a$ & $0$ & $a^{2}-b^{2}d$ \\\hline
$C_{3}^{0}$ & $0$ & $0$ & $a$ & $0$ \\\hline
$C_{3}$ & $a$ & $0$ & $a^{2}b$ & $0$ \\\hline
$C_{4}$ & $a$ & $-ab$ & $-a^{2}b$ & $0$ \\\hline
$C_{5}$ & $a-b$ & $-ab$ & $-a^{2}b$ & $0$ \\\hline
$C_{6}$ & $a-b$ & $-ab-b^{2}$ & $-a^{2}b-ab^{2}$ & $0$\\\hline
$C_{7}$ & $a^{2}+ab-b^{2}$ & $a^{2}b^{2}-ab^{3}$ & $a^{4}b^{2}-a^{3}b^{3}$ & $0$ \\\hline
$C_{8}$ & $-a^{2}+4ab-2b^{2}$ & $-a^{2}b^{2}+3ab^{3}-2b^{4}$ & $-a^{3}b^{3}+3a^{2}
b^{4}-2ab^{5}$ & $0$ \\\hline
$C_{9}$ & $a^{3}+ab^{2}-b^{3}$ & $
a^{4}b^{2}-2a^{3}b^{3}+
2a^{2}b^{4}-ab^{5}
$ & $a^{3}\cdot a_{2}$
& $0$ \\\hline
$C_{10}$ &$
a^{3}-2a^{2}b-
2ab^{2}+2b^{3}
$ & $-a^{3}b^{3}+3a^{2}b^{4}-2ab^{5}$ & $(a^{3}-3a^{2}b+ab^{2})\cdot a_{2}$ & $0$\\\hline
$C_{12}$ & $
-a^{4}+2a^{3}b+2a^{2}b^{2}-
8ab^{3}+6b^{4}
$ & $b(a-2b)(a-b)^{2}(a^{2}-3ab+3b^{2})(a^{2}-2ab+2b^{2})
$ & $a(b-a)^3 \cdot a_{2} $& $0$ \\\hline
$C_{2}\times C_{2}$ & $0$ & $ad+bd$ & $0$ & $abd^{2}$ \\\hline
$C_{2}\times C_{4}$ & $a$ & $-ab-4b^{2}$ & $-a^{2}b-4ab^{2}$ & $0$ \\\hline
$C_{2}\times C_{6}$ & $-19a^{2}+2ab+b^{2}$ & $
-10a^{4}+22a^{3}b-
14a^{2}b^{2}+2ab^{3}
$ & $
90a^{6}-198a^{5}b+116a^{4}b^{2}+
4a^{3}b^{3}-14a^{2}b^{4}+2ab^{5}
$ & $0$ 	
\end{longtable}
\endgroup}

\begin{proposition}
[{\cite[Proposition 4.4]{Barrios2020}}]\label{rationalmodels}Let $E$ be a
rational elliptic curve and suppose further that $T\hookrightarrow E\!\left(\mathbb{Q}\right)_{\text{tors}}$ where $T$ is one of the fourteen non-trivial torsion
subgroups allowed by Theorem~\ref{MazurTorThm}. Then there are integers
$a,b,d$ such that

$\left(  1\right)  $ If $T\neq C_{2},C_{3},C_{2}\times C_{2}$, then $E$ is
$\mathbb{Q}$-isomorphic to $E_{T}\!\left(  a,b\right)  $ with $\gcd\!\left(
a,b\right)  =1$ and $a$ is positive.

$\left(  2\right)  $ If $T=C_{2}$ and $C_{2}\times C_{2}\not \hookrightarrow
E(\mathbb{Q})$, then $E$ is $\mathbb{Q}$-isomorphic to $E_{T}\!\left(
a,b,d\right)  $ with $d\neq1,b\neq0$ such that $d$ and $\gcd\!\left(
a,b\right)  $ are positive squarefree integers.

$\left(  3\right)  $ If $T=C_{3}$ and the $j$-invariant of $E$ is not $0$,
then $E$ is $\mathbb{Q}$-isomorphic to $E_{T}\!\left(  a,b\right)  $ with
$\gcd\!\left(  a,b\right)  =1$ and $a$ is positive.

$\left(  4\right)  $ If $T=C_{3}$ and the $j$-invariant of $E$ is $0$, then
$E$ is either $\mathbb{Q}$-isomorphic to $E_{T}\!\left(  24,1\right)  $ or to
the curve $E_{C_{3}^{0}}\!\left(  a\right)  :y^{2}+ay=x^{3}$ for some positive
cubefree integer $a$.

$\left(  5\right)  $ If $T=C_{2}\times C_{2}$, then $E$ is $\mathbb{Q}%
$-isomorphic to $E_{T}\!\left(  a,b,d\right)  $ with $\gcd\!\left(
a,b\right)  =1$, $d$ positive squarefree, and $a$ is even.
\end{proposition}
\vspace{-0.5em}

By \cite[Theorem 7.1]{Barrios2020}, there are necessary and sufficient conditions on the parameters of $E_{T}$ which determine the primes $p$ at which $E_{T}/\mathbb{Q}$ has additive reduction. 
This leads us to apply Lemma~\ref{LemmaPap} or Tate's Algorithm to a $\mathbb{Q}_{p}$-isomorphic minimal model of $E_{T}/\mathbb{Q}_{p}$. The following lemma provides us with these minimal models.

\begin{lemma}
\label{Lemma for minimal disc} Suppose $E_{T}/\mathbb{Q}$ has additive reduction at a
prime $p$.

(1) For $T\neq C_{2},\,C_{3},\,C_{4},C_{2}\times C_{4}$, the elliptic curve
$E_{T}/
\mathbb{Q}
_{p}$ is a minimal model.

(2) For $T=C_{2}$, the elliptic curve $E_{T}/\mathbb{Q}_{p}$ is a minimal model at each odd prime $p$. The elliptic curve $E_{T}/\mathbb{Q}_{2}$ is a minimal model if
either $v_{2}\!\left(  b^{2}d-a^{2}\right)  \leq3$,
$v_{2}\!\left(  b\right)  =1$, or $ v_{2}\!\left(
b\right)  \neq1$. If $v_{2}\!\left(  b\right)  =1,v_{2}\!\left(  b^{2}%
d-a^{2}\right)  \geq4$, then the elliptic
curve $E_{T}^{\prime}$ attained from $E_{T}$ via the admissible change of
variables $x\longmapsto4x$ and $y\longmapsto8y+8x$ is a minimal model for
$E_{T}$ over $\mathbb{Q}_{2}$.

(3) For $T=C_{3},\,C_{4}$, we define
\[
a=%
\begin{cases}
c^{3}d^{2}e\qquad & \text{if }T=C_{3},\\
c^{2}d\qquad & \text{if }T=C_{4},
\end{cases}
\qquad\text{and}\qquad u_{T}=
\begin{cases}
c^2d \quad & T=C_{3},\\
c\quad & T=C_{4}.%
\end{cases}
\]
Then the elliptic curve $E_{T}^{\prime}$ attained from $E_{T}$ via the
admissible change of variables $x\longmapsto u_{T}^{2}x$ and $y\longmapsto
u_{T}^{3}y$ is a minimal model over $\mathbb{Q}_p$.

$\left(  4\right)  $ For $T=C_{2}\times C_{4}$, the elliptic curve $E_{T}/\mathbb{Q}_{p}$ is a minimal model if $p$ is an odd prime or $p=2$ with
$v_{2}\!\left(  a\right)  \leq1$. If $p=2$ and $v_{2}\!\left(  a\right)  >1$,
then the elliptic curve $E_{T}^{\prime}$ attained from $E_{T}$ via the
admissible change of variables $x\longmapsto4x$ and $y\longmapsto8y$ is a
minimal model.
\end{lemma}

\begin{proof}
This follows from Theorems 6.1 and 7.1 of \cite{Barrios2020}.
\end{proof}

\section{Local data of \texorpdfstring{$E_T$}{} \label{PfMainThm}}
At every prime $p$ for which $E_T$ has additive reduction, we determine the Kodaira-Néron type at $p$, conductor exponent $f_p$, and local Tamagawa number $c_p$ with respect to explicit conditions on the parameters of $E_T$. 
In what follows, we consider the families $E_T$ for $T=C_{2},C_{3},C_{3}^0,C_{4},C_{2}\times C_{2}$, separately. Specifically, Theorems~\ref{ThmC2at2}~and~\ref{ThmforC2podd} are for $E_{C_{2}}$; Theorems~\ref{ThmforC30}, \ref{ThmforC3}, \ref{ThmforC4}, and~\ref{ThmC2xC2} are for $E_{C_{3}^0}$, $E_{C_{3}}$, $E_{C_{4}}$, and $E_{C_{2}\times C_{2}}$, respectively.
The remaining families $E_T$ are considered in Theorem~\ref{ThmotherTs}.

The proofs of these results rely on Tate's Algorithm and Lemma~\ref{LemmaPap}. In particular, we require knowledge
of a minimal model for $E_{T}/\mathbb{Q}_{p}$ at a prime $p$ for which $E_{T}$ has additive reduction. Using Lemma~\ref{Lemma for minimal disc} we find a minimal model of $E_{T}$ over $\mathbb{Q}_{p}$ and the associated triplet $\left(
v_{p}\!\left(  c_{4}\right)  ,v_{p}\!\left(  c_{6}\right)  ,v_{p}\!\left(
\Delta\right)  \right)$. Using this triplet, we deduce the possible Kodaira-N\'{e}ron types by Lemma~\ref{LemmaPap}. In the proofs we simply refer to Tables~\ref{ta:PapTableIV}~and~\ref{ta:PapTableIandII} instead of Lemma~\ref{LemmaPap}. If Tables~\ref{ta:PapTableIV}~and~\ref{ta:PapTableIandII} do not determine the Kodaira-N\'{e}ron type uniquely, then we
consider an admissible change of variables to $E_{T}$ to attain a
minimal model $F_{T}$ from which the desired local data can be attained via Tate's
Algorithm.
We will use Lemma~\ref{LemmaIn} to prove our conclusions in the case when the Kodaira-N\'{e}ron type is $\mathrm{I}_n^{\ast}$ for $n>0$.
 To ease the presentation of the proofs, we will at times state the $p$-adic valuations of $c_{4}$, $c_{6}$, $\Delta$, and Weierstrass coefficients $a_{j}$ of $F_{T}$ without justification, but note that these valuations have been verified via SageMath~\cite{sagemath} and they are available at \cite{LDWebpage}.

\subsection{Case of \texorpdfstring{$2$}{}-torsion with no full \texorpdfstring{$2$}{}-torsion.}

We start by considering the family $E_{C_2}=E_{C_2}(a,b,d)$, where  $a,b,d$ are integers with $d\neq1,b\neq0$ such that $\gcd\!\left(  a,b\right)$ and $d$ are squarefree. Then by Proposition~\ref{rationalmodels}, $E_{C_2}$ parameterizes all rational elliptic curves that have a $2$-torsion point, but do not have full $2$-torsion. The theorem below gives the necessary and sufficient conditions on the
parameters $a,b,d$ to determine the primes at which $E_{C_2}$ has additive
reduction. This result is equivalent to \cite[Theorem~7.1]{Barrios2020} for $E_{C_2}$.

\begin{theorem}\label{TisC2Addred2Thm7}The family of elliptic curves $E_{C_2}$ has additive reduction at an odd prime $p$ if and only
if $p$ divides $\gcd\!\left(  a,bd\right)  $. Moreover, $E_{C_2}$ has additive
reduction at $2$ if and only if one of the following conditions is satisfied:

$\left(  1\right)  $ $v_{2}\!\left(  b^{2}d-a^{2}\right)  \leq7$ with
$v_{2}\!\left(  a\right)  =v_{2}\!\left(  b\right)  =1$;

$\left(  2\right)  $ $v_{2}\!\left(  b\right)  \geq3$ with $a\not \equiv
3\ \operatorname{mod}4$;

$\left(  3\right)  $ $v_{2}\!\left(  b\right)  =0,2$;

$\left(  4\right)  $ $v_{2}\!\left(  b\right)  =1$ and $v_{2}\!\left(
a\right)  \neq1$;

$\left( 5\right)  $ $v_{2}\!\left(  b^{2}d-a^{2}\right)  \geq8$ with
$a\equiv6\ \operatorname{mod}8$.

\end{theorem}

Next, we compute the local data for $E_{C_2}$ at $p=2$ and $p\neq2$ separately. The following theorem deals with the cases when $E_{C_2}$ has additive reduction at $p=2$.

\begin{theorem}
\label{ThmC2at2}The family of elliptic curves $E_{C_{2}}$ has additive reduction at $2$ if and only if the parameters of $E_{C_2}$ satisfy one of the conditions on $a,b,d$ listed in Table~\ref{TableforC2}. The Kodaira-N\'{e}ron type at $2$,
conductor exponent $f_{2}$, and local Tamagawa number $c_{2}$ are given as follows.
\end{theorem}

{\begingroup
\small
\renewcommand{\arraystretch}{1.04}
 \begin{longtable}{C{0.4in}cC{0.05in}c}
	\hline
	K-N & Conditions on $a,b,d$ & $f_{2}$ & $c_{2}$\\ type& & & \\
	\hline

	\endfirsthead
	\hline
	K-N& Conditions on $a,b,d$ & $f_{2}$ & $c_{2}$\\ type& & & \\
	\hline
	\endhead
	\hline
	\multicolumn{4}{r}{\emph{continued on next page}}
	\endfoot
	\hline
	\caption{Local data for $E_{C_{2}}(a,b,d)$ at $p=2$}
	\endlastfoot
II & \multicolumn{1}{l}{$v_{2}\!\left(  b^{2}d-a^{2}\right)  =4,\ v_{2}\!\left(
a\right)  =v_{2}\!\left(  b\right)  =1,$} & $4$ & $1$\\
& \multicolumn{1}{l}{$b^{2}d-a^{2}-8a\equiv32\ \operatorname{mod}64$} &  & \\ \cmidrule(lr){2-4}
& \multicolumn{1}{l}{$v_{2}\!\left(  ab\right)  =v_{2}\!\left(  b^{2}%
d-a^{2}\right)  =0,\ v_{2}\!\left(  d\right)  =1$} & $7$ & $1$\\ \cmidrule(lr){2-4}
& \multicolumn{1}{l}{$v_{2}\!\left(  b\right)  =v_{2}\!\left(  b^{2}%
d-a^{2}\right)  =0,\ v_{2}\!\left(  a\right)  >0$,} & $6$ & $1$\\
& \multicolumn{1}{l}{ $d\equiv3\ \operatorname{mod}%
4$} & & \\ \cmidrule(lr){1-4}
$\mathrm{III}$ & \multicolumn{1}{l}{$v_{2}\!\left(  b\right)  =v_{2}\!\left(
b^{2}d-a^{2}\right)  =0,\ v_{2}\!\left(  a\right)  >0$,} & $5$ & $2$\\
& \multicolumn{1}{l}{$d\equiv1\ \operatorname{mod}4$} &  & \\ \cmidrule(lr){2-4}
& \multicolumn{1}{l}{$v_{2}\!\left(  b\right)  =0,\ v_{2}\!\left(
b^{2}d-a^{2}\right)  =1,\ v_{2}\!\left(  a\right)  >0$} & $8$ & $2$\\ \cmidrule(lr){2-4}
& \multicolumn{1}{l}{$v_{2}\!\left(  ab\right)  =0,\ v_{2}\!\left(
b^{2}d-a^{2}\right)  =1$} & $7$ & $2$\\ \cmidrule(lr){2-4}
& \multicolumn{1}{l}{$v_{2}\!\left(  b^{2}d-a^{2}\right)  =5,\ v_{2}\!\left(
a\right)  =v_{2}\!\left(  b\right)  =1$} & $5$ & $2$\\ \cmidrule(lr){2-4}
& \multicolumn{1}{l}{$b^{2}d-a^{2}\equiv48\ \operatorname{mod}64,\ a\equiv
6\ \operatorname{mod}8$} & $3$ & $2$\\ \cmidrule(lr){1-4}
IV & \multicolumn{1}{l}{$b^{2}d-a^{2}\equiv16\ \operatorname{mod}%
64,\ a\equiv2\ \operatorname{mod}8$} & $2$ & \multicolumn{1}{c}{$1$ if
$b^{2}d-a^{2}\not \equiv 8a\ \operatorname{mod}128$}\\
& \multicolumn{1}{l}{} &  & \multicolumn{1}{c}{$3$ if $b^{2}d-a^{2}%
\equiv8a\ \operatorname{mod}128$}\\ \cmidrule(lr){1-4}
$\mathrm{I}_{0}^{\ast}$ & \multicolumn{1}{l}{$v_{2}\!\left(  b\right)  =1,\ a\equiv
3\ \operatorname{mod}4,\ d\equiv1\ \operatorname{mod}4$} & $4$ & $2$\\ \cmidrule(lr){2-4}
& \multicolumn{1}{l}{$v_{2}\!\left(  b\right)  =1,\ a\equiv
3\ \operatorname{mod}4,\ d\equiv2\ \operatorname{mod}4$} & $5$ & $2$\\ \cmidrule(lr){2-4}
& \multicolumn{1}{l}{$v_{2}\!\left(  b\right)  =1,\ a\equiv
1\ \operatorname{mod}4,\ d\equiv2\ \operatorname{mod}4$} & $5$ & $1$\\ \cmidrule(lr){2-4}
& \multicolumn{1}{l}{$v_{2}\!\left(  b\right)  =1,\ a\equiv
1\ \operatorname{mod}4,\ d\equiv3\ \operatorname{mod}4$} & $4$ & $1$\\ \cmidrule(lr){2-4}
& \multicolumn{1}{l}{$v_{2}\!\left(  b\right)  =0,\ v_{2}\!\left(  b^{2}%
d-a^{2}\right)  =2$} & $6$ & $2$\\ \cmidrule(lr){2-4}
& \multicolumn{1}{l}{$v_{2}\!\left(  b^{2}d-a^{2}\right)  =6,\ a\equiv
6\ \operatorname{mod}8$} & $4$ & $2$\\ \cmidrule(lr){1-4}
$\mathrm{I}_{1}^{\ast}$ & \multicolumn{1}{l}{$v_{2}\!\left(  b^{2}d-a^{2}\right)
=6,\ a\equiv2\ \operatorname{mod}8$} & $3$ & \multicolumn{1}{c}{$2$ if$\text{
}b^{2}d-a^{2}+16a\not \equiv 96\ \operatorname{mod}256$}\\
& \multicolumn{1}{l}{} &  & \multicolumn{1}{c}{$4$ if $b^{2}d-a^{2}%
+16a\equiv96\ \operatorname{mod}256$}\\ \cmidrule(lr){2-4}
& \multicolumn{1}{l}{$v_{2}\!\left(  b\right)  =1,\ a,d\equiv
1\ \operatorname{mod}4$} & $3$ & \multicolumn{1}{c}{$2$ if$\text{ }%
ad\equiv5\ \operatorname{mod}8$}\\
& \multicolumn{1}{l}{} &  & \multicolumn{1}{c}{$4$ if$\text{ }ad\equiv
1\ \operatorname{mod}8$}\\ \cmidrule(lr){1-4}
$\mathrm{I}_{2}^{\ast}$ & \multicolumn{1}{l}{$v_{2}\!\left(  b^{2}d-a^{2}\right)
=2,\ v_{2}\!\left(  a\right)  =v_{2}\!\left(  b\right)  =1$} & $7$ &
\multicolumn{1}{c}{$2$ if $a-d\equiv4\ \operatorname{mod}8$}\\
& \multicolumn{1}{l}{} &  & \multicolumn{1}{c}{$4$ if $a-d\equiv
0\ \operatorname{mod}8$}\\ \cmidrule(lr){2-4}
& \multicolumn{1}{l}{$v_{2}\!\left(  b^{2}d-a^{2}\right)  =7,\ a\equiv6\ \operatorname{mod}8$} & $4$ & $4$\\ \cmidrule(lr){2-4}
& \multicolumn{1}{l}{$v_{2}\!\left(  b\right)  =0,\ v_{2}\!\left(
b^{2}d-a^{2}\right)  =3$} & $6$ & $4$\\ \cmidrule(lr){2-4}
& \multicolumn{1}{l}{$v_{2}\!\left(  b\right)  =1,\ v_{2}\!\left(  a\right)
>1,\ d\equiv1\ \operatorname{mod}4$} & $6$ & \multicolumn{1}{c}{$2$ if
$d\equiv5\ \operatorname{mod}8$}\\
& \multicolumn{1}{l}{} &  & \multicolumn{1}{c}{$4$ if $d\equiv
1\ \operatorname{mod}8$}\\ \cmidrule(lr){2-4}
& \multicolumn{1}{l}{$v_{2}\!\left(  b\right)  =2,\ v_{2}\!\left(  d\right)
=0,\ a\equiv1\ \operatorname{mod}4$} & $4$ & \multicolumn{1}{c}{$2$ if
$d\equiv3\ \operatorname{mod}4$}\\
& \multicolumn{1}{l}{} &  & \multicolumn{1}{c}{$4$ if $d\equiv
1\ \operatorname{mod}4$}\\ \cmidrule(lr){1-4}
$\mathrm{I}_{3}^{\ast}$ & \multicolumn{1}{l}{$v_{2}\!\left(  b\right)  =1,\ v_{2}%
\!\left(  a\right)  >1,\ d\equiv3\ \operatorname{mod}8$} & $5$ & \multicolumn{1}{l}{$2$ if
$a\equiv0,4\ \operatorname{mod}16,\hspace{.8em} b^{2}d\equiv44\ \operatorname{mod}64$}\\
& \multicolumn{1}{l}{} &  & \multicolumn{1}{l}{\hspace{0.3em} or $a\equiv
8,12\ \operatorname{mod}16,\ b^{2}d\equiv12\ \operatorname{mod}64$}\\
& \multicolumn{1}{l}{} &  & \multicolumn{1}{l}{$4$ if $a\equiv
0,4\ \operatorname{mod}16,\hspace{.8em} b^{2}d\equiv12\ \operatorname{mod}64$}\\
& \multicolumn{1}{l}{} &  & \multicolumn{1}{l}{\hspace{0.3em} or $a\equiv
8,12\ \operatorname{mod}16,\ b^{2}d\equiv44\ \operatorname{mod}64$}\\ \cmidrule(lr){2-4}
& \multicolumn{1}{l}{$v_{2}\!\left(  b\right)  =1,\ v_{2}\!\left(  a\right)
>1,\ d\equiv7\ \operatorname{mod}8$} & $5$ & \multicolumn{1}{l}{ $2$ if
$a\equiv0,12\ \operatorname{mod}16,\ b^{2}d\equiv28\ \operatorname{mod}64$}\\
& \multicolumn{1}{l}{} &  & \multicolumn{1}{l}{\hspace{0.3em} or $a\equiv
4,8\ \operatorname{mod}16,\hspace{.8em} b^{2}d\equiv60\ \operatorname{mod}64$}\\
& \multicolumn{1}{l}{} &  & \multicolumn{1}{l}{$4$ if $a\equiv
0,12\ \operatorname{mod}16,\ b^{2}d\equiv60\ \operatorname{mod}64$}\\
& \multicolumn{1}{l}{} &  & \multicolumn{1}{l}{\hspace{0.3em} or $ a\equiv
4,8\ \operatorname{mod}16,\hspace{.8em} b^{2}d\equiv28\ \operatorname{mod}64$}\\ \cmidrule(lr){2-4}
& \multicolumn{1}{l}{$v_{2}\!\left(  b\right)  =2,\ v_{2}\!\left(  d\right)
=1,\ a\equiv1\ \operatorname{mod}4$} & $4$ & \multicolumn{1}{l}{$2$ if
$d\equiv2\ \operatorname{mod}8$ and $a\equiv1\ \operatorname{mod}8$}\\
& \multicolumn{1}{l}{} &  & \multicolumn{1}{l}{\hspace{0.3em} or $ d\equiv
6\ \operatorname{mod}8\ \text{and\ }a\equiv5\ \operatorname{mod}8$}\\
& \multicolumn{1}{l}{} &  & \multicolumn{1}{l}{$4$ if $d\equiv
2\ \operatorname{mod}8$ and $a\equiv5\ \operatorname{mod}8$}\\
& \multicolumn{1}{l}{} &  & \multicolumn{1}{l}{\hspace{0.3em} or $d\equiv
6\ \operatorname{mod}8\ \text{and\ }a\equiv1\ \operatorname{mod}8$}\\ \cmidrule(lr){1-4}
$\mathrm{I}_{4}^{\ast}$ & \multicolumn{1}{l}{$v_{2}\!\left(  b\right)  =2,\ v_{2}%
\!\left(  a\right)  =1,\ d\equiv1\ \operatorname{mod}4$} & $6$ &
\multicolumn{1}{c}{$2$ if $d\equiv5\ \operatorname{mod}8$}\\
& \multicolumn{1}{l}{} &  & \multicolumn{1}{c}{$4$ if $d\equiv
1\ \operatorname{mod}8$}\\ \cmidrule(lr){2-4}
& \multicolumn{1}{l}{$v_{2}\!\left(  b\right)  =2,\ v_{2}\!\left(  a\right)
=1,\ d\equiv3\ \operatorname{mod}4$} & $6$ & \multicolumn{1}{l}{$2$ if
$d\equiv3\ \operatorname{mod}8\text{ and }a\equiv2\ \operatorname{mod}8$}\\
& \multicolumn{1}{l}{} &  & \multicolumn{1}{l}{\hspace{0.3em} or $ d\equiv
7\ \operatorname{mod}8\text{ and }a\equiv6\ \operatorname{mod}8$}\\
& \multicolumn{1}{l}{} &  & \multicolumn{1}{l}{$4$ if $d\equiv
3\ \operatorname{mod}8\text{ and }a\equiv6\ \operatorname{mod}8$}\\
& \multicolumn{1}{l}{} &  & \multicolumn{1}{l}{\hspace{0.3em} or $ d\equiv
7\ \operatorname{mod}8\text{ and }a\equiv2\ \operatorname{mod}8$}\\ \cmidrule(lr){1-4}
$\mathrm{I}_{5}^{\ast}$ & \multicolumn{1}{l}{$v_{2}\!\left(  b\right)  =2,\ ad\equiv
4\ \operatorname{mod}16$} & 6 & \multicolumn{1}{c}{$2$ if $ad\not\equiv
4\ \operatorname{mod}32$}\\
& \multicolumn{1}{l}{} &  & \multicolumn{1}{c}{$4$ if $ad\equiv
4\ \operatorname{mod}32$}\\ \cmidrule(lr){2-4}
& \multicolumn{1}{l}{$v_{2}\!\left(  b\right)  =2,\ ad\equiv
12\ \operatorname{mod}16$} & 6 & \multicolumn{1}{l}{$2$ if $ad\equiv
12\ \operatorname{mod}32\text{ and }a\equiv2\ \operatorname{mod}8$%
}\\
& \multicolumn{1}{l}{} &  & \multicolumn{1}{l}{\hspace{0.3em} or $ ad\equiv
28\ \operatorname{mod}32\text{ and }a\equiv6\ \operatorname{mod}8$}\\
& \multicolumn{1}{l}{} &  & \multicolumn{1}{l}{$4$ if $ad\equiv
12\ \operatorname{mod}32\text{ and }a\equiv6\ \operatorname{mod}8$%
}\\
& \multicolumn{1}{l}{} &  & \multicolumn{1}{l}{\hspace{0.3em} or $ ad\equiv
28\ \operatorname{mod}32\text{ and }a\equiv2\ \operatorname{mod}8$}\\ \cmidrule(lr){1-4}
$\mathrm{I}_{n}^{\ast}$ & \multicolumn{1}{l}{$v_{2}\!\left(  b\right)  \geq
3,\ a \equiv 1\ \operatorname{mod}4,\ v_{2}\!\left(  d\right)  =0,$} & $4$
& \multicolumn{1}{c}{$2$ $\text{if }d\equiv3\ \operatorname{mod}4$}\\
& \multicolumn{1}{l}{$n=2v_{2}\!\left(  b\right)  -2$} &  &
\multicolumn{1}{c}{$4$ $\text{if }d\equiv1\ \operatorname{mod}4$}\\ \cmidrule(lr){2-4}
& \multicolumn{1}{l}{$v_{2}\!\left(  b\right)  \geq3,\ a \equiv
1\ \operatorname{mod}4,\ v_{2}\!\left(  d\right)  =1,$} & $4$ &
\multicolumn{1}{l}{$2$ if $d\equiv2\ \operatorname{mod}8\ \text{and\ }%
a\equiv5\ \operatorname{mod}8$}\\
& \multicolumn{1}{l}{$n=2v_{2}\!\left(
b\right)  -1$} &  & \multicolumn{1}{l}{\hspace{0.3em} or $d\equiv6\ \operatorname{mod}%
8\ \text{and\ }a\equiv1\ \operatorname{mod}8$}\\
& \multicolumn{1}{l}{} &  & \multicolumn{1}{l}{$4$ if $d\equiv
2\ \operatorname{mod}8\ \text{and\ }a\equiv1\ \operatorname{mod}8$%
}\\
& \multicolumn{1}{l}{} &  & \multicolumn{1}{l}{\hspace{0.3em} or $d\equiv6\ \operatorname{mod}%
8\ \text{and\ }a\equiv5\ \operatorname{mod}8$}\\ \cmidrule(lr){2-4}
& \multicolumn{1}{l}{$v_{2}\!\left(  b\right)  \geq3,\ a \equiv
2\ \operatorname{mod}4, \ v_{2}\!\left(  d\right)  =0,$} & $6$ &
\multicolumn{1}{l}{$2$ if $a\equiv2\ \operatorname{mod}8\ \text{and\ }%
d\equiv5,7\ \operatorname{mod}8$}\\
& \multicolumn{1}{l}{$n=2v_{2}\!\left(
b\right)  $} &  & \multicolumn{1}{l}{\hspace{0.3em} or $a\equiv6\ \operatorname{mod}%
8\ \text{and\ }d\equiv3,5\ \operatorname{mod}8$}\\
& \multicolumn{1}{l}{} &  & \multicolumn{1}{l}{$4$ if $a\equiv
2\ \operatorname{mod}8\ \text{and\ }d\equiv1,3\ \operatorname{mod}8$ }\\
& \multicolumn{1}{l}{} &  & \multicolumn{1}{l}{\hspace{0.3em} or $a\equiv6\ \operatorname{mod}%
8\ \text{and\ }d\equiv1,7\ \operatorname{mod}8$}\\ \cmidrule(lr){2-4}
& \multicolumn{1}{l}{$v_{2}\!\left(  b\right)  \geq3,\ a \equiv
2\ \operatorname{mod}4,\ v_{2}\!\left(  d\right)  =1,$} & $6$ &
\multicolumn{1}{l}{$2$ if $ad\equiv12\ \operatorname{mod}32,\ a\equiv
6\ \operatorname{mod}8$}\\
& \multicolumn{1}{l}{$n=2v_{2}\!\left(  b\right)  +1$} &  &
\multicolumn{1}{l}{\hspace{0.3em} or $ad\equiv28\ \operatorname{mod}32,\ a\equiv
2\ \operatorname{mod}8$}\\
& \multicolumn{1}{l}{} &  & \multicolumn{1}{l}{\hspace{0.3em} or $ ad\equiv
20\ \operatorname{mod}32$}\\
& \multicolumn{1}{l}{} &  & \multicolumn{1}{l}{$4$ if $ad\equiv
12\ \operatorname{mod}32,\ a\equiv2\ \operatorname{mod}8$}\\
& \multicolumn{1}{l}{} &  & \multicolumn{1}{l}{\hspace{0.3em} or $ad\equiv
28\ \operatorname{mod}32,\ a\equiv6\ \operatorname{mod}8$}\\
& \multicolumn{1}{l}{} &  & \multicolumn{1}{l}{\hspace{0.3em} or $ad\equiv
4\ \operatorname{mod}32$}\\ \cmidrule(lr){2-4}
& \multicolumn{1}{l}{$v_{2}\!\left(  b\right)  =0,\ n=2v_{2}\!\left(
b^{2}d-a^{2}\right)  -4\geq4$} & $6$ & $4$\\ \cmidrule(lr){2-4}
& \multicolumn{1}{l}{$a\equiv6\ \operatorname{mod}8,\ n=2v_{2}\!\left(
b^{2}d-a^{2}\right)  -12\geq4$} & $4$ & $4$\\ \cmidrule(lr){1-4}
$\mathrm{II}^{\ast}$ & \multicolumn{1}{l}{$v_{2}\!\left(  b\right)  =2,\ v_{2}\!\left(
d\right)  =1,\ a\equiv3\ \operatorname{mod}4$} & $3$ & $1$\\ \cmidrule(lr){1-4}
$\mathrm{III}^{\ast}$ & \multicolumn{1}{l}{$v_{2}\!\left(  b^{2}d-a^{2}\right)
=3,\ v_{2}\!\left(  a\right)  =v_{2}\!\left(  b\right)  =1$} & $7$ & $2$\\ \cmidrule(lr){2-4}
& \multicolumn{1}{l}{$v_{2}\!\left(  b\right)  =v_{2}\!\left(  d\right)
=1,\ v_{2}\!\left(  a\right)  >1$} & $7$ & $2$\\ \cmidrule(lr){2-4}
& \multicolumn{1}{l}{$v_{2}\!\left(  b\right)  =2,\ v_{2}\!\left(  d\right)
=0,\ a\equiv3\ \operatorname{mod}4$} & $3$ & $2$\\ \cmidrule(lr){2-4}
& \multicolumn{1}{l}{$v_{2}\!\left(  b^{2}d-a^{2}\right)  =7,\ a\equiv
2\ \operatorname{mod}8$} & $3$ & $2$\\ \cmidrule(lr){1-4}
$\mathrm{IV}^{\ast}$ & \multicolumn{1}{l}{$v_{2}\!\left(  b\right)  =1,\ a,d\equiv
3\ \operatorname{mod}4$} & $2$ & \multicolumn{1}{c}{$1$ if $ad\equiv
5\ \operatorname{mod}8$}\\
& \multicolumn{1}{l}{} &  & \multicolumn{1}{c}{\ $3$ if $ad\equiv
1\ \operatorname{mod}8$}%
\label{TableforC2}
\end{longtable}
\endgroup}
\begin{proof}
By Theorem~\ref{TisC2Addred2Thm7}, $E_{C_2}$ has additive reduction at $2$ if
and only if the parameters of $E_{C_2}$ satisfy one of the conditions in Table~\ref{TableforC2}. By Lemma~\ref{Lemma for minimal disc}, we get a minimal model for $E_{C_2}$ over $\mathbb{Q}_{2}$ from which we compute its invariants $c_{4},$ $c_{6},$ and
$\Delta$ as given below
\[
c_{4}=u^{-4}\cdot16\left(  3b^{2}d+a^{2}\right),  \qquad c_{6}=-u^{-6}%
\cdot64\left(  9b^{2}d-a^{2}\right),  \qquad\Delta=u^{-12}\cdot64b^{2}d\left(
b^{2}d-a^{2}\right)  ^{2},%
\]
with $u\in\left\{  1,2\right\}$. 
In particular, $u=2$ if and only if $v_{2}\!\left(  b\right)  =1,v_{2}\!\left(
b^{2}d-a^{2}\right)  \geq4$.
Let $F_{C_2,i}$ be the elliptic curve attained from
$E_{C_2}$ via the admissible change of variables $x\longmapsto u_{i}^{2}x+r_{i}$
and $y\longmapsto u_{i}^{3}y+u_{i}^{2}s_{i}x+w_{i}$ where $u_{i},r_{i},s_{i},$
and $w_{i}$ are as given in Table~\ref{ta:WeierTisC2} in Appendix~\ref{AppendixTables}. Note that each
$F_{C_2,i}$ is a minimal model for $E_{C_2}$ under the assumptions on $a,b,d$.

\textbf{Case 1.} Suppose the parameters of $E_{C_2}$ satisfy one of the
conditions appearing in Table~\ref{TableforC2} for the Kodaira-N\'{e}ron types II, $\mathrm{III}$, or IV. Then Table~\ref{ta:PapTableIV} implies that the theorem holds for the
cases below:
\[
\left(  v_{2}\!\left(  c_{4}\right)  ,v_{2}\!\left(  c_{6}\right)
,v_{2}\!\left(  \Delta\right)  \right)  =\left\{
\begin{array}
[c]{ll}%
\left(  4,6,7\right)  & \text{if }v_{2}\!\left(  ab\right)  =v_{2}\!\left(
b^{2}d-a^{2}\right)  =0, v_{2}\!\left(  d\right)  =1,\\
\left(  5,\geq8,9\right)  & \text{if }v_{2}\!\left(  b\right)  =0, v_{2}%
\!\left(  b^{2}d-a^{2}\right)  =1, v_{2}\!\left(  a\right)  >0, v_{2}\!\left(
d\right)  =1,\\
\left(  5,7,8\right)  & \text{if }v_{2}\!\left(  ab\right)  =0,v_{2}%
\!\left(  b^{2}d-a^{2}\right)  =1, d\equiv3\ \operatorname{mod}4.
\end{array}
\right.
\]

For each of the remaining cases we have by (\ref{valIItoIV}) that the Weierstrass model $F_{C_2,i}$ satisfies the first three steps of Tate's Algorithm.

\textbf{Subcase 1a.} Suppose $v_{2}\!\left(  b^{2}d-a^{2}\right)  =4$
with$\ v_{2}\!\left(  a\right)  =v_{2}\!\left(  b\right)  =1$. Then the
Weierstrass coefficient $a_{6}$ of $F_{C_2,1}$ is $a_{6}=-16^{-1}\left(
b^{2}d-a^{2}-8a\right)$. If $b^{2}d-a^{2}-8a\equiv32\ \operatorname{mod}64$,
then $v_{2}\!\left(  a_{6}\right)  =1$, which implies that $E_{C_2}$ has
Kodaira-N\'{e}ron type II with $c_{2}=1$. Moreover, $f_{2}=4$ since
$v_{2}\!\left(  \Delta\right)  =4$. Now suppose $b^{2}d-a^{2}\equiv
8a\ \operatorname{mod}64$ so that $v_{2}\!\left(  a_{6}\right)  \geq2$. Note
that this assumption is equivalent to $b^{2}d-a^{2}\equiv
48\ \operatorname{mod}64$ with $a\equiv6\ \operatorname{mod}8$ or
$b^{2}d-a^{2}\equiv16\ \operatorname{mod}64$ with $a\equiv
2\ \operatorname{mod}8$. If $b^{2}d-a^{2}\equiv48\ \operatorname{mod}64$
with$\ a\equiv6\ \operatorname{mod}8$, then 
$v_{2}\!\left(c_{4}\right)=5  ,v_{2}\!\left(  c_{6}\right)=5  ,\text{ and }v_{2}\!\left(  \Delta\right)=4 $
and $E_{C_2}$ has Kodaira-N\'{e}ron type II or $\mathrm{III}$  by Table~\ref{ta:PapTableIV}. In this case, $v_{2}\!\left(  a_{6}\right)  \geq2$ by (\ref{valIItoIV}) and thus $E_{C_2}$ has Kodaira-N\'{e}ron
type $\mathrm{III}$ with $\left(  c_{2},f_{2}\right)  =\left(
2,3\right)$ by Tate's Algorithm. Lastly, if $b^{2}d-a^{2}\equiv16\ \operatorname{mod}64$
with$\ a\equiv2\ \operatorname{mod}8$, then 
$ v_{2}\!\left(c_{4}\right) \geq 6  ,v_{2}\!\left(  c_{6}\right) =5 ,\text{ and }v_{2}\!\left(  \Delta\right) =4$ and $E_{C_2}$ has Kodaira-N\'{e}ron type II or IV by Table~\ref{ta:PapTableIV}. Since $v_{2}\!\left(  a_{6}\right)
\geq2$, Tate's Algorithm implies that $E_{C_2}$ has Kodaira-N\'{e}ron type IV with $f_{2}=2$. For $c_{2}$, we observe that $t^{2}%
+a_{3,1}t-a_{6,2}\equiv t^{2}+t+a_{6,2}\ \operatorname{mod}2$. Thus $c_{2}$
depends on the parity of $a_{6,2}$. Since $a_{6}=\frac{a}{2}-\frac
{b^{2}d-a^{2}}{16}$, we have that $a_{6}\equiv0\ \operatorname{mod}8$ if and
only if $8a-b^{2}d+a^{2}\equiv0\ \operatorname{mod}128$, which implies
the claimed $c_{2}$.

\begin{equation}{
\scalebox{0.95}{
\renewcommand{\arraystretch}{0.5}
\renewcommand{\arraycolsep}{.2cm}%
$\begin{array}
[c]{ccccc}%
\toprule F_{C_2,i} & \text{Conditions on }a,b,d & v_{2}\!\left(  a_{3}\right)  &
v_{2}\!\left(  a_{4}\right)  & v_{2}\!\left(  a_{6}\right) \\
\toprule F_{C_2,1} & \multicolumn{1}{l}{v_{2}\!\left(  b^{2}d-a^{2}\right)
=4,\ v_{2}\!\left(  a\right)  =v_{2}\!\left(  b\right)  =1,} & 1 & \geq1 & 1\\
& \multicolumn{1}{l}{b^{2}d-a^{2}-8a\equiv32\ \operatorname{mod}64} &  &  & \\
\cmidrule{2-5} & \multicolumn{1}{l}{b^{2}d-a^{2}\equiv48\ \operatorname{mod}%
64,\ a\equiv6\ \operatorname{mod}8} & 1 & \geq2 & \geq2\\
\cmidrule{2-5} & \multicolumn{1}{l}{b^{2}d-a^{2}\equiv16\ \operatorname{mod}%
64,\ a\equiv2\ \operatorname{mod}8} & 1 & \geq 1 & \geq2\\\midrule
F_{C_2,2} & \multicolumn{1}{l}{v_{2}\!\left(  b\right)  =v_{2}\!\left(
b^{2}d-a^{2}\right)  =0,\ v_{2}\!\left(  a\right)  >0,\ d\equiv
3\ \operatorname{mod}4} & \geq2 & \geq1 & 1\\
\cmidrule{2-5} & \multicolumn{1}{l}{v_{2}\!\left(  b\right)  =v_{2}\!\left(
b^{2}d-a^{2}\right)  =0,\ v_{2}\!\left(  a\right)  >0,\ d\equiv
1\ \operatorname{mod}4} & \geq2 & \geq1 & \geq2\\\midrule
F_{C_2,3} & \multicolumn{1}{l}{v_{2}\!\left(  b^{2}d-a^{2}\right)
=5,\ v_{2}\!\left(  a\right)  =v_{2}\!\left(  b\right)  =1} & 2 & 1 & 2\\
\bottomrule &  &  &  &
\end{array}$}
\label{valIItoIV}}
\end{equation}

\textbf{Subcase 1b.} Suppose $v_{2}\!\left(  bd\right)  =v_{2}\!\left(
b^{2}d-a^{2}\right)  =0$ with $v_{2}\!\left(  a\right)  >0$. Then we observe that 
$v_{2}\!\left(  c_{4}\right) =4  ,v_{2}\!\left(  c_{6}\right) \geq 7  ,v_{2}\!\left(\Delta\right)=6$. By Table~\ref{ta:PapTableIV}, $E_{C_2}$ has Kodaira-N\'{e}ron type II or III.
In particular, $E_{C_2}$ has Kodaira-N\'{e}ron type II if and only if the
Weierstrass coefficient $a_{6}$ of $F_{C_2,2}$ satisfies $v_{2}\!\left(
a_{6}\right)  =1$. It follows by (\ref{valIItoIV}) that $E_{C_2}$ has Kodaira-N\'{e}ron
type II (resp.\ $\mathrm{III}$) with $\left(  c_{2},f_{2}\right)  =\left(  1,6\right)  $
(resp.\ $\left(  c_{2},f_{2}\right)  =\left(  2,5\right)  $) if $d\equiv
3\ \operatorname{mod}4$ (resp.\ $d\equiv1\ \operatorname{mod}4$).

\textbf{Subcase 1c.} Suppose $v_{2}\!\left(  b^{2}d-a^{2}\right)  =5$
with$\ v_{2}\!\left(  a\right)  =v_{2}\!\left(  b\right)  =1$. Then it follows that 
$v_{2}\!\left(  c_{4}\right) =4 ,v_{2}\!\left(  c_{6}\right)\geq 7  , \text{ and }v_{2}\!\left(\Delta\right)=6$. By Table
\ref{ta:PapTableIV}, $E_{C_2}$ has the Kodaira-N\'{e}ron type II or $\mathrm{III}$.
Hence $E_{C_2}$ has Kodaira-N\'{e}ron type $\mathrm{III}$ with $\left(  c_{2},f_{2}\right)  =\left(  2,5\right)  $ by (\ref{valIItoIV}) for $F_{C_2,3}$. This completes the proof for Kodaira-N\'{e}ron types II, $\mathrm{III}$, and $\mathrm{IV}$.

\begin{equation}{
\scalebox{0.95}{
\renewcommand{\arraystretch}{0.9}
\renewcommand{\arraycolsep}{.08cm}%
$\begin{array}
[c]{cccccccc}%
\toprule F_{C_2,i} & \text{Conditions on }a,b,d & v_{2}\!\left(  a_{1}\right)  &
v_{2}\!\left(  a_{2}\right)  & v_{2}\!\left(  a_{3}\right)  & v_{2}\!\left(
a_{4}\right)  & v_{2}\!\left(  a_{6}\right)  & v_{2}\!\left(  \Delta\right) \\
\toprule F_{C_2,4} & \multicolumn{1}{l}{v_{2}\!\left(  b\right)  =1,\ a\equiv
3\ \operatorname{mod}4,\ d\equiv1\ \operatorname{mod}4} & 1 & \geq2 & 2 &
3 & 3 & 8\\
\cmidrule{2-8} & \multicolumn{1}{l}{v_{2}\!\left(  b\right)  =1,\ a\equiv
3\ \operatorname{mod}4,\ d\equiv2\ \operatorname{mod}4} & 1 & \geq2 & 3 &
4 & 3 & 9\\
\cmidrule{2-8} & \multicolumn{1}{l}{v_{2}\!\left(  b\right)  =1,\ a\equiv
1\ \operatorname{mod}4,\ d\equiv2\ \operatorname{mod}4} & 1 & 1 & 3 & 4 &
3 & 9\\
\cmidrule{2-8} & \multicolumn{1}{l}{v_{2}\!\left(  b\right)  =1,\ a\equiv
1\ \operatorname{mod}4,\ d\equiv3\ \operatorname{mod}4} & 1 & 1 & 2 & 3 &
3 & 8\\\midrule
F_{C_2,5} & \multicolumn{1}{l}{v_{2}\!\left(  b\right)  =0,\ v_{2}\!\left(
b^{2}d-a^{2}\right)  =2} & 2 & 1 & 6 & 2 & 8 & 10\\\midrule
F_{C_2,6} & \multicolumn{1}{l}{v_{2}\!\left(  b^{2}d-a^{2}\right)
=6,\ a\equiv6\ \operatorname{mod}8} & 1 & \geq2 & 2 & \geq3 & 3 & 8\\
\bottomrule &  &  &  &  &  &  &
\end{array}$}
\label{Tableforc2valforI0}%
}\end{equation}

\textbf{Case 2.} Suppose $a,b,d$ satisfy one of
the conditions in Table~\ref{TableforC2} for the Kodaira-N\'{e}ron type
$\mathrm{I}_{0}^{\ast}$. Next, we consider the curves $F_{C_2,i}$ whose Weierstrass coefficients satisfy the valuations given in \eqref{Tableforc2valforI0} under the listed assumptions on $a,b,d$. Thus the Weierstrass coefficients $a_{j}$ of each of the $F_{C_2,i}$ satisfy the first six steps of Tate's Algorithm.
For these $F_{C_2,i}$, we consider $P(t)=t^{3}%
+a_{2,1}t^{2}+a_{4,2}t+a_{6,3}$. Then  

\vspace{-0.6em}
\begin{equation*}
 P\!\left(  t\right)  \equiv   \begin{cases}
    t^{3}+t^{2}+1\ \operatorname{mod}2 & \quad\text{if }v_{2}\!\left(  b\right)
=1,\ a\equiv1\ \operatorname{mod}4,\ d\equiv2,3\ \operatorname{mod}4,\\
t\left(  t^{2}+t+1\right)  \ \operatorname{mod}2 & \quad\text{if }%
v_{2}\!\left(  b\right)  =0,\ v_{2}\!\left(  b^{2}d-a^{2}\right)  =2,\\
(t+1)(t^{2}+t+1)\ \operatorname{mod}\text{ }2 & \quad\text{if }v_{2}\!\left(
b\right)  =1,\ a\equiv3\ \operatorname{mod}4,\ d\equiv1,2\ \operatorname{mod}%
4,\text{ or}\\
& \qquad v_{2}\!\left(  b^{2}d-a^{2}\right)  =6,\ a\equiv6\ \operatorname{mod}%
8,
\end{cases}
\end{equation*}
\noindent has three distinct roots in $\overline{\mathbb{F}}_{2}$. So, the Kodaira-N\'{e}ron type is
$\mathrm{I}_{0}^{\ast}$ in all these cases. By counting the number of
roots of $P(t)$ in $\mathbb{F}_{2}$, we get the claimed $c_{2}$ as given in
Table~\ref{TableforC2}. Lastly, $f_{2}=v_{2}\!\left(  \Delta\right)  -4$ by Tate's Algorithm, which concludes the proof for this case.

\begin{equation}{
\scalebox{0.92}{
\renewcommand{\arraystretch}{0.4}
\renewcommand{\arraycolsep}{.08cm}%
$\begin{array}
[c]{ccccccccc}%
\toprule n & F_{C_2,i} & \text{Conditions on }a,b,d & v_{2}\!\left(
a_{1}\right)  & v_{2}\!\left(  a_{2}\right)  & v_{2}\!\left(  a_{3}\right)  &
v_{2}\!\left(  a_{4}\right)  & v_{2}\!\left(  a_{6}\right)  & v_{2}\!\left(
\Delta\right) \\
\toprule1 & F_{C_2,4} & \multicolumn{1}{l}{v_{2}\!\left(  b\right)
=1,\ a,d\equiv1\ \operatorname{mod}4} & 1 & 1 & 2 & \geq3 & \geq4 & 8\\
\cmidrule{2-9} & F_{C_2,6} & \multicolumn{1}{l}{v_{2}\!\left(  b^{2}%
d-a^{2}\right)  =6,\ a\equiv2\ \operatorname{mod}8} & 1 & 1 & 2 & \geq3 &
\geq4 & 8\\\midrule
2 & F_{C_2,6} & \multicolumn{1}{l}{v_{2}\!\left(  b^{2}d-a^{2}\right)
=7,\ a\equiv6\ \operatorname{mod}8} & 1 & 1 & 3 & 3 & \geq6 & 10\\
\cmidrule{2-9} & F_{C_2,5} & \multicolumn{1}{l}{v_{2}\!\left(  b\right)
=0,\ v_{2}\!\left(  b^{2}d-a^{3}\right)  =3} & 2 & 1 & 7 & 3 & 10 & 12\\
\cmidrule{2-9} & F_{C_2,7} & \multicolumn{1}{l}{v_{2}\!\left(  b^{2}%
d-a^{2}\right)  =2,\ v_{2}\!\left(  a\right)  =v_{2}\!\left(  b\right)  =1} &
3 & 1 & 3 & 3 & \geq5 & 13\\
\cmidrule{2-9} & F_{C_2,8} & \multicolumn{1}{l}{v_{2}\!\left(  b\right)
=1,\ v_{2}\!\left(  a\right)  >1,\ d\equiv1\ \operatorname{mod}4} & 2 & 1 &
5 & 3 & \geq5 & 12\\
\cmidrule{2-9} & F_{C_2,9} & \multicolumn{1}{l}{v_{2}\!\left(  b\right)
=2,\ v_{2}\!\left(  d\right)  =0,\ a\equiv1\ \operatorname{mod}4} & 1 & 1 &
3 & 3 & \geq5 & 10\\\midrule
3 & F_{C_2,10} & \multicolumn{1}{l}{v_{2}\!\left(  b\right)  =1,\ v_{2}\!\left(
a\right)  >1,\ d\equiv3\ \operatorname{mod}8} & 2 & 1 & 3 & \geq4 & \geq6 &
12\\
\cmidrule{2-9} & F_{C_2,11} & \multicolumn{1}{l}{v_{2}\!\left(  b\right)
=1,\ v_{2}\!\left(  a\right)  >1,\ d\equiv7\ \operatorname{mod}8} & 2 & 1 &
3 & \geq4 & \geq6 & 12\\
\cmidrule{2-9} & F_{C_2,12} & \multicolumn{1}{l}{v_{2}\!\left(  b\right)
=2,\ v_{2}\!\left(  d\right)  =1,\ a\equiv1\ \operatorname{mod}4} & 1 & 1 &
3 & \geq3 & \geq6 & 11\\\midrule
4 & F_{C_2,13} & \multicolumn{1}{l}{v_{2}\!\left(  b\right)  =2,\ v_{2}\!\left(
a\right)  =1,\ d\equiv1\ \operatorname{mod}4} & 2 & 1 & 5 & 4 & \geq7 & 14\\
\cmidrule{2-9} & F_{C_2,14} & \multicolumn{1}{l}{v_{2}\!\left(  b\right)
=2,\ v_{2}\!\left(  a\right)  =1,\ d\equiv3\ \operatorname{mod}4} & 2 & 1 &
4 & 4 & \geq7 & 14\\\midrule
5 & F_{C_2,15} & \multicolumn{1}{l}{v_{2}\!\left(  b\right)  =2,\ ad\equiv
4\ \operatorname{mod}16} & 3 & 1 & 4 & 5 & \geq8 & 15\\
\cmidrule{2-9} & F_{C_2,16} & \multicolumn{1}{l}{v_{2}\!\left(  b\right)
=2,\ ad\equiv12\ \operatorname{mod}16} & 3 & 1 & 4 & \geq6 & \geq8 &
15\\\midrule
n & F_{C_2,17} & \multicolumn{1}{l}{n=2v_{2}\!\left(  b\right)  -2\geq
4,\ a\equiv1\ \operatorname{mod}4,} & 1 & 1 & >\frac{n+3}{2} & \frac{n+4}{2} &
\geq n+3 & n+8\\
&  & \multicolumn{1}{l}{v_{2}\!\left(  d\right)  =0} &  &  &  &  &  & \\
\cmidrule{2-9} & F_{C_2,18} & \multicolumn{1}{l}{n=2v_{2}\!\left(  b\right)
-1\geq5,\ a\equiv1\ \operatorname{mod}4,} & 1 & 1 & \frac{n+3}{2} &
>\frac{n+4}{2} & \geq n+3 & n+8\\
&  & \multicolumn{1}{l}{v_{2}\!\left(  d\right)  =1} &  &  &  &  &  & \\
\cmidrule{2-9} & F_{C_2,19} & \multicolumn{1}{l}{n=2v_{2}\!\left(  b\right)
\geq6,\ a\equiv2\ \operatorname{mod}4,} & 2 & 1 & >\frac{n+3}{2} & \frac{n+4}{2}
& \geq n+3 & n+10\\
&  & \multicolumn{1}{l}{v_{2}\!\left(  d\right)  =0} &  &  &  &  &  & \\
\cmidrule{2-9} & F_{C_2,20} & \multicolumn{1}{l}{n=2v_{2}\!\left(  b\right)
+1\geq7,\ a\equiv2\ \operatorname{mod}4,} & 2 & 1 & \frac{n+3}{2} & >\frac
{n+4}{2} & \geq n+3 & n+10\\
&  & \multicolumn{1}{l}{ad\equiv4\ \operatorname{mod}16} &  &  &  &  &  & \\
\cmidrule{2-9} & F_{C_2,21} & \multicolumn{1}{l}{n=2v_{2}\!\left(  b\right)
+1\geq7,\ a\equiv2\ \operatorname{mod}4,} & 2 & 1 & \frac{n+3}{2} & >\frac
{n+4}{2} & \geq n+3 & n+10\\
&  & \multicolumn{1}{l}{ad\equiv12\ \operatorname{mod}16} &  &  &  &  &  & \\
\cmidrule{2-9} & F_{C_2,5} & \multicolumn{1}{l}{n=2v_{2}\!\left(  b^{2}%
d-a^{2}\right)  \geq4,\ v_{2}\!\left(  b\right)  =0} & 2 & 1 & >\frac{n+3}{2} &
\frac{n+4}{2} & >n+3 & n+10\\
\cmidrule{2-9} & F_{C_2,22} & \multicolumn{1}{l}{n=2v_{2}\!\left(  b^{2}%
d-a^{2}\right)  -12\geq4,} & 1 & 1 & >\frac
{n+3}{2} & \frac{n+4}{2} & >n+3 & n+8\\
& & \multicolumn{1}{l}{ a\equiv6\ \operatorname{mod}8} &  & &  &  &  & \\
\bottomrule &  &  &  &  &  &  &  &
\end{array}$}
\ \ \label{Tableforc2valforInstar}}
%\vspace{-1em}
\end{equation}

\textbf{Case 3.} Suppose that $a,b,d$ satisfy one of
the conditions in Table~\ref{TableforC2} for the Kodaira-N\'{e}ron type
$\mathrm{I}_{n}^{\ast}$ where $n>0$. Now, we consider the Weierstrass models $F_{C_2,i}$ listed in \eqref{Tableforc2valforInstar} along with the $2$-adic valuations of its
Weierstrass coefficients $a_{j}$ and the minimal
discriminant $\Delta$.
It follows from (\ref{Tableforc2valforInstar}) and Lemma~\ref{LemmaIn} that the
Kodaira-N\'{e}ron type and $f_{2}$ are as claimed for each of these cases. Indeed,%
\vspace{-.5em}
\begin{equation}
\left.
\begin{array}
[c]{rl}%
t^{2}+a_{3,\frac{n+3}{2}}t+a_{6,n+3} & \text{if }n\text{ is odd}\\
a_{2,1}t^{2}+a_{4,\frac{n+4}{2}}t+a_{6,n+3} & \text{if }n\text{ is even}%
\end{array}
\right\}  \equiv t^{2}+t+a_{6,n+3}\ \operatorname{mod}2 \label{TamaC2Infixed}%
\end{equation}
has distinct roots in $\overline{\mathbb{F}}_{2}$. Then from  (\ref{TamaC2Infixed}), we note that $c_{2}=4$ if $a_{6,n+3}$ is even. Otherwise, $c_{2}=2$. Consequently, if (i)  $v_{2}\!\left(  b^{2}d-a^{2}\right)  =7,\ a\equiv
6\ \operatorname{mod}8$, (ii) $v_{2}\!\left(  b\right)  =0,\ v_{2}\!\left(
b^{2}d-a^{3}\right)  =3$, (iii) $v_{2}\!\left(  b\right)  =0,\ v_{2}\!\left(  b^{2}%
d-a^{2}\right)  \geq4$, or (iv) $a\equiv6\ \operatorname{mod}8,\ v_{2}\!\left(
b^{2}d-a^{2}\right)  \geq8$, then $c_{2}=4$ since each respective $a_{6,n+3}$ is even by \eqref{Tableforc2valforInstar}. We consider the remaining cases separately using (\ref{Instara6}), which lists the
Weierstrass coefficient $a_{6}$ of $F_{C_2,i}$ under the remaining assumptions.

Since $a_{6,n+3}$ is an integer, we have that $a_{6,n+3}$ is even if and
only if $a_{6}\equiv0\ \operatorname{mod}2^{n+4}$. Below we show the proof of the local Tamagawa numbers for some of the cases and remark that the same method yields the $c_{2}$'s for the cases not considered below.

Suppose $v_{2}\!\left(  b^{2}d-a^{2}\right)  =6,\ a\equiv2\ \operatorname{mod}%
4$. Then $E_{C_2}$ has Kodaira-N\'{e}ron type $\mathrm{I}_{1}^{\ast}$ under these
conditions. To compute $c_{2}$ in this case, we consider $a_{6}$ of the curve $F_{C_2,6}$ modulo $32$ as in (\ref{Instara6}). Since $2^{-12}\left(b^{2}d-a^{2}\right)  ^{2}$ is an odd integer, it suffices to consider $2^{-3}\left(
b^{2}d-a^{2}+16a-96\right)  \ \operatorname{mod}32$. In particular,
$a_{6,n+3}$ is even if and only if $b^{2}d-a^{2}+16a\equiv
96\ \operatorname{mod}256$, which gives the claimed~$c_2$.

Suppose $v_{2}\!\left(  b\right)  =1,\ v_{2}\!\left(  a\right)  >1,\ d\equiv
3\ \operatorname{mod}8$. Then $E_{C_2}$ has Kodaira-N\'{e}ron type I$_{3}%
^{\ast}$ under these conditions. Next, we consider $a_6$ for $F_{C_2,10}$ as in (\ref{Instara6}). Since $v_{2}\!\left(  b\right)  =1$,
we observe that $a_{6,6}$ is even if and only if $\frac{a_{6}}{b^{2}}%
\equiv0\ \operatorname{mod}32$. Next, let $a=4k$ and $b=2j$ for some integers
$k,j$ with $j$ odd. Then, $d\equiv3\ \operatorname{mod}8$ implies that
$j^{2}d$ is $3$ or $11$ modulo $16$, and $\frac{a_{6}}{b^{2}}=2\left(
19j^{4}+36kj-3j^{2}d+12k^{2}-2j^{2}-8j^{3}\right)  $. It follows that
$\frac{a_{6}}{2b^{2}}\equiv25+4k\left(  1+3k\right)  -3j^{2}%
d\ \operatorname{mod}16$. If $k\equiv0,1\ \operatorname{mod}4$ (resp.\ $k\equiv2,3\ \operatorname{mod}4$), then $4k\left(  1+3k\right)  $ is
congruent to $0$ (resp.\ $8$) modulo $16$. It now follows that $\frac{a_{6}%
}{2b^{2}}\equiv0\ \operatorname{mod}16$ if
and only if $\left(  i\right)  $ $k\equiv0,1\ \operatorname{mod}4$ and
$j^{2}d\equiv3\ \operatorname{mod}16$ or $\left(  ii\right)  \ k\equiv
2,3\ \operatorname{mod}4$ and $j^{2}d\equiv11\ \operatorname{mod}16$. This gives us the claimed $c_2$ since $k=\frac{a}{4}$ and
$j^{2}d=\frac{b^{2}d}{4}$.

Suppose $v_{2}\!\left(  b\right)  =2,\ ad\equiv12\ \operatorname{mod}16$. By the above, $E_{C_2}$ has Kodaira-N\'{e}ron type $\mathrm{I}_{5}^{\ast}$ under these
conditions. Now let $a_{6}$ be as given in \eqref{Instara6} for $F_{C_2,16}$. Since $v_{2}\!\left(  b^{2}\left(
ad-4\right)  \right)  =7$, we have that $a_{6,8}$ is even if and only
if $\frac{2a_{6}}{b^{2}\left(  ad-4\right)  }\equiv0\ \operatorname{mod}8$. Next, let $a=2k,\ b=4j,$ and $d=2l$ for some odd integers $k,j,l$. In
particular, $kl\equiv3\ \operatorname{mod}4$ since $ad\equiv
12\ \operatorname{mod}16$ and thus $k-l$ is either $2$ or $6$ modulo $8$. Now observe that $\frac{2a_{6}}{b^{2}\left(  ad-4\right)} = 2+2j+k-4jl-8kl-4jkl-k^{2}l+2jk^{2}l^{2} $. In particular, $\frac{2a_{6}}{b^{2}\left(  ad-4\right)} \equiv6+k-l\ \operatorname{mod}8$. Note that $k-l\equiv2\ \operatorname{mod}8$ if and only if $kl\equiv3\ \operatorname{mod}8$ with $k\equiv3\ \operatorname{mod}8$ or $kl\equiv7\ \operatorname{mod}8$ with $k\equiv1\ \operatorname{mod}8$. Similarly, $k-l\equiv6\ \operatorname{mod}8$ if and only if $kl\equiv3\ \operatorname{mod}8$ with $k\equiv1\ \operatorname{mod}8$ or $kl\equiv7\ \operatorname{mod}8$ with $k\equiv3\ \operatorname{mod}8$. The claim now follows since $kl=\frac{ad}{4}$ and $c_2=4$ if and and only if $k-l\equiv 2\ \operatorname{mod}8$.

Lastly, we consider the case $n=2v_{2}\!\left(  b\right)  +1\geq7,a\equiv2\ \operatorname{mod}4$, and $ad\equiv12\ \operatorname{mod}16$. Now, we consider $a_6$ for $F_{C_2,22}$ as in (\ref{Instara6}).
We observe that, $c_{2}=4$ if and only if $ \frac{a_{6}}{b^{2}} \equiv 0 \ \operatorname{mod}32$. By
(\ref{Instara6}) we deduce that $\frac{a_{6}}{b^{2}}\equiv\frac{-a^{3}d^{2}%
}{4}+ad-4\ \operatorname{mod}32$. Next, we note that $\frac{a_{6}}{4b^{2}} \equiv-a+\frac{ad}{4}-1\ \operatorname{mod}8$ since  $\frac{a^{2}d^{2}}{16}\equiv
1\ \operatorname{mod}8$. The claimed $c_{2}$
now follows from the fact that $-a+\frac{ad}{4}-1\equiv 0\ \operatorname{mod}8$ if and only if (i) $a \equiv 2\ \operatorname{mod}8$ and $\frac{ad}{4} \equiv 3\ \operatorname{mod}8$ or (ii) $a \equiv 6 \ \operatorname{mod}8$ and $\frac{ad}{4} \equiv 7\ \operatorname{mod}8$.

\begin{equation}
{
\renewcommand{\arraystretch}{0.8}
\renewcommand{\arraycolsep}{.123cm}%
\scalebox{0.93}{
$\begin{array}
[c]{ccc}%
\toprule F_{C_2,i} & \text{ Conditions on }a,b,d & a_{6}\\
\toprule F_{C_2,5} & \multicolumn{1}{l}{v_{2}\!\left(  b\right)  =1,\ a,d\equiv
1\ \operatorname{mod}4} & d\left(  ab^{2}-4d\right) \\\midrule
F_{C_2,7} & \multicolumn{1}{l}{v_{2}\!\left(  b^{2}d-a^{2}\right)
=6,\ a\equiv2\ \operatorname{mod}4} & 2^{-15}\left(  b^{2}%
d-a^{2}+16a-96\right)  \left(  b^{2}d-a^{2}\right)  ^{2}\\\midrule
F_{C_2,8} & \multicolumn{1}{l}{v_{2}\!\left(  b^{2}d-a^{2}\right)
=2,\ v_{2}\!\left(  a\right)  =v_{2}\!\left(  b\right)  =1} & d\left(
d^{5}-b^{2}d^{2}-ad^{3}+ab^{2}-4d\right) \\\midrule
F_{C_2,9} & \multicolumn{1}{l}{v_{2}\!\left(  b\right)  =1,\ v_{2}\!\left(
a\right)  >1,\ d\equiv1\ \operatorname{mod}4} & -2^{-3}b^{2}\left(
8b^{6}-b^{4}-4ab^{2}+4b^{2}d-4a^{2}\right) \\\midrule
F_{C_2,10} & \multicolumn{1}{l}{v_{2}\!\left(  b\right)  =2,\ v_{2}\!\left(
d\right)  =0,\ a\equiv1\ \operatorname{mod}4} & 2^{-3}b^{2}\left(
a^{3}b^{4}d^{3}-2a^{3}b^{2}d^{2}-4ab^{2}d^{2}+8ad-8\right) \\\midrule
F_{C_2,11} & \multicolumn{1}{l}{v_{2}\!\left(  b\right)  =1,\ v_{2}\!\left(
a\right)  >1,\ d\equiv3\ \operatorname{mod}8} & 2^{-3}b^{2}(
19b^{4}+36ab^{2}-16b^{3} \\
 &  & -12b^{2}d+12a^{2}-8b^{2}) \\\midrule
F_{C_2,12} & \multicolumn{1}{l}{v_{2}\!\left(  b\right)  =1,\ v_{2}\!\left(
a\right)  >1,\ d\equiv7\ \operatorname{mod}8} & 2^{-3}b^{2}\left(
4ab^{2}-8b^{6}-15b^{4}-4b^{2}d+4a^{2}-8b^{2}\right) \\\midrule
F_{C_2,13} & \multicolumn{1}{l}{v_{2}\!\left(  b\right)  =2,\ v_{2}\!\left(
d\right)  =1,\ a\equiv1\ \operatorname{mod}4} & 4a^{3}-ab^{2}d-2b^{2}%
d+16a^{2}+20a-8\\\midrule
F_{C_2,14} & \multicolumn{1}{l}{v_{2}\!\left(  b\right)  =2,\ v_{2}\!\left(
a\right)  =1,\ d\equiv1\ \operatorname{mod}4} & 4a^{3}-ab^{2}d-16b^{2}\\\midrule
F_{C_2,15} & \multicolumn{1}{l}{v_{2}\!\left(  b\right)  =2,\ v_{2}\!\left(
a\right)  =1,\ d\equiv3\ \operatorname{mod}4} & 4a^{3}-ab^{2}d-16b^{2}%
+128bd\\

 &  & -256d^{2}-64b+256d-64\\\midrule
F_{C_2,16} & \multicolumn{1}{l}{v_{2}\!\left(  b\right)  =2,\ ad\equiv
4\ \operatorname{mod}16} & b^{2}\left(  ad-4\right) \\\midrule
F_{C_2,17} & \multicolumn{1}{l}{v_{2}\!\left(  b\right)  =2,\ ad\equiv
12\ \operatorname{mod}16} & 2^{-6}b^{2}\left(  ad-4\right)  (a^{2}%
d^{2}b-4a^{2}d-8abd+\\
&  & 16a+16b-16bd-64ad+64)\\\midrule
F_{C_2,18} & \multicolumn{1}{l}{n=2v_{2}\!\left(  b\right)  -2\geq
4,\ a\equiv1\ \operatorname{mod}4,} & b^{2}\left(  b^{4}-ab^{2}-b^{2}%
d+ad-1\right) \\
& \multicolumn{1}{l}{v_{2}\!\left(  d\right)  =0} & \\\midrule
F_{C_2,19} & \multicolumn{1}{l}{n=2v_{2}\!\left(  b\right)  -1\geq
5,\ a\equiv1\ \operatorname{mod}4,} & 2^{-3}b^{2}\left(  a^{3}%
d^{3}b-2a^{3}d^{2}-4abd^{2}+8ad-8\right) \\
& \multicolumn{1}{l}{v_{2}\!\left(  d\right)  =1} & \\\midrule
F_{C_2,20} & \multicolumn{1}{l}{n=2v_{2}\!\left(  b\right)  \geq6,\ a\equiv
2\ \operatorname{mod}4,} & 2^{-3}b^{2}(  a^{3}d^{3}b-2a^{3}%
d^{2}-4abd^{2} \\
 &  \multicolumn{1}{l}{v_{2}\!\left(  d\right)  =0} & +8ad-8d^{2}+16d-8) \\\midrule
F_{C_2,21} & \multicolumn{1}{l}{n=2v_{2}\!\left(  b\right)  +1\geq
7,\ a\equiv2\ \operatorname{mod}4,} & b^{2}\left(  a^{3}d^{3}b-a^{3}%
d^{2}-abd^{2}+ad-4\right) \\
& \multicolumn{1}{l}{ad\equiv4\ \operatorname{mod}16} & \\\midrule
F_{C_2,22} & \multicolumn{1}{l}{n=2v_{2}\!\left(  b\right)  +1\geq
7,\ a\equiv2\ \operatorname{mod}4,} & 2^{-3}b^{2}\left(  a^{3}d^{3}%
b-2a^{3}d^{2}-4abd^{2}+8ad-32\right) \\
& \multicolumn{1}{l}{ad\equiv12\ \operatorname{mod}16} & \\
\bottomrule
\end{array}$}
\ \ \label{Instara6}
}
\end{equation}

\textbf{Case 4.} Suppose $a,b,d$ satisfy one of the
conditions in Table~\ref{TableforC2} for Kodaira-N\'{e}ron types
$\mathrm{{II}^{\ast}}$, $\mathrm{{III}^{\ast}}$, or $\mathrm{{IV}^{\ast}}$. Then the theorem holds for these cases by Table~\ref{ta:PapTableIV} and the valuations below: 
\[
\left(  v_{2}\!\left(  c_{4}\right)  ,v_{2}\!\left(  c_{6}\right)
,v_{2}\!\left(  \Delta\right)  \right)  =\left\{
\begin{array}
[c]{ll}%
\left(  7,10,14\right)  & \text{if }v_{2}\!\left(  b^{2}d-a^{2}\right)
=3,\ v_{2}\!\left(  a\right)  =v_{2}\!\left(  b\right)  =1,\\
\left(  7,\geq11,15\right)  & \text{if }v_{2}\!\left(  b\right)
=v_{2}\!\left(  d\right)  =1,\ v_{2}\!\left(  a\right)  >1.
\end{array}
\right.
\]

For the remaining cases, \eqref{Tableforc2valforII*} lists the valuations of the Weierstrass
model $F_{C_2,i}$. By \eqref{Tableforc2valforII*}, each $F_{C_2,i}$ satisfies the first six steps of Tate's Algorithm and proceeds to Step 8 since $P(t)=t^{3}+a_{2,1}t^{2}+a_{4,2}t+a_{6,3}\equiv t^{3} \operatorname{mod}2$ has a triple root. Then, $E_{C_2}$ has the claimed Kodaira-N\'{e}ron type and $f_{2}$ by using Table~\ref{ta:PapTableIV} and the following triplet

\vspace{-0.1in}

\[
\left(  v_{2}\!\left(  c_{4}\right)  ,v_{2}\!\left(  c_{6}\right)
,v_{2}\!\left(  \Delta\right)  \right)  =\left\{
\begin{array}
[c]{ll}%
\left(  4,6,8\right)  & \text{if }v_{2}\!\left(  b\right)  =1,\ a,d\equiv3\ \operatorname{mod}4,\\
\left(  4,6,10\right)  & \text{if }v_{2}\!\left(  b^{2}d-a^{2}\right)
=7,\ a\equiv2\ \operatorname{mod}8,\\
\left(  4,6,10\right)  & \text{if }v_{2}\!\left(  b\right)  =2,\ v_{2}%
\!\left(  ad\right)  =0,\ a\equiv3\ \operatorname{mod}4,\\
\left(  4,6,11\right)  & \text{if }v_{2}\!\left(  b\right)  =2,\ v_{2}%
\!\left(  a\right)  =0,\ v_{2}\!\left(  d\right)  =1,\ a\equiv
3\operatorname{mod}4.
\end{array}
\right.
\]
\begin{equation}{
\renewcommand{\arraystretch}{0.2}\renewcommand{\arraycolsep}{.1cm}%
\scalebox{0.93}{
$\begin{array}
[c]{cccccccc}%
\toprule F_{C_2,i} & \text{Conditions on }a,b,d & v_{2}\!\left(  a_{1}\right)  &
v_{2}\!\left(  a_{2}\right)  & v_{2}\!\left(  a_{3}\right)  & v_{2}\!\left(
a_{4}\right)  & v_{2}\!\left(  a_{6}\right)  & \\
\toprule F_{C_2,4} & \multicolumn{1}{l}{v_{2}\!\left(  b\right)  =1,\ a,d\equiv3\ \operatorname{mod}4} & 1 & \geq2 & 2 & 3 &
\geq4 & \\\midrule
F_{C_2,6} & \multicolumn{1}{l}{v_{2}\!\left(  b^{2}d-a^{2}\right)
=7,\ a\equiv2\ \operatorname{mod}8} & 1 & \geq2 & 3 & 3 & 5 & \\\midrule
F_{C_2,9} & \multicolumn{1}{l}{v_{2}\!\left(  b\right)  =2,\ v_{2}\!\left(
ad\right)  =0,\ a\equiv3\ \operatorname{mod}4} & 1 & \geq2 & 3 & 3 & \geq5 &
\\\midrule
F_{C_2,12} & \multicolumn{1}{l}{v_{2}\!\left(  b\right)  =2,\ v_{2}\!\left(
a\right)  =0,\ v_{2}\!\left(  d\right)  =1,\ a\equiv3\ \operatorname{mod}4} &
1 & \geq2 & 3 & 3 & 4 & \\
\bottomrule &  &  &  &  &  &  &
\end{array}$}
\ \label{Tableforc2valforII*}}
\end{equation}
In particular, Table~\ref{ta:PapTableIV} determines $c_{2}$ except for when $v_{2}\!\left(  b\right)  =1,\ a,d\equiv
3\ \operatorname{mod}4$. Under these assumptions, $E_{C_2}$ has Kodaira-N\'{e}ron type $\mathrm{IV}^{\ast}$ and the Weierstrass coefficients $a_{j}$ of $F_{C_2,4}$ satisfy
$
t^{2}+a_{3,2}t-a_{6,4}\equiv t^{2}+t-\frac{a-d}{4}\text{ }\operatorname{mod}%
2.
$
This gives the claimed $c_{2}$ since $a-d\equiv0$ (resp.\ $4$) $\operatorname{mod}8$ if and only if $ad\equiv1$ (resp.\ 5) $\operatorname{mod}8$.
\end{proof}

%%%%%%%%%%%%%%%%%%%%%%%%%%%%%%%%%%
\noindent Next, we consider the case when $E_{C_2}$ has additive reduction at an odd prime
$p$.
\begin{theorem}
\label{ThmforC2podd}
The family of elliptic curves $E_{C_{2}}$ has additive reduction at an odd prime $p$ if and only if the parameters of $E_{C_2}$ satisfy one of the conditions listed in Table~\ref{TableforC2podd}. Moreover, the Kodaira-N\'{e}ron type at $p$,
conductor exponent~$f_{p}$, and local Tamagawa number $c_{p}$ are given as in Table~\ref{TableforC2podd}.
\end{theorem}
{\begingroup
\renewcommand{\arraystretch}{1.1}
 \begin{longtable}{ccccc}
	\hline
	$p$ & K-N type & Conditions on $a,b,d$ & $c_{p}$ & $f_{p}$ \\
	\hline

	\endfirsthead
	\hline
	$p$ & K-N type & Conditions on $a,b,d$ & $c_{p}$ & $f_{p}$ \\
	\hline
	\endhead
	\hline
	\multicolumn{4}{r}{\emph{continued on next page}}
	\endfoot
	\hline
	\caption{ Local data at odd primes for $E_{C_{2}}(a,b,d)$}
	\endlastfoot

$3$ & $\mathrm{I}_{0}^{\ast}$ & $v_{p}\!\left(  a\right)  =v_{p}\!\left(b\right)  =1,\ v_{p}\!\left(  d\right)  =0,$ & $2 $ & $ 2$\\
&  & $ v_{p}\!\left(  b^{2}d-a^{2}\right)  =2 $&  & \\\midrule
$p\geq5 $ & $ \mathrm{I}_{0}^{\ast} $ & $ v_{p}\!\left(  a\right)  =v_{p}\!\left(
b\right)  =1,\ v_{p}\!\left(  d\right)  =0, $ &  \multicolumn{1}{l}{$2\ \text{ if}\big(\frac{d}{p}\big)=-1$} & $ 2$\\
&  & $v_{p}\!\left(  b^{2}d-a^{2}\right)  =2,\ v_{p}\!\left(  9b^{2}%
d-a^{2}\right)  =2 $ & \multicolumn{1}{l}{$4\ \text{ if }\big(\frac{d}%
{p}\big)=1$}& \\
\cmidrule{3-5} &  & $v_{p}\!\left(  a\right)  =v_{p}\!\left(  b\right)
=1,\ v_{p}\!\left(  d\right)  =0, $ & $ 4 $ & $ 2$\\
&  & $ v_{p}\!\left(  b^{2}d-a^{2}\right)  =2,\ v_{p}\!\left(  9b^{2}%
d-a^{2}\right)  \geq3 $ &  & \\\midrule
$\neq2 $ & $ \mathrm{III} $ & $ v_{p}\!\left(  a\right)  \geq1,\ v_{p}\!\left(  b\right)
=0,\ v_{p}\!\left(  d\right)  =1 $ & $ 2 $ & $ 2$\\
\cmidrule{2-5} & $\mathrm{I}_{0}^{\ast} $ & $ v_{p}\!\left(  a\right)
>1,\ v_{p}\!\left(  b\right)  =1,\ v_{p}\!\left(  d\right)  =0 $ & 
\multicolumn{1}{l}{$2\ \text{ if }\big(\frac{d}{p}\big)=-1$} & $ 2$\\
&  &  & \multicolumn{1}{l}{$4\ \text{ if }\big(\frac{d}{p}\big)=1$} & \\
\cmidrule{2-5}  & $ \mathrm{I}_{n}^{\ast} $ & $ v_{p}\!\left(  a\right)
=1,\ v_{p}\!\left(  d\right)  =1, $ & \multicolumn{1}{l}{$2\ \text{ if }%
\big(\frac{ad/p^{2}}{p}\big)=-1$} & $ 2$\\
&   & $ n=2v_{p}\!\left(  b\right)  -1\geq1 $ &  \multicolumn{1}{l}{$4\ \text{ if
}\big(\frac{ad/p^{2}}{p}\big)=1$} &  \\
\cmidrule{3-5} &   & $v_{p}\!\left(  a\right)  =1,\ v_{p}\!\left(  d\right)  =0, $ &
\multicolumn{1}{l}{$2\ \text{ if }\big(\frac{d}{p}\big)=-1$} & $ 2$\\
&  & $ n=2v_{p}\!\left(  b\right)  -2\geq2 $ &  \multicolumn{1}{l}{$4\ \text{ if
}\big(\frac{d}{p}\big)=1$}  & \\
\cmidrule{3-5} &   & $v_{p}\!\left(  a\right)  =v_{p}\!\left(  b\right)
=1,\ v_{p}\!\left(  d\right)  =0, $ & $ 4 $ & $ 2$\\
&  & $n=2v_{p}\!\left(  b^{2}d-a^{2}\right)  -4\geq2 $ &  & \\
\cmidrule{2-5} & $\mathrm{III}^{\ast}$ & $v_{p}\!\left(  a\right)  >1,\ v_{p}%
\!\left(  b\right)  =1,\ v_{p}\!\left(  d\right)  =1$ & $2$ & $2$
\label{TableforC2podd}
\end{longtable}
\endgroup}

\begin{proof}
By Theorem~\ref{TisC2Addred2Thm7}, $E_{C_2}$ has additive reduction at an odd prime
$p$ if and only if $p$ divides $\gcd\!\left(  a,bd\right)  $. In particular, this condition is equivalent to the parameters of $E_{C_2}$ satisfying one of the conditions in
Table~\ref{TableforC2podd}. Consequently, $E_{C_2}/\mathbb{Q}_{p}$ is a minimal model by Lemma~\ref{Lemma for minimal disc} and thus the invariants $c_{4}$, $c_{6}$, and $\Delta$ satisfy
\[
v_{p}(c_{4})=v_{p}(3b^{2}d+a^{2}),\qquad v_{p}(c_{6})=v_{p}(a)+v_{p}
(9b^{2}d-a^{2}),\qquad v_{p}(\Delta)=v_{p}(b^{2}d)+2v_{p}(b^{2}d-a^{2}).
\]
In particular, if $a,b,d$ satisfy one of the conditions appearing in
Table~\ref{TableforC2podd} for the Kodaira-N\'{e}ron type $\mathrm{III}$ or $\mathrm{III}^{\ast}$, then
\[
\left(  v_{p}\!\left(  c_{4}\right)  ,v_{p}\!\left(  c_{6}\right)
,v_{p}\!\left(  \Delta\right)  \right)  =\left\{
\begin{array}
[c]{cl}%
\left(  1,\geq2,3\right)   & \text{if }p\geq5,v_{p}\!\left(  a\right)
\geq1,v_{p}\!\left(  b\right)  =0,v_{p}\!\left(  d\right)  =1,\\
\left(  \geq2,3,3\right)   & \text{if }p=3,v_{3}\!\left(  a\right)
=1,v_{3}\!\left(  b\right)  =0,v_{3}\!\left(  d\right)  =1,\\
\left(  2,\geq5,3\right)   & \text{if }p=3,v_{3}\!\left(  a\right)
>1,v_{3}\!\left(  b\right)  =0,v_{3}\!\left(  d\right)  =1,\\
\left(  3,\geq5,9\right)   & \text{if }p\geq5,v_{p}\!\left(  a\right)
\geq2,v_{p}\!\left(  b\right)  =v_{p}\!\left(  d\right)  =1,\\
\left(  \geq4,6,9\right)   & \text{if }p=3,v_{3}\!\left(  a\right)
=2,v_{3}\!\left(  b\right)  =v_{3}\!\left(  d\right)  =1,\\
\left(  4,\geq8,9\right)   & \text{if }p=3,v_{3}\!\left(  a\right)
>2,v_{3}\!\left(  b\right)  =v_{3}\!\left(  d\right)  =1.
\end{array}
\right.
\]

By Table~\ref{ta:PapTableIandII} we have that the claim holds for each case
except possibly when $p=3$ with $\left(  i\right)  $ $v_{3}\!\left(  a\right)
=1,v_{3}\!\left(  b\right)  =0,v_{3}\!\left(  d\right)  =1$ or $\left(
ii\right)  $ $v_{3}\!\left(  a\right)  =2,v_{3}\!\left(  b\right)
=v_{3}\!\left(  d\right)  =1$ which require verification of the additional
condition. If $v_{3}\!\left(  a\right)  =1,v_{3}\!\left(  b\right)
=0,v_{3}\!\left(  d\right)  =1$, then $a=3k$ and $d=3l$ for some integers $k$
and $l$ with $3\nmid kl$. Moreover,%
\begin{align*}
	\left(  \frac{c_{6}}{27}\right)  ^{2}-\frac{c_{4}}{3}+2  & \equiv3\left(
	b^{2}k^{4}l+2b^{2}l+2k^{2}\right)  +k^{6}+2\ \operatorname{mod}9 = 0\ \operatorname{mod}9
\end{align*}
since $k^{6}+2\equiv3\ \operatorname{mod}9$. Hence $E_{C_2}$ has Kodaira-N\'{e}ron type
$\mathrm{III}$ at $3$ with $f_{3}=c_{3}=2$ by Table~\ref{ta:PapTableIandII}. A
similar argument shows that if $v_{3}\!\left(  a\right)  =2,v_{3}\!\left(
b\right)  =v_{3}\!\left(  d\right)  =1$, then $E_{C_2}$ has Kodaira-N\'{e}ron type
$\mathrm{III}^{\ast}$ at $3$ with $f_{3}=c_{3}=2$ by
Table~\ref{ta:PapTableIandII}, since $\left(  \frac{c_{6}}{729}\right)
^{2}-\frac{c_{4}}{27}+2\equiv0\ \operatorname{mod}9$. 

For the remaining cases, let $F_{C_2,i}$ be the elliptic curve attained from $E_{C_2}$ via the
admissible change of variables $x\longmapsto x+r_{i}$ and $y\longmapsto
y+s_{i}x+w_{i}$ where $r_{i},s_{i},w_{i}$ are as given in
Table~\ref{ta:WeierTisC2} for $i=23,24,\ldots,26$. Then the Weierstrass
coefficients $a_{j}$ of $F_{C_2,i}$ satisfy the valuations given in
Table~\ref{TableforC2podd}. In particular, each $F_{C_2,i}$ satisfies the first six steps of Tate's Algorithm.

\textbf{Case 1.} Suppose that the parameters $a,b,d$ of $E_{C_2}$
satisfy one of the conditions appearing in Table~\ref{TableforC2podd} for the
Kodaira-N\'{e}ron type $\mathrm{I}_{0}^{\ast}$. First, we consider $P(t)=t^{3}%
+a_{2,1}t^{2}+a_{4,2}t+a_{6,3}$ where $a_{j}$ are the Weierstrass coefficients 
of  $F_{C_2,23}$ and $F_{C_2,24}$ under the assumptions given in
\eqref{Tableforc2poddvals}. In particular,%
\begin{equation*}{\small
P\!\left(  t\right)  \equiv\left\{
\begin{array}
[c]{ll}%
t\left(  t^{2}+\frac{2a}{3}t+2\right)   \operatorname{mod}3 & \text{if }%
v_{3}\!\left(  a\right)  =v_{3}\!\left(  b\right)  =1, v_{3}\!\left(
d\right)  =0, v_{3}\!\left(  b^{2}d-a^{2}\right)  =2,\\
t\left(  t^{2}+\frac{2a}{p}t+\frac{a^{2}-b^{2}d}{p^{2}}\right)
 \operatorname{mod}p & \text{if }v_{p}\!\left(  a\right)  =v_{p}\!\left(
b\right)  =1, v_{p}\!\left(  d\right)  =0, v_{p}\!\left(  b^{2}%
d-a^{2}\right)  =2, p\geq5,\\
t\left(  t^{2}-\frac{db^{2}}{p^{2}}\right)  \operatorname{mod}p & \text{if
}v_{p}\!\left(  a\right)  >1,v_{p}\!\left(  b\right)  =1, v_{p}\!\left(
d\right)  =0, v_{p}\!\left(  b^{2}d-a^{2}\right)  =2,p\geq3,
\end{array}
\right.  
\label{P(t)inC2odd}}
\end{equation*}
has distinct roots in $\overline{\mathbb{F}}_{p}$. By Tate's Algorithm and
Table~\ref{ta:PapTableIandII} we deduce that the Kodaira-N\'{e}ron type at $p$ is
$\mathrm{I}_{0}^{\ast}$ with $f_{p}=2$ for each of these cases. It remains to
compute $c_{p}$. To this end, observe that by Tate's Algorithm it suffices to
determine whether the quadratic factors in $P(t)$
splits in
$\mathbb{F}_{p}$. This is equivalent to determining whether the
discriminant of the quadratic factors is a square or not. If $p\geq5$ with $v_{p}\!\left(  a\right)  =v_{p}\!\left(  b\right)  =1,\ v_{p}\!\left(
d\right)  =0,\ v_{p}\!\left(  b^{2}d-a^{2}\right)  =2$, then the
discriminant of the quadratic factor of $P\!\left(  t\right)
\ \operatorname{mod}p$ is $\frac{b^{2}d}{p^{2}}$. In particular, the quadratic splits in $\mathbb{F}_{p}$
if and only if $\left(  \frac{d}{p}\right)  =1$. This gives the claimed $c_{p}$ when $v_{p}(9b^{2}d-a^{2})=2$. Now suppose that $v_{p}(9b^{2}d-a^{2})\geq3$, so that
$9b^{2}d-a^{2}=p^{3}l$ for some integer $l$. Then $d=\frac{p^{3}l+a^{2}%
}{9b^{2}}\equiv\frac{a^{2}}{9b^{2}}\ \operatorname{mod}p$ and thus $c_{p}=4$. For
the remaining cases, $c_{p}$ is computed in a similar manner.
\begin{equation}{\small
\renewcommand{\arraystretch}{1}\renewcommand{\arraycolsep}{.1cm}%
\begin{array}
[c]{cccccccc}%
\toprule p & F_{C_2,i} & \text{Conditions on }a,b,d & v_{p}\!\left(
a_{1}\right)   & v_{p}\!\left(  a_{2}\right)   & v_{p}\!\left(  a_{3}\right)
& v_{p}\!\left(  a_{4}\right)   & v_{p}\!\left(  a_{6}\right)  \\
\toprule3 & F_{C_2,23} & \multicolumn{1}{l}{v_{p}\!\left(  a\right)
=v_{p}\!\left(  b\right)  =1,\ v_{p}\!\left(  d\right)  =0,} & 1 & 1 & 2 & 2 &
\geq4\\
&  & \multicolumn{1}{l}{v_{p}\!\left(  b^{2}d-a^{2}\right)  =2} &  &  &  &  &
\\\midrule
\geq5 & F_{C_2,24} & \multicolumn{1}{l}{v_{p}\!\left(  a\right)  =v_{p}\!\left(
b\right)  =1,\ v_{p}\!\left(  d\right)  =0,} & \geq1 & 1 & \geq3 & 2 & \geq4\\
&  & \multicolumn{1}{l}{v_{p}\!\left(  b^{2}d-a^{2}\right)  =2} &  &  &  &  & \\\midrule
\neq2 & F_{C_2,24} & \multicolumn{1}{l}{v_{p}\!\left(  a\right)  >1,\ v_{p}%
\!\left(  b\right)  =1,\ v_{p}\!\left(  d\right)  =0} & \geq1 & \geq2 & \geq3 & 2 &
\geq4\\\cmidrule{2-8}
& F_{C_2,25} & \multicolumn{1}{l}{v_{p}\!\left(  a\right)  =1,\ v_{p}\!\left(
d\right)  =1,\ n=2v_{p}\!\left(  b\right)  -1\geq1} & 1 & 1 & n+2 &
n+2 & n+3\\\cmidrule{3-8}
& & \multicolumn{1}{l}{v_{p}\!\left(  a\right)  =1,\ v_{p}\!\left(
d\right)  =0,\ n=2v_{p}\!\left(  b\right)  -2\geq2} & 1 & 1 & n+3 &
n+2 & n+3\\\cmidrule{2-8}
& F_{C_2,26} & \multicolumn{1}{l}{v_{p}\!\left(  a\right)  =v_{p}\!\left(
b\right)  =1,\ v_{p}\!\left(  d\right)  =0,} & \geq2 & 1 & \frac{n+6}{2} &
\frac{n+4}{2} & n+4\\
&  & \multicolumn{1}{l}{n=2v_{p}\!\left(  b^{2}d-a^{2}\right)  -4\geq2} &  &
&  &  & \\
\bottomrule &  &  &  &  &  &  &
\end{array}}
%\vspace{-0.5em}
\label{Tableforc2poddvals}%
\end{equation}

\textbf{Case 2.} Let $a,b,d$ satisfy one of the conditions in
Table~\ref{TableforC2podd} for the Kodaira-N\'{e}ron type $\mathrm{I}_{n}^{\ast}$ with
$n>0$. By \eqref{Tableforc2poddvals} for $F_{C_2,25}$, $F_{C_2,26}$, and Lemma~\ref{LemmaIn}, we have that $E_{C_2}$ has Kodaira-N\'{e}ron type
$\mathrm{I}_{n}^{\ast}$ at $p$ with $f_{p}=2$ for the claimed $n$. We observe that $c_{p}=4$ if $v_{p}\!\left(  a\right)  =v_{p}\!\left(  b\right)
=1,\ v_{p}\!\left(  d\right)  =0$ with $n=2v_{p}\!\left(  b^{2}d-a^{2}\right)
-4\geq2$. Now suppose $v_{p}\!\left(  a\right)  =1,\ v_{p}\!\left(  d\right)
=1,$ and $n=2v_{p}\!\left(  b\right)  -1$. Then the Weierstrass coefficients $a_{j}$ of $F_{C_2,25}$ under these assumptions yield%
\[
t^{2}+a_{3,\frac{n+3}{2}}t-a_{6,n+3}\equiv t^{2}-\frac{ab^{2}d}{p^{n+3}%
}\ \operatorname{mod}p.
\]
This splits in $\mathbb{F}_{p}$ if and only if $\left(  \frac{ab^{2}d/p^{n+3}}%
{p}\right)  =\left(  \frac{ad/p^{2}}{p}\right)  =1$ which gives the claimed
$c_{p}$. Lastly, if $v_{p}\!\left(  a\right)  =1,\ v_{p}\!\left(  d\right)
=0,$ and $n=2v_{p}\!\left(  b\right)  -2\geq2$, then the Weierstrass
coefficients $a_{j}$ of $F_{C_2,25}$ under these assumptions gives%
\begin{equation}
a_{2,1}t^{2}+a_{4,\frac{n+4}{2}}t-a_{6,n+3}\equiv\frac{-a\left(  a+1\right)
}{p}t^{2}+\frac{ab^{2}d}{p^{n+3}}\ \operatorname{mod}p.\label{TamaC2poddIn}%
\end{equation}
This splits in $\mathbb{F}_{p}$ if and only if the $\frac{4a^{2}\left(
a+1\right)  b^{2}d}{p^{n+4}}$ is a square in $\mathbb{F}_{p}$. Since
$\frac{4a^{2}b^{2}}{p^{n+4}}$ is a nonzero square in $\mathbb{F}_{p}$, we deduce that
\eqref{TamaC2poddIn} splits in $\mathbb{F}_{p}$ if and only if $\left(  \frac{\left(  a+1\right)
d}{p}\right)  =\left(  \frac{d}{p}\right)  =1$ which concludes the proof.
\end{proof}

%%%%%%%%%%%%%%%%%%%%%%%%%%%%%%%%%%%%%%%%%%%%

\subsection{Case of \texorpdfstring{$3$}{}-torsion}
\label{Sect:3tors}
Next, we consider the families $E_{C_3}=E_{C_3}(a,b)$ for relatively prime integers $a,b$ and $E_{C_3}^0=E_{C_3}^0(a)$ for a positive cubefree integer $a$. Then by Proposition~\ref{rationalmodels}, $E_{C_3}$ and $E_{C_3}^0$ parameterize all rational elliptic curves that have a $3$-torsion point.

\begin{theorem}
\label{ThmforC30}
The family of elliptic curves $E_{C_{3}^{0}}$ has additive reduction at a prime $p$ if and only if the parameter $a$ of $E_{C_{3}^{0}}$  satisfies one of the conditions listed in Table~\ref{Table for C30}, and the Kodaira-N\'{e}ron type at $p$,
conductor exponent $f_{p}$, and local Tamagawa number $c_{p}$  are given as follows.
\end{theorem}
{\begingroup
\renewcommand{\arraystretch}{1.0}
 \begin{longtable}{ccccc}
	\hline
	$p$ & K-N type & Conditions on $a,b,d$ & $c_{p}$ & $f_{p}$ \\
	\hline

	\endfirsthead
	\hline 
	$p$ & K-N type & Conditions on $a,b,d$ & $c_{p}$ & $f_{p}$ \\
	\hline
	\endhead
	\hline
	
	\multicolumn{3}{r}{\emph{continued on next page}}
	\endfoot
	\hline
	\caption{Local data for $E_{C_{3}^0}(a)$}\label{Table for C30}
	\endlastfoot

$3 $ & $ \mathrm{II} $ & $ v_{3}\!\left(  a\right)  =0\text{ and }a\equiv\pm
1,\pm4\ \operatorname{mod}9 $ & $ 3 $ & $ 1$\\\cmidrule{2-5}
& $\mathrm{III} $ & $ v_{3}\!\left(  a\right)  =0\text{ and }a\equiv\pm
2\ \operatorname{mod}9 $ & $ 2 $ & $ 2$\\\cmidrule{2-5}
& $\mathrm{IV} $ & $ v_{3}\!\left(  a\right)  =1 $ & $ 5 $ & $ 3$\\\cmidrule{2-5}
& $\mathrm{IV}^{\ast} $ & $ v_{3}\!\left(  a\right)  =2 $ & $ 5 $ & $ 3$\\\midrule
$p\neq3 $ & $ \mathrm{IV} $ & $ v_{p}\!\left(  a\right)  =1 $ & $ 2 $ & $ 3$\\\cmidrule{2-5}
& $\mathrm{IV}^{\ast} $ & $ v_{p}\!\left(  a\right)  =2 $ & $ 2 $ & $ 3$
\end{longtable}
\endgroup}
\vspace{-1.2em}
\begin{proof}
Since the invariants associated to $E_{C_3^0}$ are $c_{4}=0$, $c_{6}=-216a^{2}$, and $\Delta=-27a^{4}$, we have that $E_{C_3^0}$ is a global minimal model when $a$ is cubefree. In particular, $E_{C_3^0}$ has additive reduction at a prime $p$ if and only if $p=3$ or $v_{p}(a)=1,2$. We consider three cases: $\left(  i\right)  $ $p=3$ with $v_{3}\!\left(a\right)  =0$,\ $\left(  ii\right)  \ v_{p}\!\left(  a\right)  =1$ for $p$ any prime, and $\left(  iii\right)  $ $v_{p}\!\left(  a\right)  =2$ for $p$ any prime. We start by noting that if $p=3$ with $v_{3}\!\left(  a\right)  =0$,
then 
$  v_{3}\!\left(  c_{4}\right) =\infty ,v_{3}\!\left(  c_{6}\right)=3, \text{ and }v_{3}\!\left(  \Delta\right)=3 $. 

By Table~\ref{ta:PapTableIandII}, $E_{C_3^0}$ has Kodaira-N\'{e}ron type II or $\mathrm{III}$ at $3$.
The claim now follows from the additional condition in
Table~\ref{ta:PapTableIandII} since $\frac{c_{4}}{3}\equiv
0\ \operatorname{mod}9$ and
\[
\left(  \frac{c_{6}}{3^{3}}\right)  ^{2}+2=64a^{4}+2\equiv\left\{
\begin{array}
[c]{ll}%
3,6\ \operatorname{mod}9 & \text{if }a\equiv\pm1,\pm4\ \operatorname{mod}9,\\
0\ \operatorname{mod}9 & \text{if }a\equiv\pm2\ \operatorname{mod}9.
\end{array}
\right.
\]

Next, suppose that $p$ is any prime and that $v_{p}\!\left(  a\right)  >0$.
Since $a$ is cubefree, we note that $v_{p}\!\left(  a\right)  =1,2$. Now
consider the Weierstrass model $F_{C_3^0}$ attained from $E_{C_3^0}  $ via the admissible change of variables $x\longmapsto x+a$ and
$y\longmapsto y+a^{2}x+a^{2}$. Then the Weierstrass coefficients $a_{j}$ of
$F_{C_3^0}$ satisfy%
\begin{equation}
\left(  v_{p}\!\left(  a_{1}\right)  ,v_{p}\!\left(  a_{2}\right)
,v_{p}\!\left(  a_{3}\right)  ,v_{p}\!\left(  a_{4}\right)  ,v_{p}\!\left(
a_{6}\right)  \right)  =\left(  \geq2,\geq v_{p}\!\left(  a\right)
,v_{p}\!\left(  a\right)  ,\geq2v_{p}\!\left(  a\right)  ,4v_{p}\!\left(
a\right)  \right)  . \label{WeierCoefficientsC30}%
\end{equation}

\textbf{Case 1.} Suppose $v_{p}\!\left(  a\right)  =1$. By
(\ref{WeierCoefficientsC30}), $F_{C_3^0}$ satisfies
the first four steps of Tate's Algorithm. Next, let $b_{j}$ denote the
quantities associated to the model $F_{C_3^0}$ and observe that $\left(
v_{p}\!\left(  b_{6}\right)  ,v_{p}\!\left(  b_{8}\right)  \right)  =\left(
2,\geq3\right)  $. By Tate's Algorithm, we conclude that $E_{C_3^0}$ has Kodaira-N\'{e}ron
type IV at $p$ with $f_{p}=v_{p}\!\left(  \Delta\right)  -2=2+3v_{p}\!\left(
3\right)  $. Moreover, $c_{p}=3$ since $t^{2}+a_{3,1}t-a_{6,2}\equiv t\left(
t+a_{3,1}\right)  \ \operatorname{mod}p$.

\textbf{Case 2.} Suppose\textbf{ }$v_{p}\!\left(  a\right)  =2$. By
\eqref{WeierCoefficientsC30}, $F_{C_3^0}$ satisfies the first six steps of Tate's Algorithm and $t^{3}+a_{2,1}t^{2}%
+a_{4,2}t+a_{6,3}\equiv t^{3}\ \operatorname{mod}p$. It follows that $E_{C_3^0}$ has Kodaira-N\'{e}ron type
$\mathrm{IV}^{\ast}$ at $p$ with $f_{p}=v_{p}\!\left(  \Delta\right)  -6=2+3v_{p}%
\!\left(  3\right)  $. Lastly, $c_{p}=3$ since $t^{2}+a_{3,2}t-a_{6,4}\equiv t\left(
t+a_{3,2}\right)  \ \operatorname{mod}p$ has distinct roots in $\mathbb{F}
_{p}$.
\end{proof}

Next, we consider the family $E_{C_{3}}\!\left(  a,b\right)$ with $\gcd(a,b)=1$.

\begin{theorem}
\label{ThmforC3}
For the family of elliptic curves $E_{C_3}=E_{C_3}(a,b)$, write $a=c^{3}d^{2}e$ where $d$ and $e$ are positive relatively prime squarefree integers. Then $E_{C_3}$ has additive reduction at a prime $p$ if and only if the parameters of $E_{C_3}$ satisfy one of the conditions listed in Table~\ref{TableforC3}, and the Kodaira-N\'{e}ron type at $p$,
conductor exponent $f_{p}$, and local Tamagawa number $c_{p}$ are given as follows.
\end{theorem}
{\begingroup
\renewcommand{\arraystretch}{1.1}
 \begin{longtable}{ccccc}
	\hline
	$p$ & K-N type & Conditions on $a,b$ & $f_{p}$ & $c_{p}$ \\
	\hline

	\endfirsthead

	\hline
	$p$ & K-N type & Conditions on $a,b,d$ & $f_{p}$ & $c_{p}$ \\
	\hline
	\endhead
	\hline

	\multicolumn{3}{r}{\emph{continued on next page}}
	\endfoot
	\hline
	\caption{Local data for $E_{C_3}(a,b)$}\label{TableforC3}
	\endlastfoot

$\neq3 $ & $ \mathrm{IV} $ & $ v_{p}\!\left(  a\right)  \equiv2\ \operatorname{mod}3 $ & $
2 $ & $ 3$\\\cmidrule{2-5}
& $\mathrm{IV}^{\ast} $ & $ v_{p}\!\left(  a\right)  \equiv1\ \operatorname{mod}3 $ & $
2 $ & $ 3$\\\midrule
$3 $ & $ \mathrm{II} $ & $ v_{3}\!\left(  a-27b\right)  =4 $ & $ 4 $ & $ 1$\\\cmidrule{3-5}
&  & $v_{3}\!\left(  a\right)  \equiv0\ \operatorname{mod}3,v_{3}\!\left(
a-27b\right)  =3, $ & $ 3 $ & $ 1$\\
&  & $bd^{2}e^{3}\left(  b^{3}d^{2}e^{5}-c\right)  \not \equiv
-2\ \operatorname{mod}9 $ &   & \\\cmidrule{2-5}
& $\mathrm{III} $ & $ v_{3}\!\left(  a\right)  \equiv0\ \operatorname{mod}%
3,v_{3}\!\left(  a-27b\right)  =3, $ & $ 2 $ & $ 2$\\
&  & $bd^{2}e^{3}\left(  b^{3}d^{2}e^{5}-c\right)  \equiv-2\ \operatorname{mod}%
9 $ &  & \\\cmidrule{2-5}
& $\mathrm{IV} $ & $ v_{3}\!\left(  a\right)  =2 $ & $ 4 $ & $ 3$\\\cmidrule{3-5}
&  & $v_{3}\!\left(  a\right)  \equiv2\ \operatorname{mod}3,\text{ }%
v_{3}\!\left(  a\right)  \neq2 $ & $ 5 $ & $ 3$\\\cmidrule{3-5}
&  & $v_{3}\!\left(  a-27b\right)  =5 $ & $ 3 $ & \multicolumn{1}{l}{$1\text{ if
}a^{2}-27ab\not \equiv 3^{8}\ \operatorname{mod}3^{9}$}\\
&  &  &  & \multicolumn{1}{l}{$3\text{ if }a^{2}-27ab\equiv3^{8}%
\ \operatorname{mod}3^{9}$}\\\cmidrule{2-5}
& $\mathrm{I}_{0}^{\ast} $ & $ v_{3}\!\left(  a-27b\right)  =6 $ & $ 2$ &
\multicolumn{1}{l}{$1\text{ if }a^{2}-27ab\equiv3^{9}\ \operatorname{mod}%
3^{10}$}\\
&  &  &  & \multicolumn{1}{l}{$2\text{ if }a^{2}-27ab\not \equiv
3^{9}\ \operatorname{mod}3^{10}$}\\\cmidrule{2-5}
& $\mathrm{I}_{n}^{\ast} $ & $ n=v_{3}\!\left(  a-27b\right)  -6\geq1 $ & $ 2$ &
\multicolumn{1}{l}{$2\text{ if }a^{2}-27ab\not \equiv 3^{9+n}%
\ \operatorname{mod}3^{10+n}$}\\
&  &  &  & \multicolumn{1}{l}{$4\text{ if }a^{2}-27ab\equiv3^{9+n}%
\ \operatorname{mod}3^{10+n}$}\\\cmidrule{2-5}
& $\mathrm{IV}^{\ast} $ & $ v_{3}\!\left(  a\right)  =1 $ & $ 3 $ & $ 3$\\\cmidrule{3-5}
&  & $v_{3}\!\left(  a\right)  \equiv1\ \operatorname{mod}3,\text{ }%
v_{3}\!\left(  a\right)  \neq1 $ & $ 5 $ & $ 3$	
\end{longtable}
\endgroup}

\begin{proof}
By \cite[Theorem 7.1]{Barrios2020}, $E_{C_3}$ has additive
reduction at a prime $p$ if and only if $(i)$ $p$ divides $de$ or $(ii)$ $p=3$ and $v_{3}\!\left(  a\right)>0$. This is equivalent to the conditions listed in Table~\ref{TableforC3}.
Using Lemma~\ref{Lemma for minimal disc} we get that the minimal
discriminant $\Delta$ and the invariants $c_{4}$, $c_{6}$
associated to a minimal model of $E_{C_3}$ over $\mathbb{Q}_p$ are%
\[
c_{4}=cd^{2}e^{3}\left(  a-24b\right),  \qquad c_{6}=-d^{2}e^{4}\left(
a^{2}-36ab+216b^{2}\right), \qquad\Delta=d^{4}e^{8}b^{3}\left(  a-27b\right)
.
\]
Next, let $u=c^{2}d$ and let $F_{C_3,1}$ and $F_{C_3,2}$ be the models attained via the
admissible change of variables $x\longmapsto u^2x+ u^2p^2$,
$y\longmapsto u^3y$ and
$x\longmapsto u^2x-\frac{u^3de^2}{9}$, $y\longmapsto u^3y-u^3cde+\frac{u^4cd^2e^3}{27}$, respectively. Moreover, the Weierstrass coefficients  $a_{j}$ of
$F_{C_3,i}$ and the quantities $b_{j}$ associated to $F_{C_3,i}$ as in \eqref{basicformulas} have the following valuations:
\begin{equation}
\renewcommand{\arraystretch}{1.02}\renewcommand{\arraycolsep}{.12cm}%
\begin{array}
[c]{ccccccccc}%
\toprule p & F_{C_3,i} & \text{Conditions on }a,b & v_{p}\!\left(  a_{1}\right)
& v_{p}\!\left(  a_{2}\right)   & v_{p}\!\left(  a_{3}\right)   &
v_{p}\!\left(  a_{4}\right)   & v_{p}\!\left(  a_{6}\right)   & v_{p}\!\left(
\Delta\right)  \\
\toprule\geq2 & F_{C_3,1} & v_{p}\!\left(  a\right)  \equiv2\ \operatorname{mod}%
3 & \geq1 & \geq2 & 1 & \geq4 & 6 & 4+v_{p}\!\left(  a-27b\right)
\\\cmidrule{3-9}
&  & v_{p}\!\left(  a\right)  \equiv1\ \operatorname{mod}3 & \geq1 & \geq2 &
2 & \geq4 & 6 & 8+v_{p}\!\left(  a-27b\right)  \\\midrule
3 & F_{C_3,2} & v_{3}\!\left(  a-27b\right)  =5 & 1 & 1 & 2 & 3 & 2 & 5\\
\bottomrule &  &  &  &  &  &  &  &
\end{array}
\label{TableforC3vals}%
\end{equation}

\vspace{-0.5em}
\textbf{Case 1.} Suppose the parameters of $E_{C_3}$ satisfy one of the
conditions in Table~\ref{TableforC3} for the Kodaira-N\'{e}ron type II or $\mathrm{III}$.
If $v_{3}\!\left(  a-27b\right)  =4$, then 
$v_{3}\!\left(c_{4}\right)=2  ,v_{3}\!\left(  c_{6}\right)=3, \text{ and }v_{3}\!\left(  \Delta\right)=4$.
By Table~\ref{ta:PapTableIandII} we
conclude that $E_{C_3}$ has Kodaira-N\'{e}ron type II at $3$ with $(c_{3},f_{3})=(1,4)$. Now suppose $v_{3}\!\left(  a\right)  \equiv0\ \operatorname{mod}%
3$ and $v_{3}\!\left(  a-27b\right)  =3$. Then 
$ v_{3}\!\left(  c_{4}\right)\geq 2,v_{3}\!\left(  c_{6}\right)=3, \text{ and }v_{3}\!\left(  \Delta\right)=3$. Thus the Kodaira-N\'{e}ron type at $3$ is either II or $\mathrm{III}$ by
Table~\ref{ta:PapTableIandII}. The claim now follows from the additional
condition in Table~\ref{ta:PapTableIandII} since%
\[
\left(  \frac{c_{6}}{27}\right)  ^{2}-\frac{c_{4}}{3}+2\equiv bd^{2}%
e^{3}\left(  b^{3}d^{2}e^{5}-c\right)  +2\ \operatorname{mod}9.
\]

\textbf{Case 2.} Suppose the parameters of $E_{C_3}$ satisfy one of the
conditions appearing in Table~\ref{TableforC3} for the Kodaira-N\'{e}ron type IV or
$\mathrm{IV}^{\ast}$. 

\qquad\textbf{Subcase 2a.} Let $p$ be any prime and $v_{p}\!\left(
a\right)  \equiv2\ \operatorname{mod}3$. Then, the Weierstrass model $F_{C_3,1}$ satisfies $v_{3}\!\left(
b_{8}\right)  \geq3$ and $v_{3}\!\left(  b_{6}\right)  =2$. Consequently, $E_{C_3}$ has Kodaira-N\'{e}ron type $\mathrm{IV}$ at $p$ and $f_{p}=v_{p}\!\left(  \Delta\right) -2$ by
(\ref{TableforC3vals}) and Tate's Algorithm. Lastly, $c_{p}=3$ since $t^{2}+a_{3,1}t-a_{6,2}\equiv t\left(  t+a_{3,1}%
\right)  \ \operatorname{mod}3$ by \eqref{TableforC3vals}.

\qquad\textbf{Subcase 2b}. Let $p$ be any prime and $v_{p}\!\left(
a\right)  \equiv1\ \operatorname{mod}3$. By \eqref{TableforC3vals}, $F_{C_3,1}$ satisfies the first six steps of Tate's Algorithm. In fact, it proceeds to Step 8 since $t^{3}+a_{2,1}t^{2}%
+a_{4,2}t+a_{6,3}\equiv t^{3}\ \operatorname{mod} p$. By \eqref{TableforC3vals}, $t^{2}+a_{3,2}t-a_{6,4}\equiv t\left(  t+a_{3,2}\right)
\ \operatorname{mod}3$ has distinct roots in $\mathbb{F}_{p}$. Hence $E_{C_3}$ has Kodaira-N\'{e}ron type $\mathrm{IV}^{\ast}$ at $p$ with
$c_{p}=3$. The claimed conductor exponent $f_{p}$ in Table~\ref{TableforC3} is now
verified since $f_{p}=v_{p}\!\left(  \Delta\right)-6$.

\qquad\textbf{Subcase 2c.} Suppose $v_{3}\!\left(  a-27b\right)  =5$. Then, the Weierstrass model $F_{C_3,1}$ satisfies $v_{3}\!\left(b_{8}\right)  \geq3$ and $v_{3}\!\left(  b_{6}\right)  =2$, and by Tate's Algorithm, $E_{C_3}$ has Kodaira-N\'{e}ron type
IV at $p$ with $f_{3}=3$. Next, observe that
$t^{2}+a_{3,1}t-a_{6,2}\equiv t^{2}-3^{-8}
a\left(  a-27b\right) \ \operatorname{mod}3$. In
particular, this polynomial splits in $\mathbb{F}_{3}$ if and only if
$3^{-8}
a\left(  a-27b\right)\equiv1\ \operatorname{mod}3$ which gives the claimed~$c_3$.

\textbf{Case 3.} Suppose $v_{3}\!\left(  a-27b\right)  \geq6$ and set
$n=v_{3}\!\left(  a-27b\right)  -6$. Then $v_{3}\!\left(  a\right)  =3$ and
thus $v_{3}\!\left(  \Delta\right)  =n+6$. Next, let $a_{j}$ denote the
Weierstrass coefficients of $F_{C_3,2}$. Then%
\begin{equation}
\left(  v_{3}\!\left(  a_{1}\right)  ,v_{3}\!\left(  a_{2}\right)
,v_{3}\!\left(  a_{3}\right)  ,a_{3}\!\left(  a_{4}\right)  ,v_{3}\!\left(
a_{6}\right)  \right)  =\left(  1,1,n+3,n+4,n+3\right)  . \label{valsC3In}%
\end{equation}

When $n=0$, $E_{C_3}$ satisfies the first six steps of Tate's Algorithm. Since $t^{3}+a_{2,1}t^{2}+a_{4,2}%
+a_{6,3}\equiv t^{3}-t^{2}+3^{-9}a\left(
a-27b\right)\ \operatorname{mod}3$ has distinct roots
over $\overline{\mathbb{F}}_{3}$, $E_{C_3}$ has Kodaira-N\'{e}ron type
$\mathrm{I}_{0}^{\ast}$ at $3$ with $f_{3}=2$. For $c_{3}$, 
we note that
$t^{3}-t^{2}+3^{-9}a\left(
a-27b\right)\ \operatorname{mod}3$ has a root in $\mathbb{F}_{3}$ if
and only if $3^{-9}a\left(
a-27b\right)\equiv2\ \operatorname{mod}3$. In fact, it has exactly one root
in $\mathbb{F}_{3}$ in this case. This gives the claimed~$c_3$.

When $n>0$, by Lemma~\ref{LemmaIn} we conclude that $E_{C_3}$ has Kodaira-N\'{e}ron type $\mathrm{I}_{n}^{\ast}$ at $3$ with $f_{3}=2$. It remains to show the claimed local Tamagawa number $c_{3}$. We observe that
\[%
\begin{array}
[c]{rll}%
t^{2}+a_{3,\frac{n+3}{2}}t-a_{6,n+3} & \equiv t^{2}-3^{-\left(  n+9\right)  } a\left(  a-27b\right)%
\ \operatorname{mod}3 & \text{if }n\text{ is odd,}\\
a_{2,1}t^{2}+a_{4,\frac{n+4}{2}}t+a_{6,n+3} & \equiv-t^{2}+3^{-\left(  n+9\right)  } a\left(  a-27b\right)%
\ \operatorname{mod}3 & \text{if }n\text{ is even.}%
\end{array}
\]
In both instances, the polynomial modulo $3$ splits in
$\mathbb{F}_{3}$ if and only if $3^{-\left(  n+9\right)  } a\left(  a-27b\right)\equiv1\ \operatorname{mod}3$. The claimed $c_{3}$ now follows.
\end{proof}

%%%%%%%%%%%%% C4 %%%%%%%%%%%%%%%%%%%

\subsection{Case of \texorpdfstring{$4$}{}-torsion}
We now consider the family of elliptic curves $E_{C_4}=E_{C_4}(a,b)$, where $a,b$ are relatively prime integers with $a$ positive. Then by Proposition~\ref{rationalmodels}, $E_{C_4}$ parameterizes all rational elliptic curves that have a $4$-torsion point.

\begin{theorem}
\label{ThmforC4}
For the family of elliptic curves $E_{C_4}=E_{C_4}(a,b)$, write $a=c^{2}d$ where $d$ is a nonzero squarefree integer. Then $E_{C_4}$ has additive reduction at a prime $p$ if and only if  the parameters of $E_{C_4}$ satisfy one of the conditions listed in Table~\ref{TableforC4}, and the Kodaira-N\'{e}ron type at $p$,
conductor exponent $f_{p}$, and local Tamagawa number $c_{p}$ are given as follows.
\end{theorem}
{\begingroup
\renewcommand{\arraystretch}{1.1}
 \begin{longtable}{ccccc}
	\hline
	$p$ & K-N type & Conditions on $a,b$ & $f_{p}$ & $c_{p}$ \\
	\hline

	\endfirsthead
	\hline
	$p$ & K-N type & Conditions on $a,b$ & $f_{p}$ & $c_{p}$ \\
	\hline
	\endhead
	\hline

	\multicolumn{3}{r}{\emph{continued on next page}}
	\endfoot
	\hline
	\caption{Local data for $E_{C_4}(a,b)$}\label{TableforC4}
	\endlastfoot

$>2 $ & $ \mathrm{I}_{n}^{\ast} $ & $ n=v_{p}\!\left(  a\right)  \text{ is
odd} $ & $ 2 $ & $ 4$\\\midrule
$2 $ & $ \mathrm{III} $ & $ v_{2}\!\left(  a\right)  =2 $ & $ 3 $ & $ 2$\\\cmidrule{2-5}
& $\mathrm{I}_{0}^{\ast} $ & $ v_{2}\!\left(  a+16b\right)  =5 $ & $ 5 $ & $ 2$\\\cmidrule{2-5}
& $\mathrm{I}_{n}^{\ast} $ & $ n=v_{2}\!\left(  a\right)  =1 $ & $ 3 $ & $ 4$\\\cmidrule{3-5}
&  & $n=v_{2}\!\left(  a\right)  =3 $ & $ 5 $ & $ 4$\\\cmidrule{3-5}
&  & $n=v_{2}\!\left(  a\right)  \geq5\text{ odd} $ & $ 6 $ & $ 4$\\\cmidrule{3-5}
&  & $n=v_{2}\!\left(  a\right)  -4\geq2\text{ is even} $ & $ 4 $ & $ 4$\\
&  & \ \ \ $\text{ and }bd\equiv1\ \operatorname{mod}4$ &  & \\\cmidrule{3-5}
&  & $n=v_{2}\!\left(  a+16b\right)  -4=3 $ & $ 4 $& \multicolumn{1}{l}{$2\text{ if
}d\left(  a+16b\right)  \equiv2^{7}\ \operatorname{mod}2^{9}$}\\
&  &  &  & \multicolumn{1}{l}{$4\text{ if }d\left(  a+16b\right)  \equiv
3\cdot2^{7}\ \operatorname{mod}2^{9}$}\\\cmidrule{3-5}
&  & $n=v_{2}\!\left(  a+16b\right)  -4\geq2 $ & $ 4 $ & \multicolumn{1}{l}{$2\text{
if }d\left(  a+16b\right)  \equiv3\cdot2^{n+4}\ \operatorname{mod}2^{n+6}$}\\
&  & $\text{with }n\neq3$ &  & \multicolumn{1}{l}{$4\text{ if }d\left(
a+16b\right)  \equiv2^{n+4}\ \operatorname{mod}2^{n+6}$}\\\cmidrule{2-5}
& $\mathrm{III}^{\ast}$ & $v_{2}\!\left(  a\right)  =6,\ bd\equiv
3\ \operatorname{mod}4$ & $3$ & $2	$
\end{longtable}
\endgroup}
\vspace{-1em}
\begin{proof}
By \cite[Theorem 7.1]{Barrios2020}, $E_{C_4}$ has additive reduction at a
prime $p$ if and only if $\left(  i\right)  $ $p$ is an odd prime dividing $d$, $\left(  ii\right)  $ $p=2$ and $v_{2}(a)\geq1$ is odd,  $\left(  iii\right)$ $p=2$ and $v_{2}(a)=2,4,6$, or $\left(  iv\right)$ $p=2$ and $v_2(a)\geq8$ is even with $bd \equiv 1 \operatorname{mod}4$. This is equivalent to the conditions listed in
Table~\ref{TableforC4}. By Lemma~\ref{Lemma for minimal disc}, the minimal discriminant $\Delta$ 
and the invariants $c_{4}$, $c_{6}$ associated to a minimal model of
$E_{C_4}$ over $\mathbb{Q}_p$ are%
\begin{align*}
v_{p}\!\left(  c_{4}\right)   &  =2v_{p}\!\left(  d\right)  +v_{p}\!\left(
a^{2}+16ab+16b^{2}\right), \\
v_{p}\!\left(  c_{6}\right)   &  =3v_{p}\!\left(  d\right)  +v_{p}\!\left(
a+8b\right)  +v_{p}\!\left(  -a^{2}-16ab+8b^{2}\right), \\
v_{p}\!\left(  \Delta\right)   &  =v_{p}\!\left(  a\right)  +6v_{p}\!\left(
d\right)  +v_{p}\!\left(  a+16b\right)  .
\end{align*}
Throughout this proof we consider $\mathbb{Q}$-isomorphic elliptic curves
$F_{C_4,i}$ attained from $E_{C_4}$ via the admissible change of variables
$x\longmapsto c^{2}x+r_{i}$ and $y\longmapsto c^{3}y+c^{2}s_{i}x+w_{i}$
where $r_{i},s_{i},w_{i}$ are defined in \eqref{admissible change c4} and
the valuations of the Weierstrass coefficients $a_{j}$ of $F_{C_4,i}$ are given in~\eqref{TableforC4vals}.
\begin{equation}\small{
\renewcommand{\arraystretch}{1.2}
\renewcommand{\arraycolsep}{.17cm}%%
\begin{array}
[c]{ccccc}\toprule
i & r_{i} & s_{i} & w_{i} & \text{Notes}\\\toprule
1  & 0 & 0 & c^4d^{4} & \\\midrule
2 & \frac{c^{4}d^{2}}{128l}\left(  dm-16l\right)  & \frac{-cd}{128l}%
(c^{2}d^{2}m+ & \frac{-c^{4}d^{2}}{256l}(c^{2}d^{2}m-2c^{3}d^{2}- &
l=2^{\left(  v_{2}\left(  c^{2}d+16b\right)  -3\right)  /2},\\
&   & 16l\left(  m-8b+4c\right)  ) & 16c^{2}dl-32bcd-128bl) & m=c^{2}%
d+16b\\\midrule
3  & \frac{c^{4}d^{2}}{64}(  \frac{d^{4}c^{2}m^{2}}{256l^{2}}
& \frac{-cd}{16}(16b+8c+ & -\frac{c^{4}d^{2}}{128}(-4c^{2}d-4m+ &
l=2^{v_{2}\left(  c^{2}d+16b\right)  /2-2},\\
& -8)  & 2c^{2}d+\frac{c^{2}d^{2}m}{8l}) & \frac{c^{4}d^{5}m^{2}}{256l^{2}%
}-\frac{cdm}{l}) & m=c^{2}d+16b\\\midrule
4  & c^{3} & cb & c^{5} & \\
\bottomrule &  &  &  &  
\end{array}
\label{admissible change c4}}
\end{equation}

\textbf{Case 1.} Suppose that the parameters of $E_{C_4}$ satisfy one of
the conditions in Table~\ref{TableforC4} for the Kodaira-N\'{e}ron type $\mathrm{III}$ or
$\mathrm{III}^{\ast}$. If $v_{2}(d)=0,\,v_{2}(a)=2$, then 
$v_{2}\!\left(  c_{4}\right)=5  ,v_{2}\!\left(  c_{6}\right)=5, \text{ and }v_{2}\!\left(\Delta\right)=4$ and by Table~\ref{ta:PapTableIV},
$E_{C_4}$ has Kodaira-N\'{e}ron type $\mathrm{II}$ or $\mathrm{III}$ at $2$. By
\eqref{TableforC4vals}, the model $F_{C_4,1}$ satisfies the
first four steps of Tate's Algorithm. Since $v_{2}\!\left(  a_{6}\right)  =2$,
we conclude $E_{C_4}$ has Kodaira-N\'{e}ron type $\mathrm{III}$ at $2$ with $\left(
c_{2},f_{2}\right)  =\left(  2,3\right)  $. Next, suppose $v_{2}\!\left(
a\right)  =6$ with $bd\equiv3\ \operatorname{mod}4$. Then 
$v_{2}(c_{4})=4,v_{2}(c_{6})=6,\text{ and }v_{2}(\Delta)=10$
and by
Table~\ref{ta:PapTableIV} the Kodaira-N\'{e}ron type of $E_{C_4}$ at $2$ is
$\mathrm{I}_{2}^{\ast}$ or $\mathrm{III}^{\ast}$. By
\eqref{TableforC4vals}, the model $F_{C_4,4}$ satisfies the
first six steps of Tate's Algorithm. Since $t^{3}+a_{2,1}t^{2}+a_{4,2}%
t+a_{6,3}\equiv t^{3}\operatorname{mod}2$, Tate's Algorithm
proceeds to Step 8 and thus $E_{C_4}$ has Kodaira-N\'{e}ron type $\mathrm{III}^{\ast}$
at $2$ with $\left(  c_{2},f_{2}\right)  =\left(  2,3\right)  $.
\begin{equation}{
\scalebox{0.97}{
\renewcommand{\arraystretch}{1.02}
\renewcommand{\arraycolsep}{.14cm}%
$\begin{array}
[c]{cccccccc}%
\toprule p & F_{C_4,i} & \text{Conditions on }a,b & v_{p}\!\left(  a_{1}\right)
& v_{p}\!\left(  a_{2}\right)   & v_{p}\!\left(  a_{3}\right)   &
v_{p}\!\left(  a_{4}\right)   & v_{p}\!\left(  a_{6}\right)  \\
\toprule\geq2 & F_{C_4,1} & n=v_{p}\!\left(  a\right)  \text{ is odd} & \geq1 &
1 & \frac{n+3}{2} & n+4 & n+5\\\midrule
2 & F_{C_4,1} & v_{2}\!\left(  a\right)  =2 & 1 & 0 & 1 & 2 & \geq3\\\cmidrule{2-8}
& F_{C_4,2} & v_{2}\!\left(  a+16b\right)  =5 & 1 & 1 & 2 & 2 & \geq
4\\\cmidrule{3-8}
&  & n=v_{2}\!\left(  a+16b\right)  -4\geq3\text{ is odd} & 1 & 1 & \frac
{n+3}{2} & n+1 & \geq n+3\\\cmidrule{2-3}\cmidrule{2-8}
& F_{C_4,3} & n=v_{2}\!\left(  a+16b\right)  -4\geq2\text{ is even} & 1 & 1 &
\frac{n+4}{2} & \frac{n+4}{2} & \geq n+3\\
\cmidrule{2-8}
& F_{C_4,4} & n=v_{2}\!\left(  a\right)  -4\geq2\text{ is even}, & 1 & 1 & \frac{n+4}{2} & \frac{n+4}{2} &
n+4\\

&  &   bd\equiv
1\ \operatorname{mod}4  & & &  &  \\\cmidrule{3-8}
&  & v_{2}\!\left(  a\right)  =6,\ bd\equiv3\ \operatorname{mod}4 & 1 &
\geq2 & 3 & 3 & 6\\
\bottomrule &  &  &  &  &  &  &
\end{array}$}
\vspace{-0.5em}
\label{TableforC4vals}}
\end{equation}

\textbf{Case 2.} Suppose $v_{2}\!\left(  a+16b\right)  =5$ so that
$v_{2}\!\left(  \Delta\right)  =9$. By (\ref{TableforC4vals}), the
model $F_{C_4,2}$ satisfies the first six steps of Tate's Algorithm and
$t^{3}+a_{2,1}t^{2}+a_{4,2}t+a_{6,3}\equiv t\left(  t^{2}+t+1\right)
\ \operatorname{mod}2$. Thus $E_{C_4}$ has Kodaira-N\'{e}ron type $\mathrm{I}%
_{0}^{\ast}$ at $2$ with $\left(  c_{2},f_{2}\right)  =\left(  2,5\right)  $.

\textbf{Case 3.} Suppose that the parameters of $E_{C_4}$ satisfy one of
the conditions in Table~\ref{TableforC4} for the Kodaira-N\'{e}ron type $\mathrm{I}%
_{n}^{\ast}$ with $n>0$. Applying Lemma~\ref{LemmaIn} to the corresponding models $F_{C_4,i}$, we have by  (\ref{TableforC4vals})
that $E_{C_4}$ has Kodaira-N\'{e}ron type $\mathrm{I}_{n}^{\ast}$ at $p$ for
the claimed $n$. Next, observe that $v_{2}\!\left(  \Delta\right)  =n+8$
when $v_{2}\!\left(  a\right)  $ is even and $v_{p}\!\left(  \Delta\right)
=n+6+v_{p}\!\left(  a+16b\right)  $ when $n=v_{p}\!\left(  a\right)  $ is odd.
This gives the claimed conductor exponent $f_{p}$. For the local Tamagawa
number $c_{p}$, it follows from (\ref{TableforC4vals}) and Lemma~\ref{LemmaIn}
that $c_{p}=4$ whenever $v_{p}\!\left(  a\right)  $ odd or $v_{2}\!\left(
a\right)  \geq6$ is even with $bd\equiv1\ \operatorname{mod}4$. It remains to
show the cases corresponding to $n=v_{2}\!\left(  a+16b\right)  -4\geq2$. For
these cases, (\ref{TableforC4vals}) implies
\[
\left.
\begin{array}
[c]{rl}%
t^{2}+a_{3,\frac{n+3}{2}}t+a_{6,n+3} & \text{if }n\text{ is odd}\\
a_{2,1}t^{2}+a_{4,\frac{n+4}{2}}t+a_{6,n+3} & \text{if }n\text{ is even}%
\end{array}
\right\}  \equiv t^{2}+t+a_{6,n+3}\ \operatorname{mod}2.
\]
By Lemma~\ref{LemmaIn}, the local Tamagawa number $c_{2}$ depends
on the parity of $a_{6,n+3}$. We now demonstrate the proof for the case when
$n\geq3$ is odd and note that the case when $n\geq2$ is even follows by a
similar argument. To this end, set $a+16b=8l^{2}k$ where $l$ is a power of $2$
and $k$ is odd. Consequently, $n=2v_{2}\!\left(  l\right)  -1$ and
$v_{2}\!\left(  l\right)  \geq2$. These assumptions imply that $c=4t$ and
$l=4j$ for some integers $t$ and $j$ with $t$ odd. Then the Weierstrass
coefficient $a_{6}$ of $F_{T,2}$ is%
\begin{align*}
&\frac{a_{6}}{d^{4}l^{2}kt^{2}}= 4d^{5}jk^{2}t^{4}-8d^{3}j^{2}k^{2}%
t^{2}-d^{4}kt^{4}-8d^{2}jkt^{2}-d^{2}kt^{2}+16j^{2}k+2dt^{2}  \\
&\equiv\left\{
\begin{array}
[c]{rl}%
4+2\left(  d-k\right)  \ \operatorname{mod}8 & \text{if }j\text{ is odd,}\\
2\left(  d-k\right)  \ \operatorname{mod}8 & \text{if }j\text{ is even,}%
\end{array}
\right. \\
 &  \equiv\left\{
\begin{array}
[c]{cl}%
0\ \operatorname{mod}8 & \text{if }\left(  i\right)  \ kd\equiv
1\ \operatorname{mod}4\text{ and }j\text{ is even or }\left(  ii\right)
\ kd\equiv3\ \operatorname{mod}4\text{ and }j\text{ is odd,}\\
4\ \operatorname{mod}8 & \text{if }\left(  i\right)  \ kd\equiv
1\ \operatorname{mod}4\text{ and }j\text{ is odd or }\left(  ii\right)
\ kd\equiv3\ \operatorname{mod}4\text{ and }j\text{ is even.}%
\end{array}
\right.
\end{align*}
The claimed $c_{2}$ now follows since~$k=\frac
{a+16b}{2^{n+4}}$.
\end{proof}

%%%%%%%%%%%%% C2xC2 %%%%%%%%%%%%%%%%

\subsection{Case of full \texorpdfstring{$2$}{}-torsion} Next, we consider the family of elliptic curves $E_{C_2 \times C_2}=E_{C_2 \times C_2}(a,b)$, where $a,b$ are relatively prime integers with $a$ even and positive. Then by Proposition~\ref{rationalmodels}, $E_{C_2 \times C_2}$ parameterizes all rational elliptic curves that have full $2$-torsion.

\begin{theorem}
\label{ThmC2xC2}
The family of elliptic curves $E_{C_{2} \times C_{2}}$ has additive reduction at a prime $p$ if and only if the parameters of $E_{C_{2} \times C_{2}}$ satisfy one of the conditions listed in Table \ref{Table for C2xC2}, and the Kodaira-N\'{e}ron type,
conductor exponent~$f_{p}$, and local Tamagawa number~$c_{p}$ at $p$ are as follows.
\end{theorem}
{\begingroup
	\renewcommand{\arraystretch}{1}
	\begin{longtable}{ccccc}
		\hline
$p$  & \text{K-N} & \text{Conditions on  }$a,b,d$  &  $f_{p}$  &  $c_{p}$ \\
		  & \text{type} &  &  & \\
		\hline
		
		\endfirsthead
		
		\hline
	$p$  & \text{K-N} & \text{Conditions on  }$a,b,d$  &  $f_{p}$  &  $c_{p}$ \\
		 & \text{type} &  &  & \\
		\hline
		\endhead
		\hline
		
		\multicolumn{5}{r}{\emph{continued on next page}}
		\endfoot
		\hline
		\caption{Local data for $E_{C_{2}\times C_{2}}(a,b,d)$} \label{Table for C2xC2}
		\endlastfoot
$\neq2$ & $\mathrm{I}_{0}^{\ast}$ & $v_{p}(d)=1,\ v_{p}(ab(a-b))=0$ & $2$ &$4$\\\cmidrule{2-5}
&$ \mathrm{I}_{n}^{\ast} $& $v_{p}(d)=1,$ & $2$ & $4$\\
&  & $n=2v_{p}(ab(a-b))\geq2 4$& & \\\midrule
$2$ &$ \mathrm{III} $& $v_{2}(d)=0,\,v_{2}(a)=1$ & $5$ & $2$\\\cmidrule{2-5}
&$ \mathrm{I}_{0}^{\ast}$ &$ v_{2}(a)=2,\,bd\equiv3\operatorname{mod} 4$ & $2$ & $2$\\\cmidrule{2-5}
&$ \mathrm{I}_{1}^{\ast} $& $v_{2}(a)=2,\,bd\equiv1\operatorname{mod} 4$ & $3$ & \multicolumn{1}{l}{$2\text{ if }bd\equiv1\operatorname{mod}8,ad\equiv
	12\operatorname{mod}16\text{ or}$}\\
&  &  &  & \multicolumn{1}{l}{$\hspace{1.8em}bd\equiv5\operatorname{mod}%
	8,ad\equiv4\operatorname{mod}16$}\\\cmidrule{5-5}
&  &  &  & \multicolumn{1}{l}{$4\text{ if }bd\equiv1\operatorname{mod}%
	8,ad\equiv4\operatorname{mod}16\text{ or}$}\\
&  &  &  & \multicolumn{1}{l}{\hspace{1.8em} $bd\equiv5\operatorname{mod}
	8,ad\equiv12\operatorname{mod}16$}\\\cmidrule{2-5}
&$ \mathrm{I}_{n}^{\ast} $& $v_{2}(d)=1,\,n=2v_{2}(a)\geq2$ & $6$ & $4$\\\cmidrule{3-5}
&  & $n=2v_{2}(a)-4\geq2, bd\equiv3\operatorname{mod}4 $& $4$ & $4$\\\cmidrule{2-5}
&$ \mathrm{III}^{\ast} $& $v_{2}(a)=3,\,bd\equiv1\operatorname{mod}4 $&
$3$ & $2$
	\end{longtable}	\endgroup}
\begin{proof}
It follows from \cite[Theorem 7.1]{Barrios2020} that $E_{C_2 \times C_2}$ has
additive reduction at a prime $p$ if and only if $\left(  i\right)  $ $p$ is an
odd prime dividing $d$, $\left(  ii\right)  $ $p=2$ with $v_{2}\!\left(  a\right)
<4$, or $\left(  iii\right)  $ $p=2$ with $v_{2}\!\left(  a\right)  \geq4$ and
$bd\not \equiv1\ \operatorname{mod}4$. An easy verification shows that this is
equivalent to the listed $p$ and corresponding conditions on $a$ and $b$ given
in Table \ref{Table for C2xC2}. If $E_{C_2 \times C_2}$ has additive reduction at $p$, then $E_{C_2 \times C_2}/\mathbb{Q}_{p}$ is a minimal model by Lemma~\ref{Lemma for minimal disc}. Next, we observe that the
minimal discriminant $\Delta$ of $E_{C_2 \times C_2}$ and the invariants $c_{4}$, $c_{6}$
satisfy
\begin{align*}
v_{p}(c_{4}) &  =v_{p}(16)+2v_{p}(d)+v_{p}(a^{2}-ab+b^{2}),\\
v_{p}(c_{6}) &  =v_{p}(32)+v_{p}(a-2b)+v_{p}(a+b)+v_{p}(2a-b)+3v_{p}(d),\\
v_{p}(\Delta) &  =v_{p}(16)+2v_{p}(a)+2v_{p}(b)+2v_{p}(a-b)+6v_{p}(d).
\end{align*}

\noindent Throughout this proof we consider $\mathbb{Q}$-isomorphic elliptic
curves $F_{C_2 \times C_2,i}$ attained from $E_{C_2 \times C_2}$ via the admissible change of variables
$x\longmapsto x+r_{i}$ and $y\longmapsto y+s_{i}x+w_{i}$ where
$r_{i},s_{i},w_{i}$ are as given in \eqref{admissible change c2xc2}. The Weierstrass
coefficients $a_{j}$ of $F_{C_2 \times C_2,i}$ satisfy the valuations given in \eqref{TableforC2XC2vals}.

%\vspace{-0.05in}
\begin{equation}
\renewcommand{\arraystretch}{1.2}
\renewcommand{\arraycolsep}{0.35cm}
\begin{array}
[c]{lllll}%
\toprule i & r_{i} & s_{i} & w_{i}\\
\midrule1  & 4 a^{2} b^{2} d^{2} & b^{2} d & 4 a^{2} b^{2} d^{2}\\
\midrule2  & a^{2}d^{2}-2abd^{2}+b^{2}d^{2}-ad & a^{2}d-2abd+b^{2}d &
a^{2}d^{2}-2abd^{2}+b^{2}d^{2}\\
\midrule3  & 6 & b d & 2b d + 4\\
\bottomrule &  &  &  &
\end{array}
\label{admissible change c2xc2}%
\end{equation}

\vspace{-1em}
\textbf{Case 1.} Suppose that the parameters of $E_{C_2 \times C_2}$ satisfy one of
the conditions in Table \ref{Table for C2xC2} for the Kodaira-N\'{e}ron type $\mathrm{III}$ or
$\mathrm{III}^{\ast}$. If $v_{2}(d)=0,\,v_{2}(a)=1$, then 
$v_{2}\!\left(  c_{4}\right)=4  ,v_{2}\!\left(  c_{6}\right)\geq7, \text{ and }v_{2}\!\left(
\Delta\right)=6$ and by Table~\ref{ta:PapTableIV},
$E_{C_2 \times C_2}$ has Kodaira-N\'{e}ron type $\mathrm{II}$ or $\mathrm{III}$ at $2$. By
\eqref{TableforC2XC2vals}, the model $F_{C_2 \times C_2,1}$ satisfies the
first four steps of Tate's Algorithm. Since $v_{2}\!\left(  a_{6}\right)  =5$,
we conclude that $E_{C_2 \times C_2}$ has Kodaira-N\'{e}ron type $\mathrm{III}$ at $2$ with $\left(
c_{2},f_{2}\right)  =\left(  2,5\right)  $. Next, suppose $v_{2}\!\left(
a\right)  =3$ with $bd\equiv1\ \operatorname{mod}4$. Then 
$  v_{2}(c_{4})=4,v_{2}(c_{6})=6,\text{ and }v_{2}(\Delta)=10$ and by
Table~\ref{ta:PapTableIV} the Kodaira-N\'{e}ron type of $E_{C_2 \times C_2}$ at $2$ is
$\mathrm{I}_{2}^{\ast}$ or $\mathrm{III}^{\ast}$. By
\eqref{TableforC2XC2vals}, the model $F_{C_2 \times C_2,3}$ satisfies the
first six steps of Tate's Algorithm. Since $t^{3}+a_{2,1}t^{2}+a_{4,2}%
t+a_{6,3}\equiv t^{3}\operatorname{mod}2$, Tate's Algorithm
proceeds to Step 8 and thus $E_{C_2 \times C_2}$ has Kodaira-N\'{e}ron type $\mathrm{III}^{\ast}$
at $2$ with $\left(  c_{2},f_{2}\right)  =\left(  2,3\right)  $.

\begin{equation}{
\scalebox{0.92}{
\renewcommand{\arraystretch}{1.1}
\renewcommand{\arraycolsep}{0.07cm}%
$\begin{array}
[c]{ccccccccc}
\toprule
p & F_{C_2 \times C_2,i} & \text{Conditions on }a,b,d & v_{p}\!\left(  a_{1}\right)   &
v_{p}\!\left(  a_{2}\right)   & v_{p}\!\left(  a_{3}\right)   & v_{p}\!\left(
a_{4}\right)   & v_{p}\!\left(  a_{6}\right)   & v_{p}\!\left(  \Delta\right)
\\\toprule
\neq2 & F_{C_2 \times C_2,1} & v_{p}(d)=1,v_{p}(ab(a-b))=0 & 1 & \geq1 & 2 & 2 & \geq4 &
6\\\cmidrule{3-9}
& & v_{p}(d)=1,n=2v_{p}(ab)\geq2 & 1 & 1 & n+2 & \frac
{n+4}{2} & \frac{3n}{2}+4 & n+6\\\cmidrule{2-9}
& F_{C_2 \times C_2,2} & v_{p}(d)=1,n=2v_{p}(a-b)\geq2 & n+1 & 1 & n+2 & \frac{n+4}{2} &
\frac{3n}{2}+4 & n+6\\\midrule
2 & F_{C_2 \times C_2,1} &  v_{2}(d)=0,\,v_{2}(a)=1 & 1 & \geq1 & 5 & 3 &5 & 6 \\\cmidrule{3-9}
&  &  n=2v_{2}(a)\geq2,v_{2}(d)=1 & 2 & 1 & n+5 & \frac{n+4}{2} &\frac{3n}{2} + 6 & n + 10\\\cmidrule{3-9}
&  & n=2v_{2}(a)-4\geq2, bd\equiv3\operatorname{mod}4  & 1 & 1 & n+7 & \frac{n+4}{2} & \frac{3n}{2} + 8 & n+8\\\cmidrule{3-9}
&  & v_{2}(a)=3,\,bd\equiv1\operatorname{mod}4 & 1 & \geq2 & 9 & 3 &11 & 10\\\cmidrule{2-9}
& F_{C_2 \times C_2,3} & v_{2}(a)=2,\,bd\equiv3\operatorname{mod}4  & 1 & \geq2 & 2 & \geq3 & 3 &8 \\\cmidrule{3-9}
&  & v_{2}(a)=2,\,bd\equiv1\operatorname{mod}4 & 1 & 1 & 2 & \geq3 & \geq 4 &8\\
\bottomrule &  & 
\end{array}$}
\label{TableforC2XC2vals}}
\end{equation}
%%%%%%%%
\textbf{Case 2.} Suppose that the parameters of $E_{C_2 \times C_2}$ satisfy one of
the conditions in Table~\ref{TableforC4} for the Kodaira-N\'{e}ron type $\mathrm{I}%
_{0}^{\ast}$. For these two cases, the Weierstrass coefficients $a_{j}$ of the models $F_{C_2 \times C_2,i}$ given in \eqref{TableforC2XC2vals} satisfy the
first six steps of Tate's Algorithm. Next, let $P\!\left(  t\right)
=t^{3}+a_{2,1}t^{2}+a_{4,2}t+a_{6,3}$. Then, one can verify the following:

\[
P\!\left(  t\right)  \equiv\left\{
\begin{array}
[c]{ll}%
t\left(  t^{2}+\frac{(a+b)d}{p}t+\frac{abd^{2}}{p^{2}}\right)
\ \operatorname{mod}p & \text{if }p>2,\ v_{p}\!\left(  d\right)
=1,\ v_{p}\!\left(  ab\left(  a-b\right)  \right)  =0,\\[0.7em]
(t+1)(t^{2}+t+1)\ \operatorname{mod}2 & \text{if }v_{2}\!\left(  a\right)
=2,\ bd\equiv3\ \operatorname{mod}4.
\end{array}
\right.
\]

\vspace{0.5em}
\noindent By Tate's Algorithm, $E_{C_2 \times C_2}$ has Kodaira-N\'{e}ron type $\mathrm{I}%
_{0}^{\ast}$ at $p$ with $f_{p}=v_{p}\!\left(  \Delta\right)  -4$. Moreover,
$c_{2}=2$ if $v_{2}\!\left(  a\right)  =2,\ bd\equiv3\ \operatorname{mod}4$.
If  $p>2,\ v_{p}\!\left(  d\right)  =1,\ v_{p}\!\left(  ab\left(  a-b\right)
\right)  =0$, then the discriminant of $t^{2}+\frac{(a+b)d}{p}t+\frac{abd^{2}%
}{p^{2}}$ is $\frac{(a-b)^{2}d^{2}}{p^{2}}$. In particular, the quadratic
polynomial splits in $\mathbb{F}_{p}$ and thus $c_{p}=4$.

\textbf{Case 3.} Suppose that the parameters of $E_{C_2 \times C_2}$ satisfy one of
the conditions in Table~\ref{TableforC4} for the Kodaira-N\'{e}ron type $\mathrm{I}%
_{n}^{\ast}$ with $n>0$. For these conditions, we consider the models $F_{C_2 \times C_2,i}$ given in \eqref{TableforC2XC2vals}. Then, by Lemma~\ref{LemmaIn}, $E_{C_2 \times C_2}$ has
Kodaira-N\'{e}ron type $\mathrm{I}_{n}^{\ast}$ at $p$ for the claimed $n$. The
conductor exponent now follows from Table~\ref{TableforC4} since $f_{p}%
=v_{p}\!\left(  \Delta\right)  -4-n$. Moreover, by Lemma~\ref{LemmaIn} and
Table~\ref{TableforC4}, we conclude $c_{p}=4$ for each case except possibly when
$v_{2}\!\left(  a\right)  =2$ with $bd\equiv1\ \operatorname{mod}4$. For this
case, we observe that the Weierstrass coefficients of $F_{C_2 \times C_2,3}$ satisfy
\[
t^{2}+a_{3,2}t-a_{6,4}\equiv t^{2}+t+\frac{b^{2}d^{2}-9ad-5bd-50-3abd^{2}/2}{4}\ \operatorname{mod}2.
\]

%%%%%%%

Note that $b^{2}d^{2}-9ad-5bd-50-\frac{3abd^{2}}{2}\equiv3\left(bd+1-\frac{abd^{2}
}{2}\right)  \ \operatorname{mod}8$ 
is either
$0$ or $4$.
The claimed $c_{2}$ now follows since this is congruent
to $0\ \operatorname{mod}8$ if and only if $\left(  i\right)  \ bd\equiv
1\ \operatorname{mod}8,\ ad\equiv4\ \operatorname{mod}16$ or $\left(
ii\right)  $ $bd\equiv5\ \operatorname{mod}8,\ ad\equiv12\ \operatorname{mod}%
16$.
\end{proof}

\subsection{Remaining torsion subgroups}
Lastly, we consider the families of elliptic curves $E_{T}=E_{T}(a,b)$, where $T=C_2\times C_4, C_2 \times C_6, C_m$ for $m\geq 5$ and $a,b$ are relatively prime integers with $a$ positive. Then by Proposition~\ref{rationalmodels}, $E_{T}$ parameterizes all rational elliptic curves such that $T$ is a subgroup of their torsion subgroup.

\begin{theorem}
\label{ThmotherTs} Consider the family of elliptic curves $E_{T}$ where $T=C_2\times C_4, C_2 \times C_6, C_m$ for $m\geq 5$. Then $E_T$ has additive reduction at a prime $p$ if and only if the parameters of $E_{T}$ satisfy one of the conditions listed in Table~\ref{OtherTs}, and the Kodaira-N\'{e}ron type at $p$,
conductor exponent~$f_{p}$, and local Tamagawa number $c_{p}$ are given as follows.
\end{theorem}
{\begingroup
		\renewcommand{\arraystretch}{1}
	\begin{longtable}{cccccc}
			\hline
		$T$  &  $p$  & \text{K-N} & \text{Conditions on  }$a,b$  &  $f_{p}$  &  $c_{p}$ \\
		&  & \text{type} &  &  & \\
		\hline
		
		\endfirsthead
		\hline
		$T$  &  $p$  & \text{K-N} & \text{Conditions on  }$a,b$  &  $f_{p}$  &  $c_{p}$ \\
		&  & \text{type} &  &  & \\
		\hline
		\endhead
		\hline
		
		\multicolumn{6}{r}{\emph{continued on next page}}
		\endfoot
		\hline
		\caption{Local data for $E_T(a,b)$ where $T=C_2\times C_4, C_2 \times C_6, C_m$ for $m\geq 5$} \label{OtherTs}
		\endlastfoot
		$C_{5}$  &  $5$  & $ \mathrm{II} $ & $ v_{5}\!\left(  a+18b\right)  =1$  & $ 2$  &$1$ \\\cmidrule{3-6}
		&  &  $\mathrm{III}  $&  $v_{5}\!\left(  a+18b\right)  \geq2 $ &  $2 $ &  $2$ \\\midrule
		$C_{6} $ &  $2 $ &  $\mathrm{IV}$  &  $v_{2}\!\left(  a+b\right)  =1 $ &  $2 $ &$3$ \\\cmidrule{3-6}
		&  &$  \mathrm{IV}^{\ast} $ & $ v_{2}\!\left(  a+b\right)  =2$  &  $2$  &$3$ \\\cmidrule{2-6}
		& $ 3$  &  $\mathrm{III} $ & $ v_{3}\!\left(  a\right)  =1 $ &  $2$  &  $2$\\\cmidrule{3-6}
		&  & $\mathrm{I}_{0}^{\ast} $ &  $v_{3}\!\left(  a+9b\right)  =2,\ v_{3}\!\left(  a\right)  =2$  &  $2$  &  $2$ \\\cmidrule{3-6}
		&  &  $\mathrm{I}_{n}^{\ast} $ &  $n=v_{3}\!\left(  a+9b\right)  -2\geq
		1,\ v_{3}\!\left(  a\right)  =2 $ &  $2$  & \multicolumn{1}{l}{ $2  \text{ if} ab+9b^{2}\not \equiv 3^{n+2}\ \operatorname{mod}3^{n+3} $}\\\cmidrule{6-6}
		&  &  &  &  & \multicolumn{1}{l}{ $4   \text{ if}  ab+9b^{2}\equiv3^{n+2}	\ \operatorname{mod}3^{n+3}$ }\\\cmidrule{4-6}
		&  &  &  $n=2v_{3}\!\left(  a\right)  -4\geq2,\ v_{3}\!\left(  a\right)  \geq3 $&  $2 $ & $ 4$ \\\midrule
		$C_{7}$  &  $7 $ & $ \mathrm{II}$  &  $v_{7}\!\left(  a+4b\right)  \geq1$  &$ 2$  &$1$ \\\midrule
		$C_{8} $ &  $2$  &  $\mathrm{I}_{n}^{\ast} $ &  $n=2v_{2}\!\left(  a\right)-1\geq3 $ & $ 4$  & $ 4$ \\\midrule
		$C_{9}$  &  $3$  &  $\mathrm{IV} $ & $ v_{3}\!\left(  a+b\right)  \geq1$  &  $3$  &$3$ \\\midrule
		$C_{10} $ & $ 5$  & $ \mathrm{III} $ & $ v_{5}\!\left(  a+b\right)  \geq1$  & $ 2$  &$2$ \\\midrule
		$C_{12}  $&  $3 $ &  $\mathrm{I}_{n}^{\ast} $ &$  n=2v_{3}\!\left(  a\right)-1\geq1$  &  $2$  &  $4$ \\\midrule
		$C_{2}\times C_{4} $ &  $2$  & $ \mathrm{I}_{n}^{\ast}$  & \multicolumn{1}{l}{ $n=1 \text{ if } v_{2}\!\left(  a\right)  =1  \text{ or }  v_{2}\!\left(  a+4b\right)  =3 $} &  $3$  &$4$ \\\cmidrule{4-6}
		&  &\multicolumn{1}{l}{}& \multicolumn{1}{l}{ $n=2v_{2}\!\left(  a+8b\right)
			-6\geq2,\ v_{2}\!\left(  a\right)  =3 $} &  $4$  &  $4$ \\\cmidrule{4-6}
		&  & \multicolumn{1}{l}{} & \multicolumn{1}{l}{ $n=2v_{2}\!\left(  a\right)
			-6\geq2,\ v_{2}\!\left(  a\right)  \geq4 $} &  $4$  &  $4$ \\\midrule
		$C_{2}\times C_{6}  $&  $3$  &  $\mathrm{I}_{n}^{\ast} $ &  $n=2,\ v_{3}\!\left(b\right)  \geq3 $ &  $2$  & $ 4$ \\\cmidrule{4-6}
		&  &  &  $n=2v_{3}\!\left(  b^{2}-9a^{2}\right)  -4,\ v_{3}\!\left(  b\right)=1 $ &  $2$  &  $4$ \\\cmidrule{4-6}
		&  &  &  $n=2v_{3}\!\left(  b-9a\right)  -2,\ v_{3}\!\left(  b\right)  =2 $ & $2$  &  $4$
	\end{longtable}
	\endgroup}

\begin{proof}
It follows from \cite[Theorem 7.1]{Barrios2020} that
$E_{T}$, where $T=C_2\times C_4, C_2 \times C_6, C_m$ for $m\geq 5$, has additive reduction at a prime $p$ if and only if $p$ is a listed prime in Table~\ref{OtherTs} and the parameters of $E_{T}$ satisfy the corresponding conditions in Table~\ref{OtherTs}. In this proof, we will refer to Corollary~\ref{CorC2xC2Instat}%
~and~\ref{CorollaryonConductor} which are consequences of Theorems \ref{ThmC2at2}, \ref{ThmforC2podd}, \ref{ThmforC30}, \ref{ThmforC3}, and \ref{ThmC2xC2}.
Using Lemma~\ref{Lemma for minimal disc}, we compute the $p$-adic valuations
for the invariants $c_{4},c_{6},\Delta$ associated to a minimal model of
$E_{T}/\mathbb{Q}_{p}$ as given in \eqref{valotherTs}.
These valuations along with Table~\ref{ta:PapTableIandII} allow us to conclude that the
theorem holds for $T=C_{5},C_{7},C_{10}$. Below we treat the remaining cases separately.

\begin{equation}%
\renewcommand{\arraystretch}{0.4}
\renewcommand{\arraycolsep}{0.32cm}
\begin{array}
[c]{cccccc}\toprule
 T  &  p  & \text{Conditions on  }a,b  &  v_{p}\!\left(  c_{4}\right)    &
 v_{p}\!\left(  c_{6}\right)    &  v_{p}\!\left(  \Delta\right)   \\\toprule
 C_{5}  &  5  &  v_{5}\!\left(  a+18b\right)  =1  &  1  &  1  &  2 %
\\\cmidrule{2-6}
&  &  v_{5}\!\left(  a+18b\right)  \geq2  &  1  &  \geq2  &  3 \\\midrule
 C_{6}  &  2  &  v_{2}\!\left(  a+b\right)  =1  &  \geq6  &  5  &
 4 \\\cmidrule{3-6}
&  &  v_{2}\!\left(  a+b\right)  =2  &  4  &  6  &  8 \\\cmidrule{2-6}
&  3  &  v_{3}\!\left(  a\right)  =1  &  2  &  3  &  3 \\\midrule
 C_{7}  &  7  &  v_{7}\!\left(  a+4b\right)  \geq1  &  \geq1  &  1  &
 2 \\\midrule
 C_{10}  &  5  &  v_{5}\!\left(  a+b\right)  \geq1  &  1  &  \geq2  &
 3 \\\midrule
 C_{12}  &  3  &  v_{3}\!\left(  a\right)  \geq2  &  2  &  3  &  5+2v_{3}%
\!\left(  a\right)   \\\midrule
 C_{2}\times C_{6}  &  3  &  v_{3}\!\left(  b\right)  =1  &  2  &  3  &
 2+2v_{3}\!\left(  b^{2}-9a^{2}\right)  \geq8 \\\cmidrule{3-6}
&  &  v_{3}\!\left(  b\right)  =2  &  2  &  3  &  4+2v_{3}\!\left(
b-9a\right)  \geq8 \\\cmidrule{3-6}
&  &  v_{3}\!\left(  b\right)  \geq3  &  2  &  3  &  8 \\\bottomrule
\end{array}
\ \ \label{valotherTs}%
\end{equation}

\textbf{Case 1.} The family of elliptic curves $E_{C_6}$ has additive reduction at $p=2$ (resp.\ $p=3$) if and only if $v_{2}\!\left(  a+b\right)  =1,2$ (resp.\ $v_{3}(a)\geq1$). We consider
three subcases.

\qquad\textbf{Subcase 1a.} Suppose $v_{2}\!\left(  a+b\right)  =1,2$. Since
$E_{C_6}$ has nonzero $j$-invariant under these assumptions, there are relatively prime integers $h$
and $k$ such that $E_{C_6}$ is $\mathbb{Q}$-isomorphic to $E_{C_{3}}\!\left(
h,k\right)$ by
Proposition~\ref{rationalmodels}. Thus $\left(  c_{2},f_{2}\right)  =\left(  3,2\right)  $ by
Theorem~\ref{ThmforC3}. It follows from (\ref{valotherTs}) and
Table~\ref{ta:PapTableIV} that $E_{C_6}$ has Kodaira-N\'{e}ron type IV at $2$ if
$v_{2}\!\left(  a+b\right)  =1$ and $E_{C_6}$ has Kodaira-N\'{e}ron type $\mathrm{IV}%
^{\ast}$ at $2$ if $v_{2}\!\left(  a+b\right)  =2$.

\qquad\textbf{Subcase 1b.} Now suppose $v_{3}\!\left(  a\right)  =1$. By
(\ref{valotherTs}) and Table~\ref{ta:PapTableIandII}, $E_{C_6}$ has Kodaira-N\'{e}ron type $\mathrm{II}$ or
$\mathrm{III}$ at $3$. For the additional condition, note
that $\left(  \frac{c_{6}}{27}\right)  ^{2}+2\equiv
\frac{c_{4}}{3}\ \operatorname{mod}9$ and thus $E_{C_6}$ has Kodaira-N\'{e}ron type~$\mathrm{III}$ with $\left(  c_{3},f_{3}\right)  =\left(  2,2\right)  $.

\qquad\textbf{Subcase 1c.} Suppose $v_{3}\!\left(  a\right)  \geq2$. We
consider the models $F_{C_6,1}$ and $F_{C_6,2}$ which are obtained via the
admissible change of variables $x\longmapsto x+ \frac{3ab-a^{2}}{9}$,
$y\longmapsto y+4b^{2}x+\frac{2\left(  a+3b\right)  ^{2}}{27}$ and
$x\longmapsto x+a^{2}$, $y\longmapsto y$, respectively. Moreover, the
Weierstrass coefficients $a_{j}$ of $F_{C_6,i}$ have the following valuations.
\begin{equation}%
\renewcommand{\arraystretch}{0.9}
\renewcommand{\arraycolsep}{0.15cm}
\begin{array}
[c]{cccccccc}\toprule
 F_{C_6,i}  & \text{Conditions on } a  &  v_{p}\!\left(  a_{1}\right)    &
 v_{3}\!\left(  a_{2}\right)    &  v_{p}\!\left(  a_{3}\right)    &
 v_{p}\!\left(  a_{4}\right)    &  v_{p}\!\left(  a_{6}\right)    &
\text{Notes}\\\toprule
 F_{C_6,1}  & \multicolumn{1}{l}{ v_{3}\!\left(  a\right)  =2 } &  1  &  1  &
 n+3  &  n+4  &  n+3  & \multicolumn{1}{l}{ n=v_{3}\!\left(  a+9b\right)  -2 %
}\\\midrule
 F_{C_6,2}  & \multicolumn{1}{l}{ v_{3}\!\left(  a\right)  \geq3 } &  1  &  1  &
 \frac{n+4}{2}  &  \frac{n+4}{2}  &  2n+8  & \multicolumn{1}{l}{ n=2v_{3}
\!\left(  a\right)  -4 }\\\bottomrule
\end{array}
\label{nforC6}%
\end{equation}
By Lemma~\ref{LemmaIn}, $E_{C_6}$ has Kodaira-N\'{e}ron type $\mathrm{I}_{n}^{\ast}$ at $p=3$ with
$f_{3}=2$ where $n\ge 1$ is as given in (\ref{nforC6}). In particular, $c_{3}=4$ by Lemma
\ref{LemmaIn} and (\ref{nforC6}) if $v_{3}\!\left(  a\right)  \geq3$. It
remains to show the case when $v_{3}\!\left(  a\right)  =2$ and $v_{3}%
\!\left(  a+9b\right)  \geq2$. For this case, set $n=v_{3}\!\left(  a+9b\right)-2$ and observe that by (\ref{nforC6}),
the model $F_{T,1}$ satisfies the first six steps of Tate's Algorithm. In
fact,
\begin{equation}
a_{2,1}=\frac{-(a+3b)^{2}}{9}\equiv2\operatorname{mod}3\ \text{ and }\
a_{6,n+3}=\frac{(a+9b)(a+3b)^{3}a^{2}}{3^{n+9}}\equiv\frac{\left(
a+9b\right)  b}{3^{n+2}}\operatorname{mod}3.\label{a2a6forC6}%
\end{equation}
Consequently, if $v_{3}\!\left(  a+9b\right)  =2$, then
$a_{6,3}\equiv1+\frac{ab}{9}\text{ }\operatorname{mod}3=2\text{ }\operatorname{mod}3$ since $a_{6,3}$ and $\frac{ab}{9}$ are not divisible by $3$. Thus $t^{3}+a_{2,1}t^{2}+a_{4,2}t+a_{6,3}\equiv\left(  t+1\right)  \left(  t^{2}+t-1\right)
\ \operatorname{mod}3$ which
has only one root in $\mathbb{F}_{3}$, and hence $c_{3}=2$.

When $n=v_{3}\!\left(  a+9b\right)  -2\geq1$, we use (\ref{nforC6}),
(\ref{a2a6forC6}), and Lemma~\ref{LemmaIn} to see that $c_3$ is determined by whether $t^{2}-a_{6,n+3}$
splits over $\mathbb{F}_{3}$. This is equivalent to $a_{6,n+3} \equiv 1 \operatorname{mod}3$, i.e., $\left(  a+9b\right)
b\equiv3^{n+2}\ \operatorname{mod}3^{n+3}$.

\textbf{Case 2.} The family of elliptic curves $E_{C_8}$ has additive reduction at a
prime $p$ if and only if $p=2$ with $v_{2}\!\left(  a\right)  \geq2$. Next,
let $F_{C_8}$ be the model attained from $E_{C_8}$ via the admissible change of
variables $x\longmapsto x+a^{2}$ and $y\longmapsto y$. Set $n=2v_{2}\!\left(
a\right)  -1$ and observe that the Weierstrass coefficients of $a_{j}$ and the minimal discriminant $\Delta$ satisfy
$
 v_{2}\!\left(  a_{1}\right) = 1 ,v_{2}\!\left(  a_{2}\right)=1
,v_{2}\!\left(  a_{3}\right)=\frac{n+3}{2}  ,v_{2}\!\left(  a_{4}\right)=n+3  ,v_{2}\!\left(
a_{6}\right)=2n+3 ,\text{ and }v_{2}\!\left(
\Delta\right) = n+8$.

By Lemma~\ref{LemmaIn}, we
conclude that $E_{C_8}$ has Kodaira-N\'{e}ron type $\mathrm{I}_{n}^{\ast}$ at $2$ with
$\left(  c_{2},f_{2}\right)  =\left(  4,4\right)  $.

\textbf{Case 3.} The family of elliptic curves $E_{C_9}$ has additive reduction at a prime $p$ if and only if $p=3$ with
$v_{3}\!\left(  a+b\right)  \geq1$. Next, let $F_{C_9}$ be the model attained
from $E_{C_9}$ via the admissible change of variables $x\longmapsto x$ and
$y\longmapsto y+a+b$. Then the Weierstrass coefficients $a_{j}$ and the
quantity $b_{j}$ associated to $F_{C_9}$ satisfy 
$
v_{3}\!\left(  a_{3}\right)\geq 1  ,v_{3}\!\left(  a_{4}\right)\geq 2
,v_{3}\!\left(  a_{6}\right)\geq2  ,v_{3}\!\left(  b_{6}\right)=2  ,v_{3}\!\left(
b_{8}\right)=3,\text{ and }v_{3}\!\left(  \Delta\right)=5$.

In particular, $F_{C_9}$ satisfies the first five steps of Tate's Algorithm and
we conclude that $E_{C_9}$ has Kodaira-N\'{e}ron type~$\mathrm{IV}$ at $3$ with
$f_{3}=3$. Lastly $c_{3}=3$, beacuse of
\[
t^{2}+a_{3,1}t-a_{6,2}\equiv\left(  t+\frac{a+b}{3}\right)  \left(
t+\frac{a+b}{3}+b\right)  \ \operatorname{mod}3.
\]

\textbf{Case 4.} The family of elliptic curves $E_{C_{12}}$ has additive reduction at a prime $p$ if and
only if $p=3$ with $v_{3}\!\left(  a\right)  \geq1$. Next, set $n=2v_{3}%
\!\left(  a\right)  -1\geq1$. By (\ref{valotherTs}) and
Table~\ref{ta:PapTableIandII}, $E_{C_{12}}$ has Kodaira-N\'{e}ron type
$\mathrm{I}_{n}^{\ast}$ at $3$ with $f_{3}=2$. Since $E_{C_{12}}$ has a $4$ torsion
point, there are relatively
prime integers $n$ and $m$ such that $E_{C_{12}}$ is $\mathbb{Q}
$-isomorphic to $E_{C_{4}}\!\left(  n,m\right)  $ by Proposition~\ref{rationalmodels}. By Theorem~\ref{ThmforC4}
we obtain $c_{3}=4$.

\begin{equation}%
\renewcommand{\arraystretch}{0.1}
\renewcommand{\arraycolsep}{0.8cm}
\begin{array}
[c]{ccccc}\toprule%
 i  &  u_{i}  &  r_{i}  &  s_{i}  &  w_{i} \\\toprule%
 1  &  1  &  a^{2}  &  0  &  0 \\\cmidrule{1-5}%
 2  &  2  &  ab+4b^{2}  &  \frac{-a}{2}  &  0 \\\cmidrule{1-5}%
 3  &  2  &  \frac{-a^{2}-4ab}{4}  &  \frac{-a}{2}  &  \frac{a\left(
a+4b\right)  ^{2}}{8} \\\bottomrule%
\end{array}
\label{AdothTs}
\end{equation}

\textbf{Case 5}. The family of elliptic curves $E_{C_{2}\times C_{4}}$ has additive reduction at a prime $p$ if
and only if $p=2$ and $v_{2}\!\left(  a+4b\right)  <4$ with $a$ even. Next,
let $F_{C_{2}\times C_{4},i}$ be the curve attained from $E_{C_{2}\times C_{4}}$ via the admissible
change of variables $x\longmapsto u_{i}^{2}x+r_{i}$ and $y\longmapsto
u_{i}^{3}y+u_{i}^{2}s_{i}x+w_{i}$ where $u_{i},r_{i},s_{i},w_{i}$ are as
defined in \eqref{AdothTs}.

\textbf{Subcase 5a.} Suppose $v_{2}\!\left(  a\right)  =1$ and observe that
the Weierstrass coefficients $a_{j}$ and the minimal discriminant $\Delta$ of
$F_{T,1}$ satisfy $
v_{2}\!\left(  a_{1}\right)=1  ,v_{2}\!\left(  a_{2}\right)=1
,v_{2}\!\left(  a_{3}\right)=2  ,v_{2}\!\left(  a_{4}\right) \geq 4  ,v_{2}\!\left(
a_{6}\right) = 5  ,\text{ and }v_{2}\!\left(  \Delta\right)=8$.

By Lemma~\ref{LemmaIn} we conclude that $E_{C_{2}\times C_{4}}$ has Kodaira-N\'{e}ron
type $\mathrm{I}_{1}^{\ast}$ at $2$ with $\left(  c_{2},f_{2}\right)  =\left(
4,3\right)  $.

\textbf{Subcase 5b.} Suppose $v_{2}\!\left(  a+4b\right)  =3$. Then
$F_{T,2}=E_{C_{2}\times C_{2}}\!\left(  A,B,1\right)  $ where $A=\frac{\left(
a+4b\right)  ^{2}}{16}$ and $\ B=b^{2}$. Note that $v_{2}(A)=2$ and
$B\equiv1\text{ }\operatorname{mod}4$. By Theorem~\ref{ThmC2xC2}, $E_{C_{2}\times C_{4}}$ has
Kodaira-N\'{e}ron type $\mathrm{I}_{1}^{\ast}$ at $2$ with $f_{2}=3$. Furthermore,
$c_{2}=4$ since $A\equiv4\text{ }\operatorname{mod}32$.

\textbf{Subcase 5c.} Suppose $v_{2}\!\left(  a\right)  \geq3$. Then
$F_{T,3}=E_{C_{2}\times C_{2}}\!\left(  A,B,1\right)  $ where $A=\frac
{-a\left(  a+8b\right)  }{16}$ and $B=\frac{-\left(  a+4b\right)  ^{2}}{16}$.
Since $v_{2}(A)\geq3$ and $B\equiv3\text{ }\operatorname{mod}4$ we conclude by
Theorem~\ref{ThmC2xC2} that $E_{C_{2}\times C_{4}}$ has Kodaira-N\'{e}ron type $\mathrm{I}_{n}^{\ast}$
at $2$ with $c_{2}=4$, $f_{2}=4$, and
\[
n=2v_{2}\!\left(  \frac{-a\left(  a+8b\right)  }{16}\right)  -4=\left\{
\begin{array}
[c]{ll}%
2v_{2}\!\left(  a+8b\right)  -6 & \text{if }v_{2}\!\left(  a\right)  =3,\\
2v_{2}\!\left(  a\right)  -6 & \text{if }v_{2}\!\left(  a\right)  \geq4.
\end{array}
\right.
\]

\textbf{Case 6}. The family of elliptic curves $E_{C_{2}\times C_{6}}$ has additive reduction at a prime $p$ if
and only if $p=3$ and $v_{3}\!\left(  b\right)  \geq1$. Then, by (\ref{valotherTs})
and Table~\ref{ta:PapTableIandII}, $E_{C_{2}\times C_{6}}$ has Kodaira-N\'{e}ron type
$\mathrm{I}_{n}^{\ast}$ with $f_{3}=2$ where $n=v_{3}\!\left(  \Delta\right)
-6$. By Corollary~\ref{CorC2xC2Instat}, we get $c_{3}=4$.
\end{proof}

\section{Consequences of main results}\label{Section on Consequences of the local data} 
In this section we deduce several results from the theorems proven in Section~\ref{PfMainThm}. The following result is a consequence of Proposition~\ref{rationalmodels} and Theorem~\ref{ThmC2xC2}.
\begin{corollary}
\label{CorC2xC2Instat} Let $E$ be a rational elliptic curve with full
$2$-torsion. If $E$ has Kodaira-N\'{e}ron type \rm{I}$_{n}^{\ast}$ at some prime $p$ with
$n>1$, then $n$ is even and $c_{p}=4$.
\end{corollary}

Similarly, looking at the local conductor exponents in Theorems~\ref{ThmC2at2},~\ref{ThmforC2podd},~\ref{ThmforC30}~and~\ref{ThmforC3} we get the following interesting fact.
\begin{corollary}
\label{CorollaryonConductor}If $E$ is a rational elliptic curve with a
$2$-torsion (resp.\ $3$-torsion) point, then $f_{p}\leq2$ for $p\neq2$
(resp.\ $p\neq3$).
\end{corollary}

As a consequence of the results proven in Section~\ref{PfMainThm}, we observe that a rational elliptic curve with non-trivial torsion subgroup has Kodaira-N\'{e}ron type $\mathrm{II}^{\ast}$ at some
prime $p$ if and only if $p=2$. Consequently, we get the following result:  

\begin{corollary}
\label{CorIIstar}
If a rational elliptic curve $E$ has Kodaira-N\'{e}ron type ${\mathrm{II}}^{\ast}$ at an odd
prime, then its torsion subgroup is trivial.
\end{corollary}

Next, we show how our results can be utilized to construct examples of rational elliptic curves with prescribed local data.
\begin{corollary}
\label{NeronC2}Let $S$ be a finite set of prime
numbers. For each $p\in S$, let $n_{p}$ be a nonnegative integer such that $n_{p}\neq1,2,3$ if $p=2$. Then there is a rational
elliptic curve $E$ with the property that it has additive reduction at each prime $p \in S$ and the Kodaira-N\'{e}ron type at $p$ is \rm{I}$_{n_p}^{\ast}$. Moreover, if $E$ has additive reduction at a prime $p$, then $p\in S\cup\left\{  2\right\}$. 
\end{corollary}
\begin{proof}
For $p\in S$, let%

\vspace*{-0.1in}
\[
\left(  l_{p},r_{p},s_{p}\right)  =\left\{
\begin{array}
[c]{cl}%
\left(  0,1,1\right)   & \text{if }p=2\text{ and }n_{p}=0,\\
\left(  1,\frac{n_{p}}{2},0\right)   & \text{if }p=2\text{ and }n_{p}%
\geq4\text{ is even},\\
\left(  1,\frac{n_{p}-1}{2},1\right)   & \text{if }p=2\text{ and }n_{p}%
\geq5\text{ is odd},\\
\left(  2,1,1\right)   & \text{if }p\neq2\text{ and }n_{p}=0,\\
\left(  1,\frac{n_{p}+2}{2},0\right)   & \text{if }p\neq2\text{ and }n_{p}%
\geq2\text{ is even},\\
\left(  1,\frac{n_{p}+1}{2},1\right)   & \text{if }p\neq2\text{ and }n_{p}%
\geq1\text{ is odd.}%
\end{array}
\right.
\]
Now define%
\[
A=\prod_{p\in S}p^{l_p},\qquad 
B=\prod_{p\in S}p^{r_{p}},\qquad
D=\prod_{p\in S}p^{s_{p}}.
\]
Then by Theorems~\ref{ThmC2at2}~and~\ref{ThmforC2podd} we have that $E_{C_{2}%
}\!\left(  A,B,D\right)  $ has additive reduction at a prime $p$ if and only if $p\in S\cup\left\{  2\right\}$. Moreover, if $p\in S$, then
its Kodaira-N\'{e}ron type at $p$ is \rm{I}$_{n_{p}}^{\ast}$.
\end{proof}

From Corollary~\ref{NeronC2}, we see that the 
elliptic curve $E=E_{C_{2}}\!\left(
210,2^{6}\cdot3^{4}\cdot5^{6}\cdot7^{5},15\right):$
	\[
	E: y^{2}=x^{3}+13440x^{2}-54296487559248001716600x
	\]
considered in Example~\ref{introexample} has Kodaira-N\'{e}ron type \rm{I}$_{n}^{\ast}$ at $p$
where $\left(  p,n\right)  $ is one of the following pairs: $\left(
2,12\right)  ,\ \left(  3,7\right)  ,\ \left(  5,11\right)  ,$ or $\left(
7,8\right)$.

\section{Global Tamagawa numbers}\label{Global Tamagwa}
Given a rational elliptic
curve $E$, we denote the \textit{global Tamagawa number} $c=\prod_{p}c_{p}$. In this section, we classify the rational elliptic curves with a torsion point of order $2$ or $3$ that have global Tamagawa number~$1$. In \cite[Lemma 2.26]{Lorenzini2012}, Lorenzini showed that there is an infinite family of rational elliptic curves with a $3$-torsion point such that $c=1$. In the notation of this article, the family constructed by Lorenzini is equivalent to the existence of infinitely many integers $a$ with $3\nmid a$ such that $
E_{C_{3}}\!\left(  -a^{3},-1\right)  :y^{2}-a^{3}xy-a^{6}y=x^{3}$
has global Tamagawa number $1$. Next, we prove Theorem~\ref{GlobalTamaat3}, which determines all rational elliptic curves with a $3$-torsion point that have global Tamagawa number
equal to $1$.

\begin{proof}[Proof of Theorem~\ref{GlobalTamaat3}]
By Proposition~\ref{rationalmodels}, $E$ is $\mathbb{Q}$-isomorphic 
to $E_{C_{3}^{0}}\!\left(  a\right)$, for some cubefree
integer $a$, or to $E_{C_{3}}\!\left(  a,b\right)$, for some relatively prime
integers $a$ and $b$, with $a$ positive. First suppose $E$ is $\mathbb{Q}$-isomorphic to $E_{C_{3}^{0}%
}=E_{C_{3}^{0}}\!\left(  a\right)  $ for some positive cubefree integer. By
Theorem~\ref{ThmforC30}, $E$ has additive reduction at $3$ and at each prime
$p$ dividing $a$. If $p$ divides $a$, 
then $c_{p}=3$. Thus $v_{p}\!\left(
a\right)  =0$ for each prime $p$, because of our assumption that $c_{p}=1$. Consequently, $a=1$ since
$a$ is positive and by Theorem~\ref{ThmforC30}, $E_{C_{3}^{0}}\!\left(
1\right)  :y^{2}+y=x^{3}$ has global Tamagawa number $1$.

Now suppose $E$ is $\mathbb{Q}$-isomorphic to $E_{C_{3}}\!\left(  k,b\right)  $ for some relatively prime
integers $k$ and $b$ with $k$ positive. Now let $p$ be a
prime and observe that if $v_{p}\!\left(  k\right)  \not \equiv
0\ \operatorname{mod}3$, then Theorem~\ref{ThmforC3} implies that $E$ has additive reduction at $p$ with $c_{p}=3$,
which is a contradiction. Thus $v_{p}\!\left(  k\right)  \equiv
0\ \operatorname{mod}3$ for each prime $p$ which is equivalent to $k$ being a
cube. Therefore $k=a^{3}$ for some positive integer $a$ and $E$ is semistable at each
prime $p\neq3$ by Theorem~\ref{ThmforC3}. Next, suppose $p$ is
a prime dividing $b$ and let $E_1$ be the elliptic curve attained from
$E_{C_{3}}\!\left(  a^{3},b\right)  $ via the admissible change of variables
$x\longmapsto a^{4}x+a^{4}b$ and $y\longmapsto a^{6}y$. Then the Weierstrass
coefficients $a_{3},a_{4},a_{6}$ of $E_1\ $are divisible by $p$ and $t^{2}%
+a_{1}t-a_{2}\equiv t\left(  t+a\right)  \ \operatorname{mod}p$. Hence, by Tate's Algorithm, $E$
has split multiplicative reduction at $p$. Hence $c_{p}=v_{p}\!\left(  \Delta\right)  =3v_{p}\!\left(  b\right)  $,
which contradicts our assumption that $c_{p}=1$. Thus $v_{p}\!\left(
b\right)  =0$ for each prime $p$ which implies that $b=\pm1$. In particular,
$\Delta=\pm a^{3}\mp27$. If $v_{p}\!\left(  \Delta\right)  \leq1$
for some prime $p$, then $c_{p}=1$ by Tate's Algorithm. If $v_{3}\!\left(
\pm a^{3}\mp27\right)  >0$, then $E$ has additive reduction at $3$ and by Theorem~\ref{ThmforC3} we have that $c_{3}=1$ if and only if one of the conditions in
(\ref{C3tamacond}) is satisfied for $v_{3}\!\left(  \pm a^{3}\mp27\right)  \geq3$.
Lastly, suppose $p\neq3$ with $v_{p}\!\left(  \Delta\right)  >1$. Since $E$
has multiplicative reduction at $p$, by Tate's Algorithm $c_{p}=1$ if and only if $E$
has non-split multiplicative reduction at $p$ and $v_{p}\!\left( \pm a^{3}\mp27\right)  $ is odd. Next, let $E_2$ be the elliptic curve attained
from $E_{C_{3}}\!\left(  a^{3},\pm 1\right)  $ via the admissible change of
variables $x\longmapsto\frac{a^{4}}{9}x-\frac{a^{6}}{9}$ and $y\longmapsto
\frac{a^{6}}{27}y-\frac{a^{7}}{9}x+\frac{a^{9}}{27}$. Then the Weierstrass
coefficients $a_{3},a_{4},a_{6}$ of $E_{2}\ $are divisible by $p$ and Tate's
Algorithm implies that $E$ has non-split multiplicative reduction if and only if
$t^{2}+a_{1}t-a_{2}\equiv t^{2}-3at+3a^{2}\ \operatorname{mod}p$ does not
split in $\mathbb{F}_{p}$. This is equivalent to $\left(  \frac{-3a^{2}}%
{p}\right)  =\left(  \frac{-3}{p}\right)  =-1$. In particular,  $p\equiv
5,11\ \operatorname{mod}12$.
\end{proof}

The next  result gives another infinite family of curves with a $3$-torsion point that has global Tamagawa number equal to $1$.
\begin{corollary}\label{CorC3c1}
There are infinitely many non-isomorphic elliptic curves with a $3$-torsion
point and global Tamagawa number equal to $1$.
\end{corollary}
\begin{proof}
By \cite{Erdos1953}, the set $S=\left\{ k\in \mathbb{Z}\mid 27k^{3}-1\text{ is squarefree}\right\} $. Then, for $k\in S$, it is the case that $v_{p}\!\left( 27\left( 27k^{3}-1\right) \right) \leq 1$ for each prime 
$p\neq 3$ and $v_{3}\!\left( 27\left( 27k^{3}-1\right) \right) =3$.
Moreover, $1-9k\not\equiv -2\ \operatorname{mod}9$. It follows by Theorem \ref{GlobalTamaat3} that $E_{C_{3}}\!\left( 729k^{3},1\right) $ has
global Tamagawa number equal to $1$ since $729k^{3}-27=27\left(
27k^{3}-1\right) $. In fact, each $k \in S$ determines a unique elliptic curve $E_{C_{3}}\!\left( 729k^{3},1\right)$ with minimal discriminant $27\left( 27k^{3}-1\right) $.
\end{proof}

Lastly, we explicitly find all rational elliptic curves with a $2$-torsion point which have $c=1$. Consequently, we give an infinite family of rational elliptic curves with a $2$-torsion and $c=1$. 

\begin{theorem}
\label{GlobalTamaat2}Let $E$ be a rational elliptic curve with a $2$-torsion
point. Then the global Tamagawa number of $E$ is $c=1$ if and only if $E$ is
$\mathbb{Q}$-isomorphic to $E_{C_{2}}(a,b,d)$ where $a,b,d$ are integers such
that $d$ and $\gcd\!\left(  a,b\right)  $ are squarefree, and satisfy one of the following conditions at each prime $p$:
\end{theorem}
{\begingroup
	\begin{longtable}{cc}
		\hline
		$p$ & \text{Conditions on } $a,b,d$\\
		\hline
		
		\endfirsthead
		
		\hline
		\toprule $p$ & \text{Conditions on } $a,b,d$\\
		\hline
		\endhead
		\hline
		
		\multicolumn{2}{r}{\emph{continued on next page}}
		\endfoot
		\hline
		\caption{Conditions for $E_{C_2}(a,b,d)$ to have global Tamagawa number equal to $1$.} \label{globaltamagawaC2}
		\endlastfoot
		$\neq2$ & $v_{p}(bd)=0,\ v_{p}\!\left(  b^{2}d-a^{2}\right)=0$\\\cmidrule{2-2}
		& $v_{p}(b)=0,\ v_{p}(d)=1,\ v_{p}\!\left(  b^{2}d-a^{2}\right)=0$\\\cmidrule{2-2}
		& $v_{p}(b)\geq1,\ v_{p}(d)=1,v_{p}(a)=0,\text{ and }\left(  \frac{-a}{p}\right)=-1$\\\midrule
		$2$ & $v_{2}(b)=3,\ v_{2}(d)=0,\ a\equiv31\ \operatorname{mod}64,\ b^{2}d-a^{2}=-1$\\\cmidrule{2-2}
		& $v_{2}(b)=3,\ v_{2}(d)=1,\text{\ }a\equiv63\ \operatorname{mod}64,\ b^{2}d-a^{2}=-1$\\\cmidrule{2-2}
		& $v_{2}(ab)=0,\ v_{2}\!\left(  d\right)  =1,\ b^{2}d-a^{2}=1$\\\cmidrule{2-2}
		& $v_{2}(a)>0,\ d\equiv3\ \operatorname{mod}4,\ b^{2}d-a^{2}=-1$\\\cmidrule{2-2}
		& $v_{2}(b)=v_{2}\!\left(  d\right)  =1,\ a\equiv5\ \operatorname{mod}8, b^{2}d-a^{2}=-1$\\\cmidrule{2-2}
		&$ v_{2}(b)=2,\ v_{2}(d)=1,\ a\equiv15\ \operatorname{mod}32,\ b^{2}d-a^{2}=-1$\\\cmidrule{2-2}
		& $v_{2}\!\left(  a\right)  =v_{2}\!\left(  b\right)  =1,\ b^{2}d-a^{2}-8a\equiv32\ \operatorname{mod}64,\ b^{2}d-a^{2}=\pm16$\\\cmidrule{2-2}
		& $v_{2}\!\left(  b\right)  =1,\ a\equiv10\ \operatorname{mod}16,\ b^{2}d-a^{2}=16$\\\cmidrule{2-2}
		& $v_{2}\!\left(  a\right)  =v_{2}\!\left(  b\right)  =1,\text{ }a\equiv2\ \operatorname{mod}8,\ b^{2}d-a^{2}=256$
	\end{longtable}
	\endgroup}

\begin{proof}
Suppose $E$ is a rational elliptic curve with a $2$-torsion point such that
its global Tamagawa number is $1$. Next, let $\Delta$ denote the minimal
discriminant of $E$. By Proposition~\ref{rationalmodels}, $E$ is $\mathbb{Q}$-isomorphic to either $E_{C_{2}}\!\left(  a,b,d\right)  $ or $E_{C_{2}\times
C_{2}}\!\left(  a,b,d\right)  $ for some integers $a,b,d$. Towards a
contradiction, suppose that $E$ is $\mathbb{Q}$-isomorphic to $E_{C_{2}\times C_{2}}\!\left(  a,b,d\right)  $. Note that $a$
and $b$ are relatively prime with $a$ even by Proposition~\ref{rationalmodels}%
. By Theorem~\ref{ThmC2xC2}, if $E$ has additive reduction at $p$, then
$c_{p}>1$ which is a contradiction. In particular, $E$ is semistable and by \cite[Theorem 7.1]{Barrios2020}, this occurs if and only if $d=1$, $v_{2}(a)\geq4$, and $b\equiv
1\ \operatorname{mod}4$. Then there
exists an odd prime $p|\left(  a-b\right)  $. Since $v_{p}\!\left(
\Delta\right)  =2v_{p}\!\left(  a-b\right)  $,
 $c_{p}\geq2$ is even by Tate's Algorithm, which is a contradiction.

It follows that $E$ is $\mathbb{Q}$-isomorphic to $E_{C_{2}}\!\left(  a,b,d\right)  $ with $\gcd\!\left(
a,b\right)  $ and $d$ squarefree. We claim that $v_{p}\!\left(  b^{2}%
d-a^{2}\right)  =0$ for each odd prime $p$. To this end, suppose
$v_{p}\!\left(  b^{2}d-a^{2}\right)  >0$ for some odd prime $p$. Then
$v_{p}\!\left(  \Delta\right)  >0$ since $v_{p}\!\left(  \Delta\right)
=v_{p}\!\left(  b^{2}d\right)  +2v_{p}\!\left(  b^{2}d-a^{2}\right)  $. If $p$
divides $b^{2}d$, then $p$ divides $a$ and thus $E$ has additive reduction at
$p$ with $c_{p}>1$ by Theorem~\ref{ThmforC2podd}. If $p$ does not divide
$b^{2}d$, then $v_{p}\!\left(  \Delta\right)  =2v_{p}\!\left(  b^{2}%
d-a^{2}\right)  $ and by Tate's Algorithm we have that $c_{p}\geq2$ is even,
which is a contradiction.

By the above, $E$ is semistable at all odd primes $p$ and thus $v_{p}\!\left(  \Delta\right)  =v_{p}\!\left(  b^{2}d\right) $. Moreover, if $E$ has good reduction at $p$, then $v_{p}\!\left(
bd\right)  =0$ and hence $c_{p}=1$. Now suppose $v_{p}\!\left(  \Delta\right)
>0$ so that $v_{p}\!\left(  b^{2}d\right)  >0$. If $v_{p}\!\left(  d\right)
=0$ and $v_{p}\!\left(  b\right)  \geq1$, then $v_{p}\!\left(  \Delta\right)
$ is even and by Tate's Algorithm we have that $c_{p}\geq2$ which is a contradiction. If $v_{p}\!\left(  b\right)  =0$ and $v_{p}\!\left(
d\right)  =1$, then $c_{p}=1$ by Tate's Algorithm. Next, suppose
$v_{p}\!\left(  b\right)  \geq1$ and $v_{p}\!\left(  d\right)  =1$. Since
$v_{p}\!\left(  \Delta\right)  \geq3$ is odd, then we have by Tate's Algorithm that
$c_{p}=1$ if and only if $E$ has non-split multiplicative reduction at $p$. To
this end, let $E_{1}$ be the elliptic curve attained from $E_{C_{2}}\!\left(
a,b,d\right)  $ via the admissible change of variables $x\longmapsto
x+b^{2}-a$ and $y\longmapsto y+\left(  bd-a\right)  x+b$. Then the Weierstrass
coefficients $a_{3},a_{4},a_{6}$ of $E_{1}$ are divisible by $p$. By Tate's
Algorithm, it follows that $E$ has non-split multiplicative reduction at $p$
if and only if $t^{2}+a_{1}t-a_{2}\equiv t^{2}-2at+a^{2}+a\ \operatorname{mod}%
p$ is irreducible over $\mathbb{F}_{p}$. This is equivalent to $\left(
\frac{-4a}{p}\right)  =\left(  \frac{-a}{p}\right)  =-1$.

We divide this part of the proof in various cases. Note that $b^{2}d-a^{2}=\pm2^{n}$ for $n$ a nonnegative integer, by the above.

\textbf{Case 1.} Suppose $E$ is semistable at $2$. Then,  by \cite[Theorems~4.4~and~7.1]{Barrios2020}, we have $\left(  i\right)  $
$v_{2}\!\left(  b^{2}d-a^{2}\right)  \geq8$, $v_{2}\!\left(  a\right)
=v_{2}\!\left(  b\right)  =1$, and $a\equiv2\ \operatorname{mod}8$ or $\left(
ii\right)  $ $v_{2}\!\left(  b\right)  \geq3$ with $a\equiv
3\ \operatorname{mod}4$.

\qquad\textbf{Subcase 1a.} Suppose $v_{2}\!\left(  b^{2}d-a^{2}\right)  \geq
8$, $v_{2}\!\left(  a\right)  =v_{2}\!\left(  b\right)  =1$, and
$a\equiv2\ \operatorname{mod}8$. By \cite[Theorem~4.4]{Barrios2020}, $v_{2}\!\left(  \Delta\right)
=2v_{2}\!\left(  b^{2}d-a^{2}\right)  -16$. If $v_{2}\!\left(  b^{2}d-a^{2}\right)  >8$, then $E$
has multiplicative reduction at $2$ and $v_{2}\!\left(  \Delta\right)
\geq2$ is even. Then, by Tate's Algorithm $c_{2}\geq2$ is even, which
is a contradiction. If $v_{2}\!\left(  b^{2}d-a^{2}\right)  =8$, then $E$ has
good reduction at $2$ and $c_{2}=1$.

\qquad\textbf{Subcase 1b}. Suppose $v_{2}\!\left(  b\right)  \geq3$ with
$a\equiv3\ \operatorname{mod}4$. Then $b^{2}d-a^{2}\equiv-1\ \operatorname{mod}4$ which
implies that $b^{2}d-a^{2}=-1$. By \cite[Theorem~4.4]{Barrios2020},
$v_{2}\!\left(  \Delta\right)  =v_{2}\!\left(  b^{2}d\right)  -6$. Note that if $v_{2}\!\left(  b\right)  >3$ and
$v_{2}\!\left(  d\right)  =0$, then $c_{2}\geq2$ is even by Tate's Algorithm, which is a contradiction.
So suppose $v_{2}\!\left(  b\right)  >0$ with $v_{2}\!\left(  d\right)  =1$.
By Tate's Algorithm, $c_{2}=1$ if and only if $E$ has non-split
multiplicative reduction at $2$. Next, let $E_{2}$ be the elliptic curve attained
from $E_{C_{2}}\!\left(  a,b,d\right)  $ via the admissible change of
variables $x\longmapsto4x+b^{2}d-a$ and $y\longmapsto8y+4\left(  bd-a\right)
x+bd$. Then the Weierstrass coefficients $a_{3},a_{4},a_{6}$ of $E_{2}$ are even
and by Tate's Algorithm, $E$ has non-split multiplicative reduction at $2$ if
and only if $t^{2}+a_{1}t-a_{2}\equiv t^{2}+t+\frac{a^{2}+a}{4}%
\ \operatorname{mod}2$ is irreducible. Note that $b^{2}d=a^{2}-1%
\equiv0\ \operatorname{mod}16$ implies that $a\equiv7\ \operatorname{mod}8$. Hence $a^{2}+a\equiv0\ \operatorname{mod}8$. In particular, $t^{2}%
+t+\frac{a^{2}+a}{4}\ \operatorname{mod}2$ splits over $\mathbb{F}_{2}$ which
contradicts our assumptions. It remains to show the case
when $v_{2}\!\left(  b\right)  =3$. If $v_{2}\!\left(  b\right)  =3$ with
$v_{2}\!\left(  d\right)  =0$, then $E$ has good reduction at $2$ and thus
$c_{2}=1$. Now observe that this holds if and only if $b^{2}d=a^{2}-1\equiv
64\ \operatorname{mod}128$. This is equivalent to $a\equiv31\ \operatorname{mod}64$ since $a\equiv
3\ \operatorname{mod}4$. Lastly, if
$v_{2}\!\left(  b\right)  =3$ with $v_{2}\!\left(  d\right)  =1$, then
$v_{2}\!\left(  \Delta\right)  =1$ and thus $c_{2}=1$. Moreover, this holds if
and only if $b^{2}d=a^{2}-1\equiv128\ \operatorname{mod}256$ which is equivalent to
$a\equiv63\ \operatorname{mod}64$ since $a\equiv3\ \operatorname{mod}4$.

\textbf{Case 2.} Suppose $E$ has additive reduction at $2$. Then $a,b,$ and
$d$ satisfy one of the conditions appearing in Theorem~\ref{ThmC2at2}. It therefore suffices to consider
the cases corresponding to $c_{2}=1$ in Table~\ref{TableforC2}.

\qquad\textbf{Subcase 2a.} Suppose $v_{2}\!\left(  b^{2}d-a^{2}\right)  >0$. If
$v_{2}\!\left(  b^{2}d-a^{2}\right)  =4$, $v_{2}\!\left(  a\right)
=v_{2}\!\left(  b\right)  =1,$ and $b^{2}d-a^{2}-8a\equiv
32\ \operatorname{mod}64$, then $b^{2}d-a^{2}=\pm16$. If $b^{2}d-a^{2}%
\equiv16\ \operatorname{mod}64$,\ $a\equiv2\ \operatorname{mod}8,$ and
$b^{2}d-a^{2}\not \equiv 8a\ \operatorname{mod}128$, then $b^{2}d-a^{2}=16$. Moreover, $8a\not \equiv
16\ \operatorname{mod}128$ if and only if $a\equiv10\ \operatorname{mod}16$.
This concludes the proof for the cases when $c_{2}=1$ with $v_{2}\!\left(
b^{2}d-a^{2}\right)  >0$ by Theorem~\ref{ThmC2at2}.

\qquad\textbf{Subcase 2b.} Suppose $v_{2}\!\left(  b\right)  =v_{2}\!\left(
b^{2}d-a^{2}\right)  =0$. Then $b^{2}d-a^{2}=\pm1$. If $v_{2}\!\left(
a\right)  =0$ and $v_{2}\!\left(  d\right)  =1$, then $b^{2}d-a^{2}%
\equiv1\ \operatorname{mod}4$ and thus $b^{2}d-a^{2}=1$. If $a$ is even and
$d\equiv3\ \operatorname{mod}4$, then $b^{2}d-a^{2}\equiv3\ \operatorname{mod}%
4$ and thus $b^{2}d-a^{2}=-1$. This concludes the proof for the cases when
$c_{2}=1$ with $v_{2}\!\left(  b\right)  =v_{2}\!\left(  b^{2}d-a^{2}\right)
=0$ by Theorem~\ref{ThmC2at2}.

\qquad\textbf{Subcase 2c.} Suppose $v_{2}\!\left(  b^{2}d-a^{2}\right)  =0$
with $b$ even. Then $a$ is odd and $b^{2}d-a^{2}\equiv-1\ \operatorname{mod}4$. Hence $b^{2}d-a^{2}=-1$. If $v_{2}\!\left(  b\right)  =2,\ v_{2}\!\left(
d\right)  =1,$ and $a\equiv3\ \operatorname{mod}4$, then $b^{2}d\equiv
32\ \operatorname{mod}64$ and thus $a\equiv15\ \operatorname{mod}32$ since
$a^{2}-1\equiv32\ \operatorname{mod}64$. If $v_{2}\!\left(  b\right)
=1,\ d\equiv3\ \operatorname{mod}4$, and either $a\equiv1\ \operatorname{mod}%
4$ or $a\equiv3\ \operatorname{mod}4$ with $ad\equiv5\ \operatorname{mod}8$,
then $b^{2}d-a^{2}\equiv3\ \operatorname{mod}8$ and thus $b^{2}d-a^{2}\neq-1$
which is a contradiction. If $v_{2}\!\left(  b\right)  =1,\ a\equiv
1\ \operatorname{mod}4,$ and $d\equiv2\ \operatorname{mod}4$, then
$a\equiv5\ \operatorname{mod}8$ since $b^{2}d\equiv8\ \operatorname{mod}16$
and $b^{2}d-a^{2}\equiv-1\ \operatorname{mod}16$. This concludes the proof for when $v_{2}\!\left(  b^{2}d-a^{2}\right)  =0$
with $b$ even, and thus we have the forward direction of the theorem.

For the converse, observe that if $E_{C_{2}}(a,b,d)$ satisfies one of the conditions in Table~\ref{globaltamagawaC2} for each prime $p\mid \Delta$, then the previous arguments imply that $c=1$.
\end{proof}

\begin{corollary}
There are infinitely many non-isomorphic rational elliptic curves with a
$2$-torsion point and global Tamagawa number equal to $1$.
\end{corollary}

\begin{proof}
By \cite{Nagel1922}, the set $S=\left\{  2\left(  2k^{2}+2k+1\right)  \mid
k\in
\mathbb{Z},\ 2k^{2}+2k+1\text{ is squarefree}\right\}  $ is infinite. Now let $d\in S$ so that there is an integer $k$ such that $d=\left(  2k+1\right)  ^{2}+1$. In particular, the elliptic
curve
\[
E_{C_{2}}\!\left(  2k+1,1,d\right)  :y^{2}=x^{3}+\left(  4k+2\right)  x^{2}-x
\]
has minimal discriminant $\Delta=64d$ and by Theorem~\ref{GlobalTamaat2} we
have that its global Tamagawa number is $1$. Next let $d^{\prime}\in S$ with
$d^{\prime}=\left(  2k^{\prime}+1\right)  ^{2}+1$ for some integer $k^{\prime}$. Then the elliptic curves $E_{C_{2}}\!\left(  2k+1,1,d\right)  $ and
$E_{C_{2}}\!\left(  2k^{\prime}+1,1,d\right)  $ are $\mathbb{Q}$-isomorphic if and only if $d=d^{\prime}$ which shows that there are
infinitely many non-isomorphic rational elliptic curves with a $2$-torsion
point and global Tamagawa number equal to $1$.
\end{proof}

\noindent \textbf{Acknowledgments.} We would like to thank Imin Chen and Alyson Deines for their helpful comments as we prepared this manuscript. We would also like to thank the referee for his/her detailed comments and suggestions.

\newpage
\begin{appendix}
\section{}
\label{AppendixTables}

{\setlength{\tabcolsep}{12pt}
\renewcommand{\arraystretch}{0.95} 
\begin{longtable}{ccccc}
	\caption[Admissible]{Maps from $E_{C_{2}}$ to $F_{C_{2},i}$;  $x\longmapsto u_{i}^{2}x+r_{i}$ and
$y\longmapsto u_{i}^{3}y+u_{i}^{2}s_{i}x+w_{i}$.}\\
	\midrule
	$i$ & $u_{i}$ & $r_{i}$ & $s_{i}$ & $w_{i}$\\
	\midrule
	\endfirsthead
	\caption[]{\emph{continued}}\\
	\midrule
	$i$ & $u_{i}$ & $r_{i}$ & $s_{i}$ & $w_{i}$\\
	\midrule
	\endhead
	\midrule
	\multicolumn{2}{r}{\emph{continued on next page}}
	\endfoot
	\midrule
	\endlastfoot
$1$ & $2$ & $4$ & $2$ & $8$\\ \cmidrule(lr){1-5}
$2$ & $1$ & $b$ & $b$ & $a$\\ \cmidrule(lr){1-5}
$3$ & $2$ & $8$ & $2$ & $16$\\ \cmidrule(lr){1-5}
$4$ & $1$ & $-a$ & $1$ & $2d$\\ \cmidrule(lr){1-5}
$5$ & $1$ & $16\left(  b^{2}d-a^{2}\right)  $ & $2$ & $8\left(  b^{2}%
d-a^{2}\right)  $\\ \cmidrule(lr){1-5}
$6$ & $2$ & $\frac{b^{2}d-a^{2}}{8}$ & $2$ & $\frac{b^{2}d-a^{2}}{4}$\\ \cmidrule(lr){1-5}
$7$ & $1$ & $d^{2}-a$ & $2d$ & $2d$\\ \cmidrule(lr){1-5}
$8$ & $1$ & $\frac{b^{2}}{2}$ & $bd$ & $b^{4}$\\ \cmidrule(lr){1-5}
$9$ & $1$ & $\frac{ab^{2}d}{2}-a$ & $-bd-a$ & $b$\\ \cmidrule(lr){1-5}
$10$ & $1$ & $\frac{3b^{2}}{2}$ & $b$ & $b^{3}+b^{2}$\\ \cmidrule(lr){1-5}
$11$ & $1$ & $\frac{b^{2}}{2}$ & $b$ & $b^{4}+b^{2}$\\ \cmidrule(lr){1-5}
$12$ & $1$ & $a+2$ & $3a^{2}+12a+12$ & $4$\\ 
%\cmidrule(lr){1-5}
$13$ & $1$ & $a$ & $2$ & $4b$\\ \cmidrule(lr){1-5}
$14$ & $1$ & $a$ & $2$ & $4b-16d+8$\\ \cmidrule(lr){1-5}
$15$ & $1$ & $-a$ & $4$ & $2b$\\ \cmidrule(lr){1-5}
$16$ & $1$ & $\frac{abd}{4}-a-b$ & $4$ & $abd-2b$\\ \cmidrule(lr){1-5}
$17$ & $1$ & $b^{2}-a$ & $bd-a$ & $b$\\ \cmidrule(lr){1-5}
$18$ & $1$ & $\frac{bda}{2}-a$ & $-bd-a$ & $b$\\ \cmidrule(lr){1-5}
$19$ & $1$ & $\frac{abd}{2}-a$ & $-bd-a$ & $b\left(  d-1\right)  $\\ \cmidrule(lr){1-5}
$20$ & $1$ & $abd-a$ & $-bd-a$ & $2b$\\ \cmidrule(lr){1-5}
$21$ & $1$ & $\frac{abd}{2}-a$ & $-bd-a$ & $2b$\\ \cmidrule(lr){1-5}
$22$ & $2$ & $\frac{\left(  b^{2}d-a^{2}\right)  ^{2}}{2048}$ & $2$ &
$\frac{3\left(  b^{2}d-a^{2}\right)  ^{2}}{1024}$ \\  \cmidrule(lr){1-5}
$23$ & $1$ & $9$ & 3 & 9\\  \cmidrule(lr){1-5}
$24$ & $1$ & $3b^{2}d+a^{2}$ & $\frac{3b^{2}d+a^{2}}{p}$ & $(3b^{2}d+a^{2})p$\\  \cmidrule(lr){1-5}
$25$ & $1$ & $-a$ & $a$ & $ab^{2}$\\  \cmidrule(lr){1-5}
$26$ & $1$ & $b^{2}d-a^{2}$ & $\frac{b^{2}d-a^{2}}{p}$ & $(b^{2}d-a^{2})p$
	
\label{ta:WeierTisC2}	
\end{longtable}}

{\setlength{\tabcolsep}{6pt}
\renewcommand{\arraystretch}{0.79} 
\begin{longtable}{ccccccc}
	\caption[Red]{Local data at $p=2$ given in terms of $v_{2}\!\left(c_{4}\right)$, $v_{2}\!\left(  c_{6}\right)$, and $v_{2}\!\left(  \Delta\right)$.
	}\\
	\toprule
	N\'{e}ron Type & $c_{2}$ &$v_{2}\!\left(c_{4}\right)$ & $v_{2}\!\left(  c_{6}\right)$ & $v_{2}\!\left(  \Delta\right)  $ & Max Step in Tate's Algorithm &
$f_{2}$\\
	\toprule
	\endfirsthead
	\caption[]{\emph{continued}}\\
	\hline
	N\'{e}ron Type & $c_{2}$ &$v_{2}\!\left(c_{4}\right)$ & $v_{2}\!\left(  c_{6}\right)$ & $v_{2}\!\left(  \Delta\right)  $ & Max Step in Tate's Algorithm &
$f_{2}$\\
	\hline
	\endhead
	\hline
	\multicolumn{2}{r}{\emph{continued on next page}}
	\endfoot
	\hline
	\endlastfoot
II &$1$ &$\geq4$ & $5$ & $4$ & $5$ & $4$\\ \cmidrule(lr){3-7}
& &$4$ & $\geq6$ & $6$ & $4$ & $6$\\ \cmidrule(lr){3-7}
& &$4$ & $6$ & $7$ &  & $7$\\ \cmidrule(lr){3-7}
& &$\geq5$ & $6$ & $6$ &  & $6$\\ \cmidrule(lr){1-7}
$\mathrm{III}$ & $2$&$4$ & $5$ & $4$ & $5$ & $3$\\ \cmidrule(lr){3-7}
& &$5$ & $5$ & $4$ &$4$ & $3$\\ \cmidrule(lr){3-7}
& &$4$ & $\geq6$ & $6$ & $4$ & $5$\\ \cmidrule(lr){3-7}
& &$5$ & $7$ & $8$ &  & $7$\\ \cmidrule(lr){3-7}
& &$5$ & $\geq8$ & $9$ &  & $8$\\ \cmidrule(lr){1-7}
IV &$1$ or $3$ &$4$ & $5$ & $4$ & $5$ & $2$\\ \cmidrule(lr){3-7}
& &$\geq6$ & $5$ & $4$ & $5$ & $2$\\ \cmidrule(lr){1-7}
$\mathrm{I}_{0}^{\ast}$ & $1$ or $2$ &$4$ & $6$ & $8$ & $8$& $4$\\\cmidrule(lr){3-7}
& &$\geq6$ & $7$ & $8$ & $8$ & $4$\\ \cmidrule(lr){3-7}
& &$4$ & $6$ & $9$ &  & $5$\\ \cmidrule(lr){3-7}
& &$\geq6$ & $8$ & $10$ &  & $6$\\ \cmidrule(lr){1-7}
$\mathrm{I}_{1}^{\ast}$ & $2$ or  $4$& $4$ & $6$ & $8$ &$8$ & $3$\\ \cmidrule(lr){3-7}
& &$6$ & $7$ & $8$ & $7$ & $3$\\ \cmidrule(lr){1-7}
$\mathrm{I}_{2}^{\ast}$ & $2$ or  $4$&$4$ & $6$ & $10$ & $9$ & $4$\\ \cmidrule(lr){3-7}
& &$6$ & $\geq9$ & $12$ & $7$ & $6$\\ \cmidrule(lr){3-7}
& &$6$ & $9$ & $13$ &  & $7$\\ \cmidrule(lr){1-7}
$\mathrm{I}_{3}^{\ast}$ &$2$ or  $4$ &$4$ & $6$ & $11$ & $10$ & $4$\\ \cmidrule(lr){3-7}
& &$6$ & $\geq9$ & $12$ & $7$ & $5$\\ \cmidrule(lr){1-7}
I$_{n(n\ge 4)}^{\ast}$ &$2$ or  $4$ &$4$ & $6$ & $8+n$ & & $4$\\ \cmidrule(lr){3-7}
& &$6$ & $9$ & $10+n$ &  & $6$\\ \cmidrule(lr){1-7}
$\mathrm{IV}^{\ast}$ & $1$ or $3$ &$4$ & $6$ & $8$ & $8$ & $2$\\ \cmidrule(lr){3-7}
& &$\geq7$ & $7$ & $8$ & $8$ & $2$\\ \cmidrule(lr){1-7}
$\mathrm{III}^{\ast}$ & $2$ &$4$ & $6$ & $10$ & $9$ & $3$\\ \cmidrule(lr){3-7}
& &$7$ & $9$ & $12$ &  & $5$\\ \cmidrule(lr){3-7}
& &$7$ & $10$ & $14$ &  & $7$\\ \cmidrule(lr){3-7}
& &$7$ & $\geq11$ & $15$ &  & $8$\\ \cmidrule(lr){1-7}
$\mathrm{II}^{\ast}$ &$1$ &$4$ & $6$ & $11$ & $10$ & $3$\\ \cmidrule(lr){3-7}
& &$\geq8$ & $9$ & $12$ & & $4$\\ \cmidrule(lr){3-7}
& &$\geq8$ & $10$ & $14$ & & $6$
\label{ta:PapTableIV}	
\end{longtable}}
%%%%%%%%%%%%%%%%%%%%%%%%%%%%%
{\renewcommand{\arraystretch}{0.93} 
\begin{longtable}{cccccccc}
	\caption[Red]{Local data at $p\geq3$ given in terms of $v_{p}\!\left(c_{4}\right)$, $v_{p}\!\left(  c_{6}\right)$, and $v_{p}\!\left(  \Delta\right)$.
	}\\
	\toprule
	N\'{e}ron Type & $c_{p}$ &$p$ & $v_{p}\!\left(  c_{4}\right)  $ & $v_{p}\!\left(
c_{6}\right)  $ & $v_{p}\!\left(  \Delta\right)  $ & Additional Condition &
$f_{p}$\\
	\toprule
	\endfirsthead
	\caption[]{\emph{continued}}\\
	\hline
	N\'{e}ron Type &$c_{p}$ & $p$ & $v_{p}\!\left(  c_{4}\right)  $ & $v_{p}\!\left(
c_{6}\right)  $ & $v_{p}\!\left(  \Delta\right)  $ & Additional Condition &
$f_{p}$\\
	\hline
	\endhead
	\hline
	\multicolumn{2}{r}{\emph{continued on next page}}
	\endfoot
	\hline
	\endlastfoot
II & $1$ & $3$ & $\geq2$ & $3$ & $3$ & $ \left(  \frac{c_{6}}{27}\right)
^{2}+2\not \equiv \frac{c_{4}}{3}\ \operatorname{mod}9$ & $3$\\  \cmidrule(lr){4-8}
& &  & $2$ & $4$ & $3$ &  & $3$\\  \cmidrule(lr){4-8}
& &  & $2$ & $3$ & $4$ &  & $4$\\  \cmidrule(lr){4-8}
& &  & $\geq3$ & $4$ & $5$ &  & $5$\\  \cmidrule(lr){3-8}
& & $\geq5$ & $\geq1$ & $1$ & $2$ &  & $2$\\  \cmidrule(lr){1-8}
$\mathrm{III}$ & $2$ & $3$ & $\geq2$ & $3$ & $3$ & $\left(  \frac{c_{6}}{27}\right)
^{2}+2\equiv\frac{c_{4}}{3}\ \operatorname{mod}9$ & $2$\\  \cmidrule(lr){4-8}
& &  & $2$ & $\geq5$ & $3$ &  & $2$\\  \cmidrule(lr){4-8}
& & $\geq5$ & $1$ & $\geq2$ & $3$ &  & $2$\\  \cmidrule(lr){1-8}
$\mathrm{IV}$ & $1$ or $3$ & $3$ & $2$ & $3$ & $5$ &  & $3$\\  \cmidrule(lr){4-8}
&  & & $3$ & $5$ & $6$ &  & $4$\\   \cmidrule(lr){4-8}
&  & & $\geq4$ & $5$ & $7$ &  & $5$\\  \cmidrule(lr){3-8}
& & $\geq5$ & $\geq2$ & $2$ & $4$ &  & $2$\\  \cmidrule(lr){1-8}
$\mathrm{I}_{0}^{\ast}$ & $1,2,$ & $3$ & $2$ & $3$ & $6$ &  & $2$\\  \cmidrule(lr){4-8}
& or $4$ &  & $3$ & $\geq6$ & $6$ &  & $2$\\ \cmidrule(lr){3-8}
& & $\geq5$ & $2$ & $\geq3$ & $6$ &  & $2$\\ \cmidrule(lr){4-8}
& &  & $\geq2$ & $3$ & $6$ &  & $2$\\  \cmidrule(lr){1-8}
I$_{n(n>0)}^{\ast}$ & $2$ or $4$ & $p\geq3$ & $2$ & $3$ & $6+n$ &  & $2$\\  \cmidrule(lr){1-8}
$\mathrm{IV}^{\ast}$ & $1$ or $3$ & $3$ & $\geq4$ & $6$ & $9$ & $\left(  \frac{c_{6}}{3^{6}}\right)
^{2}+2\not \equiv \frac{c_{4}}{27}\ \operatorname{mod}9$ & $3$\\ \cmidrule(lr){4-8}
&  & & $4$ & $7$ & $9$ &  & $3$\\ \cmidrule(lr){4-8}
&  & & $4$ & $6$ & $10$ &  & $4$\\ \cmidrule(lr){4-8}
&  & & $\geq5$ & $7$ & $11$ &  & $5$\\ \cmidrule(lr){3-8}
& & $\geq5$ & $\geq3$ & $4$ & $8$ &  & $2$\\  \cmidrule(lr){1-8}
$\mathrm{II}^{\ast}$ & $2$& $3$ & $\geq4$ & $6$ & $9$ & $\left(  \frac{c_{6}}{3^{6}%
}\right)  ^{2}+2\equiv\frac{c_{4}}{27}\ \operatorname{mod}9$ & $2$\\ \cmidrule(lr){4-8}
&  & & $4$ & $\geq8$ & $9$ &  & $2$\\ \cmidrule(lr){3-8}
& & $\geq5$ & $3$ & $\geq5$ & $9$ &  & $2$\\  \cmidrule(lr){1-8}
$\mathrm{II}^{\ast}$ & $1$ & $3$ & $4$ & $6$ & $11$ &  & $3$\\ \cmidrule(lr){4-8}
&  & & $5$ & $8$ & $12$ &  & $4$\\ \cmidrule(lr){4-8}
&  & & $\geq6$ & $8$ & $13$ &  & $5$\\ \cmidrule(lr){3-8}
& & $\geq5$ & $\geq4$ & $5$ & $10$ &  & $2$%
\label{ta:PapTableIandII}	
\end{longtable}}
\end{appendix}
\newpage
\bibliographystyle{amsalpha}
\bibliography{Local_Data}
\end{document}